\documentclass[12pt]{amsart}       
\usepackage{amsmath,amssymb,bbm,ifthen,url,wasysym,stmaryrd,xfrac,enumerate,etoolbox,graphicx,
  xr,xypic,pict2e}
\usepackage[arrow, matrix, curve]{xy}
\usepackage{bm}
\newcommand{\fourcitation}[1]{$\llbracket$#1$\rrbracket$}
\newboolean{fourpaper}
\setboolean{fourpaper}{true}
\ifthenelse{\boolean{fourpaper}}{%
  \externaldocument[FOUR::]{Fourier-v3}%
  \newcommand{\fourcite}[1]{\fourcitation{\ref{FOUR::#1}}}%
}{%
  \newcommand{\fourcite}[1]{\ref{FOUR::#1}}}

\usepackage[colorlinks,final,backref=page,hyperindex]{hyperref}

\newcommand{\pow}{\wp}
\newcommand{\fpow}{\pow^{(\omega)}}

\newcommand{\NN}{\mathbb{N}}
\newcommand{\ZZ}{\mathbb{Z}}
\newcommand{\QQ}{\mathbb{Q}}
\newcommand{\RR}{\mathbb{R}}
\newcommand{\CC}{\mathbb{C}}
\newcommand{\II}{\mathbb{I}}
\newcommand{\nnz}[1]{#1^\times}

\newcommand{\Hom}{\mathrm{Hom}}
\newcommand{\End}{\mathrm{End}}
\newcommand{\Aut}{\mathrm{Aut}}

\newcommand{\Ab}{\mathbf{Ab}}

\newcommand{\Nil}{\mathbf{Nil}}
\newcommand{\Cnt}{\mathbf{Cnt}}
\newcommand{\SES}{\mathbf{SES}}
\newcommand{\Mod}{\mathbf{Mod}}

\newcommand{\Bil}{\mathbf{Bil}}
\newcommand{\GenBil}{\mathbf{Bil}^{\mathord{\times}\mspace{-3.5mu}}}
\newcommand{\PropBil}{\mathbf{Bil}^{\mathord{\dagger}\mspace{-3.5mu}}}
\newcommand{\Hei}{\mathbf{Hei}}
\newcommand{\GenHei}{\mathbf{Hei}^{\mathord{\times}\mspace{-3.5mu}}}

\newcommand{\CAlg}{\mathbf{C\kern-0.1ex{}Alg}}

\newcommand{\Bi}{\mathbf{Bi}}
\newcommand{\Du}{\mathbf{Du}}

\newcommand{\Four}{\mathcal{F}}

\newcommand{\Genarr}[1]{\smash{\raisebox{-1.5pt}{$\overset{#1}{\kern2\mu\longrightarrow\kern2\mu}$}}}
\newcommand{\Genbiarr}[1]{\smash{\raisebox{-1.5pt}{$\overset{#1}{\kern2\mu\longleftrightarrow\kern2\mu}$}}}

\newcommand{\Par}{\mathcal{P}}
\newcommand{\Grp}{\mathbf{Grp}}
\newcommand{\ModH}[1]{\ifthenelse{\equal{#1}{}}{\mathbf{ModH}}{\mathbf{ModH}(#1)}}
\newcommand{\AlgH}[1]{\ifthenelse{\equal{#1}{}}{\mathbf{AlgH}}{\mathbf{AlgH}(#1)}}

\newcommand{\inner}[2]{\langle#1|#2\rangle}

\newcommand{\funcinner}[2]{(#1|#2)}

\newcommand{\cum}{{\textstyle\varint}}

\newcommand{\im}{\operatorname{im}}

\newcommand{\tr}{\operatorname{tr}}
\newcommand{\GF}{\operatorname{GF}}

\newcommand{\dotarrow}{\mathrel{\vbox{\offinterlineskip\ialign{\hfil##\hfil\cr\scalebox{1.2}{%
        \normalfont.}\cr\noalign{\kern-.1ex}$\rightarrow$\cr}}}}
\newcommand{\der}{\partial}
\newcommand{\Cat}{\mathbf{Cat}}

\newcommand{\Tor}{\mathbb{T}}
\newcommand{\pont}{\varpi}

\newcommand{\trp}{^{\top\!}}

\newcommand{\primeswithunit}{\mathbb{P}_\ast}

\newcommand{\Sym}{\operatorname{Sym}}
\newcommand{\zentrum}{\mathcal{Z}}
\newcommand{\Heis}{\mathbf{Heis}}
\newcommand{\HeiGrp}{\mathfrak{H}}
\newcommand{\BilFrm}{\mathfrak{B}}

\newcommand{\medmat}[9]{\left(\begin{smallmatrix}#1&#2&#3\\#4&#5&#6\\#7&#8&#9\end{smallmatrix}\right)}
\newcommand{\Ext}{\operatorname{Ext}}

\newcommand{\adm}{\mathcal{A}}
\newcommand{\bspc}{\mathcal{B}}
\newcommand{\ev}{\mathrm{ev}}
\newcommand{\exptor}[1]{#1_{\sharp}}
\newcommand{\Texp}{\exptor{T}}
\newcommand{\Tpexp}{\exptor{T'}}

\newcommand{\lr}{\longrightarrow}
\newcommand{\la}{\langle}
\newcommand{\ra}{\rangle}
\newcommand{\dual}[2]{\la#1\vert#2\ra}
\newcommand{\Imp}{\Longrightarrow}
\newcommand{\briota}{\buildrel\iota\over\longrightarrow}
\newcommand{\bri}{\buildrel i\over\longrightarrow}
\newcommand{\brpi}{\buildrel\pi\over\longrightarrow}
\newcommand{\brj}{\buildrel j\over\longrightarrow}
\newcommand{\brp}{\buildrel p\over\longrightarrow}
\newcommand{\brq}{\buildrel q\over\longrightarrow}

\newcommand{\brbullet}{\buildrel\bullet\over\longrightarrow}
\newcommand{\Z}{\mathbb Z}

\DeclareFontFamily{U}{mathb}{\hyphenchar\font45}
\DeclareFontShape{U}{mathb}{m}{n}{
<-6> mathb5 <6-7> mathb6 <7-8> mathb7
<8-9> mathb8 <9-10> mathb9
<10-12> mathb10 <12-> mathb12
}{}
\DeclareSymbolFont{mathb}{U}{mathb}{m}{n}
\DeclareMathSymbol{\Prec}{\mathrel}{mathb}{"CE}
\DeclareMathSymbol{\Succ}{\mathrel}{mathb}{"CF}

\makeatletter
\DeclareRobustCommand{\pto}{\mathrel{\mathpalette\p@to\to}}
\DeclareRobustCommand{\ppto}{\mathrel{\mathpalette\pp@to\to}}
\DeclareRobustCommand{\pppto}{\mathrel{\mathpalette\ppp@to\to}}
\DeclareRobustCommand{\pgets}{\mathrel{\mathpalette\p@gets\gets}}
\newcommand{\p@to}[2]{%
  \ooalign{\hidewidth$\m@th#1\mapstochar\mkern2mu$\hidewidth\cr$\m@th#1\to$\cr}%
}
\newcommand{\pp@to}[2]{%
  \ooalign{\hidewidth$\m@th#1\mkern-2mu\mapstochar\,\mapstochar\mkern2mu$\hidewidth\cr$\m@th#1\to$\cr}%
}
\newcommand{\ppp@to}[2]{%
  \ooalign{\hidewidth$\m@th#1\mkern-4mu\mapstochar\,\mapstochar\,\mapstochar\mkern2mu$\hidewidth\cr$\m@th#1\to$\cr}%
}
\newcommand{\p@gets}[2]{%
  \ooalign{\hidewidth$\m@th#1\mapstochar\mkern5mu$\hidewidth\cr$\m@th#1\gets$\cr}%
}
\makeatother

\newcommand{\isomarrow}{\xrightarrow{
   \,\smash{\raisebox{-0.65ex}{\ensuremath{\scriptstyle\sim}}}\,}}

\def\<#1.#2{{}_{#1\!}#2}
\def\>#1.#2{#1_{#2}}
\def\[#1.#2{{}_{[#1]}#2}
\def\]#1.#2{#1_{[#2]}}

\makeatletter
\newcommand{\oset}[3][0ex]{%
  \mathrel{\mathop{#3}\limits^{
    \vbox to#1{\kern-2\ex@
    \hbox{$\scriptstyle#2$}\vss}}}}
\makeatother

\makeatletter
\newcommand{\adjunction}[4]{%
  #1\colon #2%
  \mathrel{\vcenter{%
    \offinterlineskip\m@th
    \ialign{%
      \hfil$##$\hfil\cr
      \longrightharpoonup\cr
      \noalign{\kern-.3ex}
      \smallbot\cr
      \longleftharpoondown\cr
    }%
}}%
  #3 \noloc #4%
}
\newcommand{\longrightharpoonup}{\relbar\joinrel\rightharpoonup}
\newcommand{\longleftharpoondown}{\leftharpoondown\joinrel\relbar}
\newcommand\noloc{%
  \nobreak
  \mspace{6mu plus 1mu}
  {:}
  \nonscript\mkern-\thinmuskip
  \mathpunct{}
  \mspace{2mu}
}
\newcommand{\smallbot}{%
  \begingroup\setlength\unitlength{.15em}%
  \begin{picture}(1,1)
  \roundcap
  \polyline(0,0)(1,0)
  \polyline(0.5,0)(0.5,1)
  \end{picture}%
  \endgroup
}
\makeatother

\let\phi\varphi
\let\epsilon\varepsilon
\let\mathbb\mathbbm

\newtheorem{theorem}{Theorem}
\newtheorem{proposition}[theorem]{Proposition}
\newtheorem{lemma}[theorem]{Lemma}
\newtheorem{corollary}[theorem]{Corollary}

\newtheorem{fact}[theorem]{Fact}

\theoremstyle{definition}
\newtheorem{definition}[theorem]{Definition}
\newtheorem{example}[theorem]{Example}
\newtheorem{remark}[theorem]{Remark}

\newcommand{\myexend}{\hfill$/\!/$}
\newenvironment{myexample}{\begin{example}}{\myexend\end{example}}
\newenvironment{myremark}{\begin{remark}}{\hfill$\circledcirc$\end{remark}}

\title{Heisenberg Groups via Algebra}

\author{G{\"u}nter Landsmann \and Markus Rosenkranz}
\address{%
  RISC,
  Johannes Kepler University, A-4040 Linz, Austria}
\email{marcus@rosenkranz.or.at,landsmann@risc.uni-linz.ac.at}

\date{\today}

\begin{document}
\maketitle

\begin{abstract}
  We introduce a general class of Heisenberg groups motivated by applications of algebraic Fourier
  theory. Basic properties are examined from a homological perspective.
\end{abstract}

\tableofcontents

\section{Introduction}

\subsection{Terminological conventions and notation.} We write~$\pow(S)$ for the power set of a
set~$S$ and~$\fpow(S)$ for its finite version (the set of all finite subsets of~$S$). Locally
compact includes Hausdorff. Define bicharacter (left side additive, right side multiplicative!) and
nondegenerate bicharacter. Column vectors $K^n$, row vectors $K_n$, the $n \times n$ identity matrix
$I_n$. All vector space duals in the algebraic sense. Explain action as scalars versus
operators. Category of groups is denoted by~$\Grp$, abelian ones by~$\Ab$. Remind of twisted
modules~$M[\sigma]$ for arbitrary modules~$M \in \>\Mod.R$ and homomorphism~$\sigma\colon R' \to R$,
to be defined in detail. An involution is an automorphism (of groups or of $K$-algebras) that is
inverse to itself. Algebras generally assumed commutative but nonunital. We shall frequently employ
the abbreviation~$x^- := x^{-1}$ for the inverse of an invertible element in a (multiplicatively
written) monoid. Use~$\chi_S$ for the characteristic function of a set~$S \subseteq \RR^n$, meaning
$+1$ within~$S$. Unit interval~$\II = [0,1]$.  Explain
notation~$\mathbb{P} = \{ 2, 3, 5, 7, 11, \dots \}$ for the prime numbers
and~$\primeswithunit = \{ 1, 2, 3, 5, 7, 11, \dots \}$ for its extension by unity. Also introduce
notation~$\NN_{>0} = \{ 1, 2, 3, \dots \}$ for the positive natural numbers. Write~$\RR_{\ge 0}$ for
the nonnegative reals, consequently~$\RR_{>0}$ for the positive reals. The symmetrc algebra over an
$R$-module~$M$ is denoted by~$\Sym(M) \cong T(M)/I(M)$, where~$I(M) \trianglelefteq T(M)$ is the
ideal generated by~$\{x \otimes y - y \otimes x \mid x,y \in M\}$. In case of ambiguity, the base
ring~$R$ may be explicated in writing~$\Sym_R(M) \cong T_R(M)/I_R(M)$. Note that for a set~$X$ one
has~$R[X] = \Sym_R(RX)$, where~$RX$ is the free $R$-module over the basis~$X$. The center of a
group~$G$ is denoted by~$\zentrum G$. The dual of a poset~$P$ is denoted by~$\hat{P}$.

Given a $R$-module~$M$, an $n$-form is a bilinear map~$M^n \to R$; if~$n$ is suppressed, we
take~$n=2$. An \emph{alternating form} is a bilinear map~$\omega\colon M \oplus M \to R$ such
that~$\omega(x,x) = 0$ for all~$x \in M$; it is called a \emph{symplectic form} if it is moreoever
\emph{nondegenerate} in the sense that~$\omega(x,-)\colon M \to M$ is injective for all~$x \in M$.
We refer to the structure~$(M, \omega)$ as an alternating or symplectic module, respectively. In
this paper we will always deal with the case~$R = \ZZ$ so that~$M$ is an abelian group. In this
case, a form is called a \emph{bicharacter}, a nondegenerate one a \emph{duality}, and an
alternating duality induces the structure of a \emph{symplectic $\ZZ$-module}. (This slightly
awkward expression is necessary since the term \emph{symplectic group} is already reserved for the
linear transformations of a symplectic module that leave its symplectic form invariant.)

We will use the letter sequence `SES' to stand for `short exact sequence'. Sometimes, when clear from context we write $\widetilde{X}$ for the inverse image $\pi^{-1}(X)$.

\subsection{Remark.} The material of this paper is to some extent coupled with that of its
\emph{companion paper}~\cite{RosenkranzLandsmann2020}, where the focus is on the structures actually
arising in constructive analysis. Whenever we refer to specific places in the companion paper, we
will use the shorthand~\fourcitation{\dots} for~\cite[\dots]{RosenkranzLandsmann2020}.

\section{Heisenberg Groups in Algebra}

\subsection{Nilquadratic Groups and Symplectic Forms}\label{sub:nilquadratic-groups}

From a purely algebraic perspective, Heisenberg groups may be viewed as the simplest nonabelian
groups with \emph{bipartite internal symmetry}. We shall make this vague characterization more
precise in the next subsection.

Let us first recall some basic facts about nilpotent groups. A central series for a group~$H$ may be
defined as a normal series~$(H_k)$ having central factors. More precisely, an \emph{increasing
  central series}~\cite[\S5.1]{Robinson2012} has~$H_0 = 1$ and~$H_{k+1}/H_k \le \zentrum(H/H_k)$,
whereas a \emph{decreasing central series} is characterized~\cite[(2.4.5)]{Cohn2003}
\cite[Ex.~5.39]{Rotman1995} by $H_0 = H$ and~$H_{k-1}/H_k \le \zentrum(H/H_k)$. The center
conditions are, respectively, equivalent to~$[H, H_{k+1}] \le H_k$ and~$[H, H_{k-1}] \le H_k$. It is
known that: (i) If some increasing central series terminates with~$H$, then all such series do
so. (ii) If some decreasing central series terminates with~$1$, then all such series do so. (iii)
These two conditions are equivalent and may be taken as the definition of a nilpotent group~$H$,
where increasing and decreasing central series may be turned into each other by reversion. The
\emph{nilpotency class} of~$H$ is then defined as the minimal length of any central series.

Defining the higher centers~\cite[p.~113]{Rotman1995} by~$Z_{k+1}/Z_k := \zentrum(H/Z_k)$ one
obtains a canonical increasing central series with~$Z_1 = \zentrum(G)$; it is called the \emph{upper
  central series} since it majorizes all (increasing) central
series~\cite[Prop.~5.1.9]{Robinson2012}. In a similar fashion, the lower commutator
groups~$A_{k+1} := [H, H_k]$ yield a decreasing central series with~$A_1 = [H, H]$; it is called the
\emph{lower central series} since in minorizes all (decreasing) central series. For a nilpotent
group~$H$, the length of both~$(A_k)$ and~$(Z_k)$ is the nilpotency class of~$H$.

From the viewpoint of group extensions, one may characterize nilpotent groups as those obtained from
the trivial group via \emph{central extensions}. Recall that a group
extension~$1 \to T \to H' \to H \to 1$ is \emph{central} if the image of~$T$ is contained
in~$\zentrum(H)$; we call it~\emph{strictly central} if the image is exactly~$\zentrum(H)$. Now we
declare~$H = 1$ to be nilpotent of class~$0$, and we stipulate that~$H'$ is nilpotent of class~$r+1$
if there exists a central extension~$1 \to T \to H' \to H \to 1$ with~$H$ nilpotent of
class~$r$. This characterization allows one to extract a central series from a given group~$H$ of
nilpotency class $r$ by unfolding the successive central extensions:
\begin{equation}
  \label{eq:centextns}
  \left\{\vcenter{\vbox{
        \xymatrix @M=0.5pc @R=1pc @C=2pc%
        {%
          1 \ar[r] & H_1 \ar[r] & H \ar[r] & H/H_1 \ar[r] & 1,\\
          1 \ar[r] & H_2/H_1 \ar[r] & H/H_1 \ar[r] & H/H_2 \ar[r] & 1,\\
          1 \ar[r] & H_3/H_2 \ar[r] & H/H_2 \ar[r] & H/H_3 \ar[r] & 1,\\
          & & \vdots\\
          1 \ar[r] & H/H_{r-1} \ar[r] & H/H_{r-1} \ar[r] & 1 \ar[r] & 1,    
        }}}
  \right.
\end{equation}
so
that~$1 = H_0 \triangleleft H_1 \triangleleft H_2 \triangleleft \cdots \triangleleft H_{r-1}
\triangleleft H_r = H$
is an increasing central series for~$H$. Of course, one may as well proceed obveresly, resulting in
a decreasing central series.

Let us describe these relations in category-theoretic terms. Our starting point is~$\Nil_r$, the
collection of nilpotent groups of class at most~$r$, viewed as a full subcategory of $\Grp$. Central
series of length~$r$ without repetition are taken as the objects of another category~$\Cnt_r$. While
it does not matter, we may fix increasing central series for definiteness. We define a
morphism~$(\phi_0, \phi_1, \dots, \phi_r)$ by the obvious commuting diagram
\def\nsg{\ar@{}[r]|{\triangleleft}}
\begin{equation*}
  \xymatrix @M=1pc @R=1.5pc @C=0.5pc%
  { 
    1 = H_0 \ar@<2ex>[d]^{\phi_0} \nsg & H_1 \ar[d]^{\phi_1} \nsg & \dots \nsg 
    & H_{r-1} \ar[d]^{\phi_{r-1}} \nsg & H_r = H, \ar@<-3ex>[d]^{\phi_{r}}\\
    1 = H_0' \nsg & H_1' \nsg & \dots \nsg & H_{r-1}' \nsg & H_r' = H,
}    
\end{equation*}
where~$\phi_0\colon H_0 \to H_0', \dots, \phi_r\colon H_r \to H_r'$ are group
homomorphisms. Clearly, $\phi_0 = 1$ and all~$\phi_j \: (j < r)$ are determined by~$\phi := \phi_r$.
There is an obvious functor~$\Cnt_r \to \Nil_r$ that maps any central series of a nilpotent
group~$H$ to the bare group~$H$ and the morphism $(\phi_0, \phi_1, \dots, \phi_r)$
to~$\phi = \phi_r$. This functor is clearly dense (in fact, surjective on objects) and faithful
(since~$\phi$ is uniquely determined) and full (as one sees by taking for example upper central
series), hence it constitutes an \emph{equivalence of categories}. Note
that~$(\phi_0, \phi_1, \dots, \phi_r)$ also determines unique morphisms (in the category~$\SES$ of
short exact sequences) from each central extension in~\eqref{eq:centextns} to the corresponding
central extensions for~$H'$.

Note that~$\Nil_0 = \{ 1 \}$ and $\Nil_1 = \Ab$; we shall here be interested in the simplest
noncommutative case---the \emph{nilquadratic groups}~$\Nil_2$. In this case, there is only one row
in~\eqref{eq:centextns} for a given~$H \in \Nil_2$, we shall write the corresponding central
extension as~$1 \to T \to H \to P \to 0$. For reasons that will soon become clear, we refer to the
abelian group~$T$ as a \emph{torus} and to the abelian group~$P$ as a \emph{phase space} of the
nilquadratic group~$H$. Note that~$P$ is written additively, whereas~$T$ and of course~$H$ are
written multiplicatively.

\begin{lemma}\label{L1} Let $G\in\Grp$, and $k_G\colon G\times G\lr G$, $(x,y)\mapsto[x,y]$. Then
$k_G$ is left-linear $\iff$ $[G,G]\subseteq Z(G)$ $\iff$ $k_G$ is right-linear.
\end{lemma}
\begin{proof}
If $[G,G]\subseteq Z(G)$ then
\begin{eqnarray*}
[x,y][x,z]&=&xyx^{-1}y^{-1}[x,z]=xyx^{-1}[x,z]y^{-1}=xyzx^{-1}z^{-1}y^{-1}=[x,yz]\cr
[x,z][y,z]&=&xzx^{-1}z^{-1}[y,z]=x[y,z]zx^{-1}z^{-1}=xyzy^{-1}x^{-1}z^{-1}=[xy,z]
\end{eqnarray*}
i.e., $k_G$ is bilinear.

\vspace{2mm}
Assume that $k_G$ is left-linear. Take $x,y,g\in G$, and set $z:=y^{-1}g^{-1}y$. Then
\begin{eqnarray*}
z[x,y]g&=&zxyx^{-1}y^{-1}\overbrace{(yzy^{-1})^{-1}}^g=zxyx^{-1}y^{-1}yz^{-1}y^{-1}=zxyx^{-1}z^{-1}y^{-1}\cr
&=&[zx,y]=[z,x][z,y]=z\underbrace{yz^{-1}y^{-1}}_gxyx^{-1}y^{-1}=zg[x,y]
\end{eqnarray*}
$\Rightarrow$ $[x,y]g=g[x,y]$. Thus $[G,G]\subseteq Z(G)$.

\vspace{2mm}
Now let $[\bullet,\bullet]$ be right-linear, $x,y,g\in G$. Then $[x,g^{-1}y]=[x,g^{-1}][x,y]$, and therefore
\begin{eqnarray*}
xg^{-1}yx^{-1}y^{-1}g&=&xg^{-1}x^{-1}g[xy]\cr
yx^{-1}y^{-1}g&=&x^{-1}g[xy]\cr
xyx^{-1}y^{-1}g&=&g[xy]\cr
[x,y]g&=&g[x,y]
\end{eqnarray*}
Again, $[G,G]\subseteq Z(G)$.
\end{proof}
\begin{definition} We say that a group $G$ {\tt has bilinear commutator} in case that the map $k_G\colon G\times G\to G$, $(x,y)\mapsto[x,y]$ is bilinear.
\end{definition}

\begin{fact}
  \label{fct:char-cent-ext}
  Let~$H$ be a group. Then~$H$ is a central extension of an abelian group iff~$H/\zentrum(H)$ is
  abelian iff~$[H,H] \le \zentrum(H)$ iff~$H \in \Nil_2$ iff~$H$ has bilinear commutator.
\end{fact}
\begin{proof}
  Assume~$1 \to T \hookrightarrow H \to P \to 0$ is a central extension of the abelian group~$P$,
  assuming~$T \subseteq H$ for simplicity. This yields abelian subgroups
  $\zentrum(H)/T \trianglelefteq H/T \cong P$ whose
  quotient~$\tfrac{H/T}{\zentrum(H)/T} \cong H/\zentrum(H)$ is abelian as well. Conversely, if we
  assume~$P := H/\zentrum(H)$ abelian, we have $1 \to \zentrum(H) \to H \to H/\zentrum(H) \to 0$ is
  as a central extension of the abelian group~$P$. This takes care of the first equivalence; the
  second is immediate from the definition of~$[H,H]$.

  For showing the next equivalence, assume again that~$H/\zentrum(H)$ is abelian. Then the upper
  central series~$Z_0 = 1, Z_1 = \zentrum(H)$ ends with
  $Z_2 = \pi^{-1} \, \zentrum\big(H/\zentrum(H)\big) = H$, where~$\pi\colon H \to H/\zentrum(H)$ is
  the canonical projection. Hence~$H$ is indeed nilpotent of class as most~$2$. Finally, assume
  now~$H \in \Nil_2$, and take an arbitrary central series
  $1 = H_0 \trianglelefteq H_1 \trianglelefteq H_2 = H$. Since the upper central series majorizes
  any other central series~\cite[Ex.~5.39]{Rotman1995}, we
  get~$H_2 \le Z_2 = \pi^{-1} \, \zentrum\big(H/\zentrum(H)\big)$
  or~$\pi(H) \le \zentrum\big( H/\zentrum(H) \big)$. But this implies that~$H/\zentrum(H)$ is
  abelian. The last point is an immediate consequence of Lemma \ref{L1}.
\end{proof}

\begin{lemma}\
\label{L4}
\begin{enumerate}
\item $G,H\in\Grp$. Then $G,H\in\Nil_2\iff G\times H\in\Nil_2$.
\item $1\lr N\bri E\brpi G\lr 1$ SES in $\Grp$. Then\\
$E\in\Nil_2\Imp N\in\Nil_2$ and $G\in\Nil_2$.
\end{enumerate}
\end{lemma}
\begin{proof}
If $G,H\in{\rm Nil}_2$ then $G'\subseteq Z(G)\land H'\subseteq Z(H)$. Therefore
\begin{equation*}
\left[\left(\begin{matrix}g_1\cr h_1\end{matrix}\right),\left(\begin{matrix}g_2\cr h_2\end{matrix}\right)\right]=\left(\begin{matrix}g_1g_2g_1^{-1}g_2^{-1}\cr h_1h_2h_1^{-1}h_2^{-1}\end{matrix}\right)=\left(\begin{matrix}[g_1,g_2]\cr[h_1,h_2]\end{matrix}\right)\in Z(G)\times Z(H)=Z(G\times H)
\end{equation*}
that is, $(G\times H)'\subseteq Z(G\times H)$.

\vspace{2mm}
Conversely, assume that $G\times H\in{\rm Nil}_2$. Take $[g_1,g_2]\in G'$. Then
\begin{equation*}
\left(\begin{matrix}[g_1,g_2]\cr[1,1]\end{matrix}\right)=\left[\left(\begin{matrix}g_1\cr 1\end{matrix}\right),\left(\begin{matrix}g_2\cr 1\end{matrix}\right)\right]\in Z(G\times H)=Z(G)\times Z(H)
\end{equation*}
and therefore $[g_1,g_2]\in Z(G)$. Similarly $H\in{\rm Nil}_2$.

\vspace{2mm}
Consider the SES $1\lr N\bri E\brpi G\lr 1$ in $\Grp$ with $E\in\Nil_2$. Then $E'\subseteq Z(E)$, and so
\begin{equation*}
G'=\pi(E)'=\pi(E')\subseteq\pi(Z(E))\subseteq Z(\pi(E))=Z(G)
\end{equation*}
that means, $G\in\Nil_2$. For $r,s,r_1\in N$
\begin{equation*}
i([r,s]\cdot r_1)=[i(r),i(s)]i(r_1)=i(r_1)[i(r),i(s)]=i(r_1\cdot[r,s])
\end{equation*}
and so $[r,s]r_1=r_1[r,s]$ $\forall r,s,r_1\in N$. Therefore $N'\subseteq Z(N)$, that is, $N\in\Nil_2$.

\end{proof}

It is now easy to see that~$\Nil_2$ is also equivalent to the category~$\SES_2$ of \emph{central
  extensions of abelian groups}; this is a full subcategory of~$\SES$. The \emph{extraction
  functor}~$U\colon \SES_2 \to \Nil_2$ maps a short exact sequence $1 \to T \to H \to P \to 0$
to~$H$ and a morphism such as
\begin{equation}
  \label{eq:ses-morphism}
  \xymatrix @M=0.5pc @R=1pc @C=2pc%
  { 1 \ar[r] & T \ar[r] \ar[d]^t & H \ar[r] \ar[d]^h & P \ar[r] \ar[d]^p & 0\\
    1 \ar[r] & T' \ar[r] & H' \ar[r] & P' \ar[r] & 0 }
\end{equation}
to the group homomorphism~$h\colon H \to H'$.

\begin{proposition}
  \label{eq:ses2-nil2}
  We have an equivalence~$U\colon \SES_2 \isomarrow \Nil_2$.
\end{proposition}
\begin{proof}
  The functor~$U$ is full and dense since any nilquadratic group~$H$ yields a canonical central
  extension~$1 \to \zentrum(H) \hookrightarrow H \to H/\zentrum(H) \to 0$, and any group
  homomorphisms~$H \to H'$ is obtained from a morphism $(t,h,p)$ between such canonical extensions,
  where~$t\colon \zentrum(H) \to \zentrum(H')$ is the restriction of~$h$,
  and~$p\colon H/\zentrum(H) \to H'/\zentrum(H')$ the induced projection. To see that~$U$ is
  faithful, let~$T \oset{\iota}{\rightarrowtail} H \oset{\pi}{\twoheadrightarrow} P$
  and~$T' \oset{\iota'}{\rightarrowtail} H' \oset{\pi'}{\twoheadrightarrow} P'$ be central
  extensions in~$\SES_2$ with two morphisms~$(t_1, h, p_1)$ and~$(t_2, h, p_2)$ between them. We
  have then~$\iota t_1(c) = h\iota(c) = \iota t_2(c)$ for any~$c \in T$; since~$\iota$ is injective,
  this yields~$t_1 = t_2$. Similarly, $p_1 = p_2$ follows from $p_1 \pi(u) = \pi' h(u) = p_2 \pi(u)$
  for~$u \in h$ because~$\pi$ is surjective.
\end{proof}

Note that the equivalences~$\SES_2 \isomarrow \Cnt_2 \isomarrow \Nil_2$ act as surjections on
objects---they provide progressively less information about a nilquadratic group~$H$. First one
discards knowledge about \emph{how} a specific torus is embedded in~$H$, and then also \emph{which}
torus is embedded. Since~$1 \triangleleft T \triangleleft H$ is an increasing central series, the
lower and upper central series impose \emph{bounds on the choice of the torus}: One must
have~$[H,H] \le T \le \zentrum(H)$. Since the converse follows from Fact~\ref{fct:char-cent-ext}, we
obtain the following characterization of central extensions in terms of their tori.

\begin{fact}
  \label{fct:char-cent-exseq}
  An exact sequence~$T \oset{\iota}{\rightarrowtail} H \oset{\pi}{\twoheadrightarrow} P$ describes a
  central extension of the abelian group~$P$ iff $[H, H] \le \iota(T) \le \zentrum(H)$.
\end{fact}

According to Fact~\ref{fct:char-cent-ext}, there are two extreme cases of a central extension: The
one corresponding to the lower central series~$H \triangleright [H, H] \triangleright 0$ is given
by~$[H, H] \oset{\iota}{\rightarrowtail} H \oset{\pi}{\twoheadrightarrow} H^{\mathrm{ab}}$, and it
has the abelianization for its phase space. On the other hand, the extension associated to the upper
central series~$1 \triangleleft \zentrum(H) \triangleleft H$ is given
by~$\zentrum(H) \oset{\iota}{\rightarrowtail} H \oset{\pi}{\twoheadrightarrow} \mathrm{Inn}(H)$,
where the phase space may be taken as the group of inner automorphisms. We call an extension
\emph{strictly central} if it is of the latter type, meaning its torus is~$T = \zentrum(H)$.

\begin{myexample}
  \label{ex:non-central-heis}
  This can be illustrated by a minor variation of the \emph{classical Heisenberg group} (to be
  developed later in Example~\ref{ex:classical-vector-group}). Let us endow the set
  $H := \nnz{\CC} \times \RR \times \RR$ with the group law
  \begin{equation}
    \label{ex:class-heis-mod}
    c(x,\xi) \cdot c'(x',\xi') = cc' e^{i\tau \xi x'} \, (x+x', \xi+\xi'),
  \end{equation}
  introducing the \emph{center convention} of
  writing~$(c,x,\xi) \in \nnz{\CC} \times \RR \times \RR$ in the form~$c(x,\xi)$ with center
  elements exposed. It is easy to see that~$\zentrum(H) = \nnz{\CC}$ and~$[H,H] = \Tor$, via the
  embedding~$\Tor \subset \nnz{\CC} \hookrightarrow H$ operating as~$c(x,\xi) \mapsto (c,0,0)$.
  Clearly, $\nnz{\CC}$ and~$\Tor$ are normal in~$H$, and one
  gets~$H/\zentrum(H) \cong \RR \oplus \RR$ via $c(x,\xi) \, \nnz{\CC} \leftrightarrow (x,\xi)$ and
  likewise also $H/[H,H] \cong \RR_{>0} \oplus (\RR \oplus \RR)$ with bijection
  $c(x,\xi) \, \Tor \leftrightarrow |c| \, (x,\xi)$. Therefore we obtain the $\SES_2$ morphism
  \begin{equation*}
    \xymatrix @M=0.75pc @R=1.25pc @C=1pc%
    { 1 \ar[r] & \Tor \ar[r] \ar@<-0.5ex>@{^{(}->}[d] & H \ar[r] \ar@{=}[d] 
      & \RR_{>0} \oplus (\RR \oplus \RR) \ar[r] \ar@{->>}[d] & 0,\\
      1 \ar[r] & \nnz{\CC} \ar[r] & H \ar[r] & \RR \oplus \RR \ar[r] & 0. }  
  \end{equation*}
  Here the \emph{top} sequence corresponds to the \emph{lower central
    series}~$1 \triangleleft \Tor \triangleleft H$, the \emph{bottom} one to the \emph{upper central
    series}~$1 \triangleleft \nnz{\CC} \triangleleft H$.
\end{myexample}

In a morphism~\eqref{eq:ses-morphism} of central extension, the middle map~$h\colon H \to H'$
obviously determines the lateral maps~$t\colon T \to T'$ and~$p\colon P \to P'$
since~$\SES_2 \to \Nil_2$ is faithful as we have noted. The converse is not quite true: Having the
maps~$t$ and~$p$ does not determine~$h$; an additional map is needed for fixing~$h$. To state the
precise conditions, we use the language of \emph{group cohomology}~\cite{HiltonStammbach1971}. Let
us write~$C^n(P, T)$ for the chain group of all functions~$\phi\colon P^n \to T$ with the standard
differential~$d^n\colon C^n(P, T) \to C^{n+1}(P, T)$ given by
\begin{align*}
  & (d^n\phi)(z_1, \dots, z_{n+1}) := z_1 \cdot \phi(z_2, \dots, z_n)\\
  & \quad {} + \sum_{i=1}^n (-1)^i \, \phi(z_1, \dots, z_{i-1}, z_i z_{i+1}, \dots, z_{n+1}) +
    (-1)^{n+1} \, \phi(z_1, \dots, z_n),
\end{align*}
for arbitrary extension. Since we consider here only central extensions, the action is trivial and
the multiplication by~$z_1$ may be dropped.

Writing~$Z^n(P, T) := \ker(d^n)$ for the cocycles, $B^n(P, T) := \im(d^{n-1})$ for the coboundaries,
the corresponding \emph{cohomology groups} are given as usual by~$H^n(P, T) = Z^n(P, T)/B^n(P, T)$.
Group operations on cycles are written additively. For functions~$t\colon T \to \tilde{T}$
and~$p\colon \tilde{P} \to P$ and a cochain~$\phi \in C^n(P, T)$, we use the
notation~$t_*(\phi) := t \circ \phi \in C^n(P, \tilde{T})$
and~$p^*(\phi) := \phi \circ (p \times \cdots \times p) \in C^n(\tilde{P}, T)$. We take the liberty
of identifying the extension group~$H$ with~$T \times_\gamma P$, where~$\gamma \in Z^2(P, T)$ is a
suitable \emph{factor set} for the extension. We assume normalized factor sets throughout.

\begin{proposition}
  \label{prop:ses-morph}
  Let~$E\colon T \oset{\iota}{\rightarrowtail} H \oset{\pi}{\twoheadrightarrow} P$
  and~$E'\colon T' \oset{\iota'}{\rightarrowtail} H' \oset{\pi'}{\twoheadrightarrow} P'$ be central
  extensions with factor sets~$\gamma$ and~$\gamma'$, respectively. Then the Hom set~$\SES_2(E,E')$
  is in bijective correspondence with
  \begin{align*}
    \{ (t,p,\psi) \mid {} & t \in \Hom(T,T') \land p \in \Hom(P, P') \land \psi \in C^1(P, T')\\
                       & \qquad \land d^2(\psi) = t_* \gamma / p^* \gamma' \}
  \end{align*}
  such that~$(t,p,\psi)$ corresponds to~$(t,h,p)\colon E \to E'$ with the middle map given
  by~$h(c,z) = \big( \psi(z) \, t(c), p(z) \big)$.
\end{proposition}

\begin{proof}
  Assume first that~$(t,h,p)\colon E \to E'$ is a morphism in~$\SES_2$. Then we must
  have~$h(c,z) = \big( h_{T'}(c,z), h_{P'}(c,z) \big)$ for suitable maps $h_{T'}\colon H \to T'$
  and~$h_{P'}\colon H \to P'$. From~$\pi' h = p \pi$ we obtain immediately that~$h_{P'}(c,z) = p(z)$
  and from~$h \iota = \iota' t$ that~$h(c,0) = \big( t(c), 0\big)$. Then~$(c,z) = (c, 0) \, (1, z)$
  implies $h(c,z) = \big( t(c), 0 \big) \, \big( h_{T'}(1, z), p(z) \big)$, which is of the required
  form if we set~$\psi(z) := h_{T'}(1, z)$. Using the cocycles~$\gamma$ and~$\gamma'$, it is now
  easy to check that~$h(c,z) = \big( t(c) \, \psi(z), p(z) \big)$ is a homomorphism iff~$t$ and~$p$
  are homomorphisms and
  \begin{equation*}
    \frac{t \gamma(z, w)}{\gamma'\big(p(z), p(w)\big)} 
    = \frac{\psi(z) \, \psi(w)}{\psi(z+w)},
  \end{equation*}
  which is the same as~$d^2(\psi) = \tau_* \gamma / p^* \gamma'$.
\end{proof}

\begin{myexample}
  As a \emph{typical example} of Proposition~\ref{prop:ses-morph}, consider the special case
  $t = 1_T$ and~$p = 1_P$ on the group~$H = \nnz{\CC} \times \RR \times \RR$ of
  Example~\ref{ex:non-central-heis} with the $1$-chain~$\psi\colon \RR \times \RR \to \nnz{\CC}$
  given by~$\psi(x,\xi) := e^{i\tau x} e^{i\tau\xi}$. Since in this case~$\gamma = \gamma'$
  and~$t = 1_T$, the $1$-chain~$\psi$ must be a $1$-cycle, which means a crossed
  homomorphism. Having trivial action, crossed homomorphisms are in fact plain homomorphisms. As a
  consequence, the conditition~$d^1(\psi) = \tau_* \gamma - p^* \gamma'$ is just that~$\psi$ be a
  homomorphism, which is indeed the case.

  For a more crucial example of a $\SES_2$ morphism, we refer to the \emph{twist map} discussed
  below (Subsection~\ref{sub:heis-twist}).
\end{myexample}

One characteristic feature of nilquadratic groups is the \emph{symplectic structure} they induce. We recall some basic facts, essentially
following~\cite{BonattoDikranjan2017}. We call a subgroup $H$ of a group $G$ {\tt self-centralizing} iff
$C_G(H)\subseteq H$. 
\begin{lemma}\label{lemma-max-abelian-lemma} Let $H\le G$ in $\Grp$. Then
\begin{enumerate}\label{max-abelian-lemma}
\item $C_G(H)\subseteq H\iff Z(H)=C_G(H)$;
\item $C_G(H)\subseteq H$ and $H$ abelian $\iff$ $H$ maximal abelian;
\item $H$ maximal abelian $\iff$ $C_G(H)=H$.
\end{enumerate}
\end{lemma}
\begin{proof}
1. $C_G(H)\subseteq H\Imp Z(H)=C_G(H)$ is obvious. Trivially also $Z(H)=C_G(H)\Imp C_G(H)=Z(H)\subseteq H$.

2. Assume $C_G(H)\subseteq H$ and $H\in\Ab$. If $H\le A\le G$ with $A\in\Ab$ then $A\subseteq C_G(H)\subseteq H$, hence $A=H$, i.e., $H$ is max. abelian. Conversely, assume $H$ being max. abelian.
Take $x\in C_G(H)$. Then $H\cdot\la x\ra$ is an abelian group containing $H$, whence $x\in H$. Consequently $C_G(H)\subseteq H$ and $H\in\Ab$.

3. is obvious from 1. and 2.
\end{proof}

\begin{proposition}\label{L2} Let~$E\colon T \oset{\iota}{\rightarrowtail} H \oset{\pi}{\twoheadrightarrow} P$ be an exact sequence of groups. Then
$E\text{ is central }\iff k_H\text{ factors }\xymatrix{H\times H\ar[d]_-{\pi\times\pi}\ar[r]^-{[\bullet,\bullet]}&H\cr P\times P\ar[ur]_-\omega}$
\end{proposition}
\begin{proof}
Let $\varepsilon$ be central, $\pi(u)=\pi(u')\land\pi(v)=\pi(v')$. Then $u^{-1}u'=\iota(a)$, $v^{-1}v'=\iota(b)$ i.e., $u'=u\iota(a)$, $v'=v\iota(b)$
\begin{equation*}
[u',v']=u\iota(a)v\iota(b)\iota(a)^{-1}u^{-1}\iota(b)^{-1}v^{-1}=uvu^{-1}v^{-1}=[u,v]
\end{equation*}
Thus $\omega(\pi(u),\pi(v)):=[u,v]$ is well-defined. Conversely, assume that $[\bullet,\bullet]$ factors through $P\times P$. For $c\in T$, $e\in H$
\begin{eqnarray*}
[\iota(c),e]&=&(\omega\circ\pi\times\pi)(\iota(c),e)=\omega(\pi \iota(c),\pi(e))=\omega(1,\pi(e))\cr
&=&\omega(\pi(1),\pi(e))=[1,e]=1
\end{eqnarray*}
Thus $\iota(T)\subseteq Z(H)$ i.e., $E$ is central (in particular, $T\in\Ab)$.
\end{proof}

\begin{corollary}\label{cor-equivalent-central-extensions} Consider central extensions of $P\in\Grp$
\begin{equation*}
\varepsilon:1\lr T\bri E\brpi P\lr1,\ \varepsilon':1\lr T\brj F\brp P\lr1.
\end{equation*}
If $\varepsilon\sim\varepsilon'$ via $(1_T,\varphi,1_P)\colon\varepsilon\cong\varepsilon'$, then $\varphi'\circ\omega_\varepsilon=\omega_{\varepsilon'}$
(where $\varphi'$ is induced by $\varphi$).
\end{corollary}
\begin{proof}
\begin{equation*}\xymatrix{
& F\times F\ar[dd]_-{p\times p}\ar[dddrr]^-{k_F}&&\cr
E\times E\ar[dd]_-{\pi\times\pi}\ar[dddrr]^{k_E}\ar[ur]^-{\varphi\times\varphi} &&&\cr
& P\times P\ar[drr]^-{\omega_{\varepsilon'}}&&\cr
P\times P\ar[drr]^-{\omega_\varepsilon}\ar@{=}[ur]&&& F'\cr
&& E'\ar[ur]_-{\varphi'}&\cr
&&&
}\end{equation*}
$(u,v)\in F\times F$, $u=\varphi(x)$, $v=\varphi(y)$.
\begin{eqnarray*}
(\varphi'\circ\omega_\varepsilon\circ p\times p)(u,v)&=&\varphi\big(\omega_\varepsilon(p\varphi(x),p\varphi(y))\big)=\varphi\big(\omega_\varepsilon(\pi(x),\pi(y))\big)\cr
&=&\varphi[x,y]=[\varphi(x),\varphi(y)]=[u,v]
\end{eqnarray*}
Thus, $\varphi'\circ\omega_\varepsilon=\omega_{\varepsilon'}$.
\end{proof}

\begin{corollary}\label{commutator-bilinear-omega-form}\label{cor-omega-bilinear} Let $P\in\Grp$ and
$\varepsilon\colon T \oset{\iota}{\rightarrowtail} E \oset{\pi}{\twoheadrightarrow} P$
be central. Then
\begin{enumerate}
\item $k_E\colon E\times E\lr E'$ is bilinear $\iff$ $\omega_\varepsilon\colon P\times P\lr E'$ is bilinear $\iff$ $E'\subseteq Z(E)$.
\item $P\in\Ab$ $\Imp$ $\omega_\varepsilon\colon P\times P\lr T$ is bilinear.
\end{enumerate}
\end{corollary}
\begin{proof}
\begin{eqnarray*}
&&\omega_\varepsilon(\pi(u)\pi(v))=[uv,w]\cr
&&\omega_\varepsilon(\pi(u)\pi(w))\omega_\varepsilon(\pi(v)\pi(w)=[u,w][v,w].
\end{eqnarray*}
By Fact \ref{fct:char-cent-ext}, point (1) and Fact \ref{fct:char-cent-exseq}
\begin{equation*}
\omega_\varepsilon\text{ bilinear }\iff k_E\text{ bilinear }\iff E'\subseteq Z(E)\iff E\in{\rm Nil}_2
\end{equation*}
If $P\in\Ab$ then $E'\subseteq i(T)\subseteq Z(E)$, hence $\omega_\varepsilon$ is bilinear and may considered having values in $T$
\begin{equation*}
\omega_\varepsilon\colon P\times P\lr T,\ \omega_\varepsilon(\pi(u),\pi(v))=i^{-1}[u,v].
\end{equation*}
\end{proof}

\begin{proposition}\label{max-abelian-central-ext} Let $G\le P\in\Grp$ and
$\varepsilon\colon T \oset{\iota}{\rightarrowtail} E \oset{\pi}{\twoheadrightarrow} P$ central. Then
\begin{enumerate}
\item $C_E(\widetilde{G})=\pi^{-1}\big(G^\perp)$;
\item $C_E(\widetilde{G})\subseteq\widetilde{G}\iff G^\perp\subseteq G$;
\item $\widetilde{G}$ maximal abelian $\iff$ $G=G^\perp$.
\end{enumerate}
\end{proposition}
\begin{proof}
1. If $e\in C_E(\widetilde{G})$ then $[e,\widetilde{x}]=1$ $\forall\widetilde{x}\in\widetilde{G}$, $\omega_\varepsilon(\pi(e),x)=1$ $\forall x\in G$, $\pi(e)\in G^\perp$, $e\in\pi^{-1}(G^\perp)$. If, conversely, $e\in\pi^{-1}(G^\perp)$ then $\omega_\varepsilon(\pi(e),x)=1$ $\forall x\in G$,
$[e,\widetilde{x}]=1$ $\forall\widetilde{x}\in\widetilde{G}$, hence $e\in C_E(\widetilde{G})$.
2. $C_E(\widetilde{G})\subseteq\widetilde{G}\iff\pi^{-1}(G^\perp)\subseteq\pi^{-1}(G)\iff G^\perp\subseteq G$.\\

3. If $\widetilde{G}$ max. abelian then, by Lemma \ref{lemma-max-abelian-lemma}, $\widetilde{G}=C_E(\widetilde{G})$ hence
\begin{equation*}
G=\pi(\widetilde{G})=\pi C_E(\widetilde{G})=\pi\pi^{-1}(G^\perp)=G^\perp.
\end{equation*}
Conversely, if $G=G^\perp$ then
\begin{equation*}
\widetilde{G}=\pi^{-1}(G)=\pi^{-1}(G^\perp)=C_E(\widetilde{G}).
\end{equation*}
hence, again by Lemma \ref{lemma-max-abelian-lemma}, $\widetilde{G}$ is maximal abelian.
\end{proof}
The following theorem lists once more the relevant facts

\begin{theorem}\label{wiederholung} Let $\varepsilon:1\lr T\bri E\brpi P\lr1$ be central.
\begin{enumerate}
\item The commutator of $E$ factors through $P\times P$ $\xymatrix{E\times E\ar[d]_-{\pi\times\pi}\ar[r]^-{[\bullet,\bullet]}&E'\cr P\times P\ar[ur]_-\omega}$.
\item Let $G\le P$. Then $\widetilde{G}$ max. abelian $\iff G^\perp=G$.
\item If $G\le P$ and $G^\perp=G$ then $G\in\Ab$.
\item If $P=G\times\Gamma$ and $G^\perp=G$ and $\Gamma^\perp=\Gamma$ then $P\in\Ab$ and $E'\le i(T)\le Z(E)$.
\item If $X\subseteq P$ then\\
$X^\perp=\{u\mid\forall_{x\in X}\,\omega(u,x)=1\}=\{u\mid\forall_{x\in X}\,\omega(x,u)=1\}$

\noindent is a subgroup of $P$.
\end{enumerate}\end{theorem}

\begin{definition}
  \label{def:comm-form}
  Let~$E\colon T \oset{\iota}{\rightarrowtail} H \oset{\pi}{\twoheadrightarrow} P$ be an extension
  in~$\SES_2$ and choose an arbitrary set-theoretic section~$s$ of~$\pi$. The \emph{commutator form}
  $\omega_E\colon P \times P \to T$ is defined by~$\omega_E(w,z) := \iota^{-1} \, [s(w), s(z)]$.
\end{definition}

It is easy to see that $[s(w), s(z)] \in T$ and that~$\omega_E(w, z)$ does not depend on the choice
of~$s$; hence~$\omega_E$ is well-defined. Moreoever, one checks that equivalent extensions have the
same commutator form, so~$\omega$ depends only on the cohomology class of any
\emph{cocylce}~$\gamma \in Z^2(P, T)$ describing the equivalence class of~$E$. Explicitly, one
obtains the commutator form~$\omega_E(w,z) = \gamma(w,z)/\gamma(z,w)$ in terms of the cocycle.

Without taking recourse to a section, one can in fact define a
``commutator
form''~$[,]\colon H/\zentrum(H) \times H/\zentrum(H) \to [H,H]$ for an
arbitrary group~$H$. But it turns out that~$[,]$ is bilinear precisely
when~$H$ is nilquadratic; see Fact~2.4
in~\cite{BonattoDikranjan2017}. Since we are here only interested in
the nilquadratic setting, the commutator form~$\omega_E$ is in fact
bilinear as we shall now check. Of course~$\omega_E$ is always
alternating in the sense that~$\omega_E(z, z) = 1$ for all~$z \in
P$. As usual this implies that~$\omega_E$ is antisymmetric, which here
takes on the somewhat unusual
form~$\omega_E(w,z) \, \omega_E(z,w) = 1$. As an \emph{alternating
  bilinear form}, we may thus think of the commutator form
as~$\omega_E \in \Hom_{\ZZ}(\Lambda^2 P, T)$. But let us first provide
the short proof of bilinearity.

\begin{proposition}
  \label{prop:symp-dual}
  Let~$E\colon T \oset{\iota}{\rightarrowtail} H \oset{\pi}{\twoheadrightarrow} P$ be an extension
  in~$\SES_2$. Then~$\omega_E$ is bilinear.
\end{proposition}
\begin{proof}
  We show~$\omega(w_1 + w_2, z) = \omega(w_1, z) \, \omega(w_2, z)$
  for~$w_1, w_2, z \in P$. We may choose a section~$s$ of~$\pi$
  with~$s(w_1 + w_2) = s(w_1) \, s(w_2)$; as remarked above, the
  commutator form~$\omega_E$ is independent of these choices. Let us
  then write $u_1 = s(w_1)$, $u_2 = s(w_2)$ and~$v = s(z)$ for the
  corresponding elements over~$\pi$. Since~$[H,H] \le \zentrum(H)$ by
  Fact~\ref{fct:char-cent-ext} we obtain
  \begin{align*}
    [u_1 ,v] \, [u_2, v] &= u_1 v u_1^{-1} [v^{-1}, u_2]v^{-1} = u_1 v \, [v^{-1}, u_2] \, u_1^{-1} v^{-1}\\
    &= u_1 u_2 v u_2^{-1} u_1^{-1} v^{-1} = [u_1 u_2, v],
  \end{align*}
  which implies~$\omega(w_1 + w_2, z) = \omega(w_1, z) \, \omega(w_2, z)$ as desired. By
  antisymmetry, we have also~$\omega(w, z_1 + z_2) = \omega(w, z_1) \, \omega(w, z_2)$.
\end{proof}

For the remainder of this section (and indeed the rest of the paper), we shall work
with~$\SES_2$. Therefore we may work with the alternating
form~$\omega_E \in \Hom_{\ZZ}(\Lambda^2 P, T)$ much in the same way as in \emph{classical symplectic
  geometry} (where the abelian groups are finite-dimensional vector spaces, i.e.\@ equipped with an
additional scalar action).

In our context, the fundamental notion is the \emph{symplectic
  orthogonal} of any subgroup~$G \le P$. We give the definition in a
slightly more general setting; this will be useful for bringing out
some parallels with the ``natural orthogonal''
of~\cite[\S~A.1]{RegensburgerRosenkranz2009}.

\begin{definition}
  \label{def:orthogonal}
  Let~$M, N, T$ be modules over a commutative unital ring~$R$. For
  fixed bilinear form~$\alpha\colon M \times N \to T$, the
  \emph{orthogonal} of submodules~$M_0 \le M$ and~$N_0 \le N$ is given
  by~$M_0^\perp := \{ y \in N \mid M_0 \perp_\alpha y \}$
  and~$N_0^\perp = \{ x \in M \mid x \perp_\alpha N_0 \}$, respectively.
\end{definition}

As usual, \emph{orthogonality relation} is defined
by~$x \perp_\alpha y \Leftrightarrow \alpha(x,y) = 0$ for~$(x,y) \in M \times N$. Moreover, we use
the notation~$M_0 \perp_\alpha y$ as short\-hand for $\forall_{x \in M_0}\: x \perp_\alpha y$;
similarly with~$x \perp_\alpha N_0$.

Strictly speaking, we should have called~$M_0^\perp$ and~$N_0^\perp$ the \emph{left} and
\emph{right} orthogonal, respectively. In applications, either membership in~$M$ and~$N$ serves to
disambiguate the two orthogonals (Example~\ref{ex:nat-orth} below) or the orthogonality relation is
in fact symmetric so that the two notions coincide (Examples~\ref{ex:symm-orth}
and~\ref{ex:sympl-orth} below). For principal rings like fields or~$\ZZ$, it is known---see for
example~\cite[Thm.~11.4]{Roman1992}, which generalizes to this setting---that~$\perp_\alpha$ is a
symmetric relation iff~$\alpha$ is symmetric or alternating.

\begin{proposition}
  If~$\alpha\colon M \times N \to T$ is a bilinear form, its orthogonals create an antitone Galois
  connection between the submodules of~$M$ and $N$, with the
  biorthogonals~$\,^{\perp\perp}$ as closure operators.
\end{proposition}
\begin{proof}
  It is obvious that both orthogonals are antitone maps with respect to inclusion and
  that~$M_0 \le M_0^{\perp\perp}$ as well as $N_0 \le N_0^{\perp\perp}$ holds. It follows
  that~$N_0 \le M_0^\perp \Leftrightarrow M_0 \le N_0^\perp$, so we have an anitone Galois
  connection. The statement about the biorthogonals is a general property of Galois
  connections~\cite[\S7.27]{DaveyPriestley2002}.
\end{proof}

The fixed points of the closure operator are the \emph{closed submodules} of the Galois
connection. To be specific, let us write
\begin{align*}
  \mathbf{Cl}(M) &= \{ M_0 \le M \mid M_0^{\perp\perp} = M_0 \},\\
  \mathbf{Cl}(N) &= \{ N_0 \le N \mid N_0^{\perp\perp} = N_0 \}
\end{align*}
for the corresponding posets on the left and right side of the Galois connection. They are in fact
not just posets but lattices.

\begin{proposition}
  \label{prop:Galois-compl-lattice}
  The posets~$\mathbf{Cl}(M)$ and~$\mathbf{Cl}(N)$ induced by the bilinear
  form~$\alpha\colon M \times N \to T$ are complete lattices with~$M_1 \wedge M_2 = M_1 \cap M_2$
  and~$M_1 \vee M_2 = (M_1+M_2)^{\perp\perp}$ for~$\mathbf{Cl}(M)$, similarly
  for~$\mathbf{Cl}(N)$.

  \smallskip

  \noindent The orthogonal is a lattice anti-isomorphism~$\mathbf{Cl}(M) \isomarrow \mathbf{Cl}(N)$.
\end{proposition}
\begin{proof}
  Since~$M_0 \mapsto M_0^{\perp\perp}$ is a closure operator~\cite[\S7.27i]{DaveyPriestley2002}, its
  system of closed sets is a topped intersection structure~\cite[\S7.4]{DaveyPriestley2002}. We
  conclude that~$\mathbf{Cl}(M)$ is a complete lattice with operations as
  given~\cite[\S2.32]{DaveyPriestley2002} and that~$M_0 \mapsto M_0^\perp$ is a lattice
  anti-isomorphism as claimed~\cite[\S7.27i]{DaveyPriestley2002}.
\end{proof}

Whereas the lattice~$\mathbf{Sub}(M)$ of \emph{all} submodules of~$M$ is modular, the lattice
$\mathbf{Cl}(M)$ is in general---see Example~\ref{ex:symm-orth}---\emph{not
  modular}. Of course any sublattice of a modular lattice is again modular, but~$\mathbf{Cl}(M)$ is
\emph{not} a sublattice of~$\mathbf{Sub}(M)$ since its join operation differs in general.

It is clear that the lattice~$\mathbf{Sub}(M)$ is bounded by~$0$ and~$M$. As to the
lattice~$\mathbf{Cl}(M)$, it is clear that~$M$ is still the maximal element since it is obviously
closed. But~$0$ is closed and hence the minimal element of~$\mathbf{Cl}(M)$ if and only
if~$M^\perp = 0$, which is equivalent to~$\alpha\colon M \times N \to T$ being \emph{non-degenerate}
(on the left). Since we always have~$0^\perp = M$, the orthogonal will then act in between the
global bounds~$0$ and~$M$ \emph{inclusively}, swapping the two by its action.

It should also be noted that the orthogonal in general is \emph{not a lattice endomorphism}
on~$\mathbf{Sub}(M)$ since there we have
\begin{align}
  \label{eq:plus-homo}
  (M_1 + M_2)^\perp &= M_1^\perp \cap M_2^\perp
  \quad\text{but only}\\
  \label{eq:inter-nonhomo}
  (M_1 \cap M_2)^\perp &\ge M_1^\perp + M_2^\perp,
\end{align}
and it is straightforward to give examples where the latter inclusion is strict (see
Examples~\ref{ex:symm-orth}, \ref{ex:nat-orth} and~\ref{ex:sympl-orth} below). Since the left-hand
side is closed, it may be replaced by~$M_1^\perp \vee M_2^\perp$ to strengthen the inequality.

As to be expected, the biorthogonal also fails to be a lattice endomorphism
on~$\mathbf{Sub}(M)$. But we do have the \emph{biorthogonal inequalities}
\begin{align*}
  (M_1 \cap M_2)^{\perp\perp} &\le M_1^{\perp\perp} \cap M_2^{\perp\perp},\\
  (M_1 + M_2)^{\perp\perp} &\ge M_1^{\perp\perp} + M_2^{\perp\perp},
\end{align*}
where the first becomes equal iff the strengthened~\eqref{eq:inter-nonhomo} does for~$M_1, M_2$, and
the second iff the original~\eqref{eq:inter-nonhomo} does for~$M_1^\perp, M_2^\perp$. This asymmetry
can be traced to the fact that intersections of closed submodules are again closed while sums of
closed submodules generally are not (Proposition~\ref{prop:Galois-compl-lattice}).

Let us now have a look how these relations play out in some important special settings. The
orthogonals introduced in Examples~\ref{ex:symm-orth}, \ref{ex:nat-orth} and~\ref{ex:sympl-orth}
are, respectively, called the~\emph{symmetric orthogonal}, the~\emph{natural orthogonal} and
the~\emph{symplectic orthogonal}.

\begin{myexample}
  \label{ex:symm-orth}
  The simplest examples is when~$M=N$ are inner product spaces over a field~$T=K$
  with~$\alpha(x,y) = \inner{x}{y}$ the \emph{inner product}, where the orthognal has its original
  geometric significance. By definition, this requires~$\alpha$ to be nondegenerate. This includes
  especially the setting of \emph{Hilbert spaces}, where the closed subspaces in the sense of the
  Galois connection are precisely the closed subspaces in the sense of the
  topology~\cite[Cor.2.2.4]{KadisonRingrose1997}. The lattice~$\mathbf{Cl}(H)$ is not modular if~$H$
  is an infinite-dimensional Hilbert space~\cite[Prop.~4.4]{Redei2013}; cf.\@ also
  \cite[Prob.~14]{Halmos2012}.

  For a finite-dimensional vector space (which is naturally a Hilbert space when~$K=\RR$), all
  subspaces are closed. Otherwise, it is also easy to find examples where the
  inclusion~\eqref{eq:inter-nonhomo} is strict: Take any disjoint dense subsets~$M_1$ and~$M_2$ of
  the separable real Hilbert space~$M = L^1(\RR)$, for example the linear span~$M_1$ of the Haar
  basis and~$M_2 = \RR[x]$.  Then~$M_1^{\perp\perp} = M = M_2^{\perp\perp}$ implies that
  both~$M_1^\perp$ and~$M_2^\perp$ collapse to~$M^\perp = 0$, so
  that~$(M_1 \cap M_2)^\perp = 0^\perp = M$ is clearly larger
  than~$M_1^\perp \vee M_2^\perp = M_1^\perp + M_2^\perp = 0$.
\end{myexample}

\begin{myexample}
  \label{ex:nat-orth}
  If~$M$ and~$N$ are modules over $T=R$ with a given bilinear form
  $\alpha = b\colon M \times N \to R$, we recover the setting
  of~\cite[\S~A.1]{RegensburgerRosenkranz2009}, which is particularly important
  when~$N = M^* = \Hom_R(M,R)$ is the dual module and $\alpha(x, \beta) = \beta(x)$ is the
  corresponding \emph{natural pairing}. In the algebraic approach to boundary
  problems~\cite{RosenkranzRegensburger2008a}%
  \cite{RosenkranzPhisanbut2013}, the module~$M$ is in fact a vector space over a suitable
  field~$K$. If~$M$ is infinite-dimensional over~$K$, primal subspaces~$M_0 \le M$ are always closed
  but there are many plenty dual subspaces~$N_0 \le M^*$ that are not
  closed~\cite[\S9.2/6]{Koethe1969}. Since in this case~$\mathbf{Cl}(M) = \mathbf{Sub}(M)$ is
  modular, its isomorphic twin~$\mathbf{Cl}(M^*)$ is as well. Moreover, one
  has~$N_1 \vee N_2 = N_1 + N_2$ in~$\mathbf{Cl}(M^*)$, so the biorthogonal may be dropped for
  finite (but not for infinite) joins~\cite[\S9.3/3]{Koethe1969}.

  The inclusion~\eqref{eq:inter-nonhomo} is of course an identity in~$\mathbf{Cl}(M)$ since primal
  subspaces are always closed and the finite join in~$\mathbf{Cl}(M^*)$ is the sum of subspaces. But
  for dual subspaces~$N_1, N_2 \le M^*$, one may construct a counterexample similar to the one in
  Example~\eqref{ex:symm-orth}. For example, choose~$M = C^\infty(\RR^2)$ and set
  \begin{align*}
    N_1 &= [\ev_{\xi,\eta} \circ \der_x \mid (\xi, \eta) \in \RR^2] + [\ev_{\eta-1, \eta}],\\
    N_2 &= [\ev_{\xi,\eta} \circ \der_y \mid (\xi, \eta) \in \RR^2] + [\ev_{\xi, \xi-1}].
  \end{align*}
  Then one may check that again~$N_1^\perp = 0 = N_2^\perp$ so that~$N_1$ and~$N_2$ are ``dense''
  disjoint subspaces of~$N^*$, and we obtain as before the strict
  inequality~$M = (N_1 \cap N_2)^\perp > N_1^\perp + N_2^\perp = 0$.
\end{myexample}

\begin{myexample}
  \label{ex:sympl-orth}
  For our present purposes, an \emph{alternating form}~$\alpha = \omega\colon P \times P \to T$ is
  given. Here~$P$ and~$T$ are abelian groups, i.e.\@ modules over $R = \ZZ$. As for the inner
  product spaces, the right and left orthogonal are identical since also here~$M=N$
  and~$x \perp_\alpha y \Leftrightarrow y \perp_\alpha x$. As we shall see below
  (Example~\ref{ex:sympl-spaces}), there are in general many examples of non-closed subgroups.
    
  We can see this, and various other properties, from the following example. Given a vector
  space~$V$ over a field~$K$ and using additive notation, the \emph{canonical symplectic
    structure}~$\Omega_V\colon P \times P \to K$ on the phase space~$P := V \times V^*$ is given
  by~$\Omega_V(x,\xi \mid x',\xi') := \xi'(x) - \xi(x')$. It is easy to check that the symplectic
  orthogonal of a subspace~$\adm \times \bspc \le P$ is here given
  by~$(\adm \times \bspc)^\perp = \bspc' \times \adm'$, where we write~$\adm'$ and~$\bspc'$ for the
  natural orthogonal of Example~\ref{ex:nat-orth}, so as to avoid confusion with the symplectic
  orthogonal. Hence~$\adm \times \bspc \in \mathbf{Cl}(P)$ iff~$\bspc \in \mathbf{Cl}(V^*)$.
  Setting~$\bspc_1 := 0 \times N_1$ and~$\bspc_2 := 0 \times N_2$ for the
  space~$M = C^\infty(\RR^2)$ used in Example~\ref{ex:nat-orth}, we obtain again an example where
  the inclusion~\eqref{eq:inter-nonhomo} is in fact strict.
\end{myexample}

The symplectic orthogonal will be the crucial one for us. Thus fix a group~$P$ with alternating
form~$\omega\colon P \times P \to T$. The following \emph{standard terminology} for
subgroups~$G \le P$ is adopted from symplectic geometry (where~$P$ is a vector space, having a
scalar action in addition to the additive structure):
\begin{itemize}
\item\emph{Symplectic:} $G \cap G^\perp = 0$ $\Leftrightarrow$
  $\omega|_G$ nondegenerate
\item\emph{Isotropic:} $G \le G^\perp$ $\Leftrightarrow$
  $\omega|_G = 1$
\item\emph{Coisotropic:} $G^\perp \le G$ $\Leftrightarrow$ $\tilde\omega$
  nondegenerate on~$G/G^{\perp}$
\item\emph{Lagrangian:} $G = G^\perp$ $\Leftrightarrow$ $\omega$
  isotropic and coisotropic
\end{itemize}

We denote the corresponding classes by~$\mathbf{Sym}(P)$, $\mathbf{Iso}(P)$, $\mathbf{Co}(P)$,
$\mathbf{Iso}(P) \cap \mathbf{Co}(P)$, omitting reference to~$P$ where it is clear from the
context. It is obvious that~$\mathbf{Iso}$ and~$\mathbf{Co}$ are, respectively, downward and upward
closed. Beyond that, however, one observes some awkward \emph{asymmetries}, ultimately due to the
failure of closure for arbitrary subgroups~$G \le P$. For example, one has
$G \in \mathbf{Iso} \Leftrightarrow G^{\perp\perp} \in \mathbf{Iso}$
whereas~$G \in \mathbf{Co} \Rightarrow G^{\perp\perp} \in \mathbf{Co}$, which in general cannot be
upgraded to an equivalence. Moreover, we
have~$G \in \mathbf{Iso} \Leftrightarrow G^\perp \in \mathbf{Co}$ versus again just an
implication~$G \in \mathbf{Co} \Rightarrow G^\perp \in \mathbf{Iso}$.

It should also be noted that~$\mathbf{Iso}$ and~$\mathbf{Co}$ are \emph{not sublattices} of either
$\mathbf{Sub}$ or~$\mathbf{Cl}$; they are just meet and join sub-semilattices, respectively (since
they are downward/upward closed). Hence the Lagrangian subgroups---which are of course
symplectically closed---are neither: For~$G, H \in \mathbf{Iso} \cap \mathbf{Co}$ we
have~$(G \cap H)^\perp = G^\perp + H^\perp = G + H$
and~$(G + H)^\perp = G^\perp \cap H^\perp = G \cap H$.

Let us illustrate these points for the \emph{canonical symplectic structure} introduced in
Example~\ref{ex:sympl-orth}.

\begin{myexample}
  \label{ex:sympl-spaces}
  As pointed above (Example~\ref{ex:sympl-orth}), one has~$\adm'' = \adm$ for primal subspaces
  $\adm \le V$, but in general only~$\bspc'' \ge \bspc$ for dual subspaces~$\bspc \le V^*$.
  Hence~$\adm \times \bspc \le P$ is symplectically closed iff~$\bspc$ is orthogonally closed. In
  all cases, the subspaces~$\adm \times 0$ and~$0 \times \bspc$ are iso\-tropic, $\adm \times V^*$
  and~$V \times \bspc$ coisotropic, $\adm \times \adm'$ and~$\bspc' \times \bspc''$ Lagrangian. For
  an example of a symplectic subspace, assume~$\adm \dotplus \bspc^\perp = V$ and~$\bspc = \bspc''$,
  for instance $\adm = \ker(T)$ for regular boundary problems~$(T, \bspc)$ over an
  integro-differential algebra~$(V, \der, \cum)$; then~$\adm \times \bspc$ is a symplectic subspace.
  If~$\bspc$ is not orthogonally closed,
  then~$G := \bspc' \times \bspc \in \mathbf{Iso} \setminus \mathbf{Co}$ but we
  have~$G^\perp = G^{\perp\perp} = \bspc' \times \bspc'' \in \mathbf{Iso} \cap \mathbf{Co}$.

  As to the lattice operations, take~$G = \adm \times 0$ and~$H = 0 \times \bspc$, again
  with~$\adm \dotplus \bspc^\perp =V$. Then we know that~$G, H \in \mathbf{Iso}$,
  but~$G + H = \adm \times \bspc$ is symplectic and hence certainly not isotropic. Similarly,
  setting now $G = \adm \times V^*$ and~$H = V \times \bspc$ for the same~$\adm$ and~$\bspc$ will
  ensure that we have~$G, H \in \mathbf{Co}$ while~$G \cap H = \adm \times \bspc$ is again
  symplectic and thus not coisotropic.
\end{myexample}

What all this tells us is that the classes~$\mathbf{Iso}$ and~$\mathbf{Co}$ do not sit nicely
in~$\mathbf{Sub}(P)$. This is completely analogous to the case of the natural orthogonal proposed in
Example~\ref{ex:nat-orth} above. In fact, one might view the symplectic orthogonal (on a vector
space) as a kind of amalgamation of the natural orthogonal. The remedy taken for the latter---in
particular for the algebraic approach to boundary problems~\cite{RosenkranzRegensburger2008a}%
\cite{RosenkranzPhisanbut2013}---is to restrict the dual subspaces~$\bspc \le V^*$ to the
orthogonally closed ones~\cite[\S9.3]{Koethe1969}: From the viewpoint of the Galois connection, a
non-closed subspace~$\bspc$ is just an arbitrary choice of ``generators'' for the actual object of
interest: The invariant spaces taking part in the Galois connection, namely the
pair~$(\bspc', \bspc'')$ that stabilizes under successive orthogonals. In the same vein, it is
apposite to discard the full subgroup lattice~$\mathbf{Sub}(P)$ in favor of the
lattice~$\mathbf{Cl}(P)$ of \emph{symplectically closed groups} expounded in
Proposition~\ref{prop:Galois-compl-lattice}.

Thus the orthogonal is a lattice \emph{anti-automorphism}~$\mathbf{Cl} \isomarrow \mathbf{Cl}$,
inducing a semilattice \emph{anti-isomorphism} between
$\mathbf{Iso}' := \mathbf{Co}^\perp = \mathbf{Iso} \cap \mathbf{Cl}$
and~$\mathbf{Co}' := \mathbf{Iso}^\perp = \mathbf{Co} \cap \mathbf{Cl}$ that \emph{stabilizes the
  Lagrangian subgroups} $\mathbf{Iso}' \cap \mathbf{Co}' = \mathbf{Iso} \cap \mathbf{Co}$.

\begin{myexample}
  As in the case of the inner orthogonal (Example~\ref{ex:symm-orth}), the lattice of closed
  subgroups is in general \emph{not modular}. In fact, one may adopt the setting in a way similar to
  the canonical symplectic structure (Example~\ref{ex:sympl-orth}). If~$V$ is any vector space with
  inner product~$\inner{}{}$, we define the symplectic form~$\Psi_V\colon P \times P \to K$
  on~$P := V \times V$ by the analogous relation
  $\Psi_V(x, \xi \mid x', \xi') = \inner{\xi'}{x} - \inner{\xi}{x'}$. Again we have
  $(\adm \times \bspc)^\perp = \bspc' \times \adm'$ for~$\adm, \bspc \le V$, using the prime now to
  denote the inner orthogonal as opposed to the symplectic orthogonal $^\perp$. It follows
  that~$\adm \le V$ is symmetrically closed iff~$\adm \times 0$ is symplectically closed. If we
  choose for~$V$ an infinite-dimensional Hilbert space, we have seen that~$\mathbf{Cl}(V)$ is not
  modular, hence the sublattice~$V \times 0$ of~$\mathbf{Cl}(P)$ is also not modular, showing
  that~$\mathbf{Cl}(P)$ cannot be modular either.
\end{myexample}

Recalling that we should restrict ourselves to symplectically closed subgroups, the isotropic and
coisotropic semilattices may be characterized by the \emph{extremal properties}, just as in
classical symplectic geometry (where closure is not needed due to the hypothesis of finite
dimension).

\begin{lemma}
  \label{lem:lagrangian-extremal}
  Let~$\omega$ be an alternating form on an abelian group~$P$. Then a subgroup~$G \le P$ is
  Lagrangian iff $G$ is maximal in~$\mathbf{Iso}'(P)$ iff $G$ is minimal in~$\mathbf{Co}'(P)$.
\end{lemma}
\begin{proof}
  We show first that~$G$ is Lagrangian iff~$G$ is maximal in~$\mathbf{Iso}(P)$. Assume
  that~$G \le P$ is Lagrangian. If~$H \ge G$ is isotropic, we have~$H \le H^\perp \le G^\perp = G$
  and thus~$G = H$. Hence~$G$ is maximal istropic. Similarly, if~$H \le G$ is coisotropic, we
  get~$G = G^\perp \le H^\perp \le H$ so that~$G = H$ and~$G$ is seen to be minimal
  coisotropic. Next assume~$G$ is a maximal isotropic subgroup of~$P$; we must show that~$G$ is then
  Lagrangian. If~$G^\perp$ is larger than~$G$, we take~$x \in G^\perp \setminus G$ and let~$H \le P$
  be the subgroup generated by~$x$ and~$G$. Each element of~$H$ has the form~$kx+g$ for~$k \in \ZZ$
  and~$g \in G$, and we have
  \begin{equation*}
    \omega(kx+g, k'x+g') = \omega(x,x)^{kk'} \, \omega(x,g')^k \, \omega(g,x)^{k'} \, \omega(g,g') = 1
  \end{equation*}
  since all four factors are unity: the first since~$\omega(x,x)=1$,
  the second and third because~$x \in G^\perp$, and the fourth by the
  hypothesis that~$G$ is isotropic. Since~$\omega$ is trivial on~$H$,
  this is a strictly larger isotropic group containing~$G$, which
  contradicts the maximality of~$G$. Hence we conclude
  that in fact~$G^\perp = G$.

  (Note that if~$G$ is Lagrangian then $G$ is minimal in~$\mathbf{Co}(P)$; the proof is similar to
  the maximality statement. It is not clear to us if the converse is true. Since we do not need this
  for the characterization in~$\mathbf{Cl}$, we refrain from further investigations of this case.)

  Now we show that maximality in~$\mathbf{Iso}(P)$ and in~$\mathbf{Iso}'(P)$ are equivalent for an
  arbitrary~$G \in \mathbf{Cl}(P)$. Given the former, assume now $H \in \mathbf{Iso}'(P)$
  satisfies~$G \le H$, then clearly~$G = H$ by maximality in~$\mathbf{Iso}(P)$. Hence~$G$ is also
  maximal in~$\mathbf{Iso}'(P)$. Conversely, assume this is the case. If~$H \in \mathbf{Iso}(P)$ is
  such that~$G \le H$, then also $G = G^{\perp\perp} \le H^{\perp\perp}$.
  Since~$H^{\perp\perp} \in \mathbf{Iso}'(P)$, we have~$G = H^{\perp\perp}$ by maximality
  in~$\mathbf{Iso}'(P)$ so that~$H \le H^{\perp\perp} = G \le H$ and hence~$G = H$. Together with
  the above maximality result, we see thus that~$G$ is Lagrangian iff~$G$ is maximal
  in~$\mathbf{Iso}(P)$. The minimality characterization is now an immediate consequence
  since~$G \mapsto G^\perp$ is a lattice
  anti-isomorphism~$\mathbf{Iso}' \isomarrow \mathbf{Co}'$.
\end{proof}

As the preceding result confirms, Lagrangian subgroups (briefly called ``Lagrangians'') of an
abelian group~$P$ reflect a significant and natural part of the symplectic structure. For our
algebraic investigation of the Fourier transform, we shall in fact need an even richer
structure---not just one but two Lagrangians, interlaced in a symmetric compound. Hence we call a
pair of Lagrangians~$(G, \Gamma)$ a \emph{Lagrangian bisection} if they form a direct
decomposition~$G \dotplus \Gamma = P$.

This is a straightforward generalization from the classical setting, where one just
requires~\cite[p.~21]{BatesWeinstein1997} the so-called ``transversality condition''
$G + \Gamma = P$. This may also be generalized: We call~$(P, \omega)$ a \emph{symplectic
  $\ZZ$-module} if~$P$ is a symplectic subgroup of itself. In other words, $\omega$ is to be
nondegerate on all of~$P$. In such cases (including the symplectic vector spaces of the classical
setting), transversality is sufficient.

\begin{lemma}
  Let~$(P, \omega)$ be a symplectic $\ZZ$-module with two Lagrangians $G, \Gamma \le P$ such
  that~$G + \Gamma = P$. Then~$(G, \Gamma)$ is a Lagrangian bisection.
\end{lemma}
\begin{proof}
  We need only show~$G \cap \Gamma = 0$, so assume~$z \in G \cap \Gamma$. Since we
  have~$G^\perp = G$ and~$\Gamma^\perp = \Gamma$, this implies that~$z$ is orthogonal to any element
  of~$G$ and also orthogonal to any element of~$\Gamma$. Now if $x+\xi$ with~$x \in G$
  and~$\xi \in \Gamma$ is an arbitrary element of~$P$, we obtain that
  $\omega(z,x+\xi) = \omega(z,x) \, \omega(z, \xi) = 1$. Hence~$z$ is orthogonal to all of~$P$,
  which implies~$z=0$ by the nondegeneracy of~$\omega$.
\end{proof}

We shall now apply the ``symplectic machinery'' for investigating nilquadratic
extensions~$E\colon T \oset{\iota}{\rightarrowtail} H \oset{\pi}{\twoheadrightarrow} P$. Indeed, we
know from Proposition~\ref{prop:symp-dual} that they always come with an alternating
form~$\omega_E$. It is thus to be expected that certain properties of~$H$ will be reflected in the
symplectic structure induced by~$\omega_E$. In particular, the subgroup lattice of~$H$ will become
visible as a kind of mirror image in~$(P, \omega_E)$. Certain subgroup types turn out to be
prominent in this context: Recall that a subgroup~$H_0$ of a group~$H$ is called
\emph{self-centralizing} if $C_H(H_0) \le H_0$ or equivalently if~$C_H(H_0) = \zentrum(H_0)$. A
self-centralizing abelian subgroup is the same as a \emph{maximal abelian subgroup}; such subgroups
can also be characterized by~$C_H(H_0) = H_0$. Most importantly, there is a counterpart of
Lagrangian bisections: For any group~$H$ with a distinguished abelian subgroup~$\hat{T}$, we
call~$(\tilde{G}, \tilde{\Gamma})$ an \emph{abelian bisection over $\hat{T}$} if~$\tilde{G}$
and~$\tilde{\Gamma}$ are maximal abelian subgroups of~$H$ such
that~$\tilde{G} \cap \tilde{\Gamma} = \hat{T}$ and~$\tilde{G} \, \tilde{\Gamma} = H$.

\begin{theorem}
  \label{thm:sympl-corr}
  Let~$E\colon T \oset{\iota}{\rightarrowtail} H \oset{\pi}{\twoheadrightarrow} P$ be an extension
  in~$\SES_2$ with symplectic form~$\omega_E \in \Hom_{\ZZ}(\Lambda^2 P, T)$.
  Then~$G \mapsto \tilde{G} := \pi^{-1}(G)$ is a monotone bijection between subgroups of~$P$ and
  those subgroups of~$H$ that contain~$\hat{T} := \iota T$, with the following properties:
  \begin{enumerate}
  \item The group~$H$ is abelian iff~$\omega_E$ is trivial.
  \item The extension~$E$ is strictly central iff~$\omega_E$ is nondegenerate.
  \item The subgroup~$\tilde{G}$ is abelian iff~$G$ is isotropic.
  \item The subgroup~$\tilde{G}$ is self-centralizing iff~$G$ is coisotropic.
  \item\label{it:corr-lag} The subgroup~$\tilde{G}$ is maximal abelian iff~$G$ is Lagrangian.
  \item\label{it:corr-sum} We have~$\tilde{G}_1 \, \tilde{G}_2 = H$ iff~$G_1 + G_2 = P$.
  \item We have~$\tilde{G}_1 \cap \tilde{G}_2 = \hat{T}$ iff $G_1 \cap G_2 = 0$.
  \item\label{it:ab-split} We get an abelian bisection $(\tilde{G}, \tilde{\Gamma})$ over~$\hat{T}$
    iff~$(G, \Gamma)$ forms a Lagrangian bisection.
  \end{enumerate}
\end{theorem}
\begin{proof}
$G \mapsto \pi^{-1}(G)$ is a monotone bijection between arbitrary subgroups of~$P$ and those subgroups of~$H$ that contain~$T$. We now
  go through all items in order.
  \begin{enumerate}
  \item\label{it:ab-eqv-trivial-comf} If~$H \in \Ab$,
    then~$\omega_E(w,z) = \iota^{-1} \, [s(w), s(z)] = 1$
    for~$w, z \in P$. Conversely, assume~$\omega_E$ is
    trivial. Then~$[s(w), s(z)] = 1$ for all $w, z \in P$. Since the
    section~$s$ is arbitrary, this means~$[u, v] = 1$ for
    all~$u,v \in H$.
  \item The nondegeneracy of~$\omega_E$ means~$\forall_{w \in P} \: [s(w), s(z)] = 1$ implies
    $z =0$. By the arbitrariness of the section~$s$, the premise is equivalent
    to~$s(z) \in \zentrum(H)$ and then---for the same
    reason---to~$\pi^{-1}(z) \subseteq \zentrum(H)$. Assuming nondegeneracy, let us now prove strict
    centrality.  Taking~$c \in \zentrum(H)$ and setting~$z = \pi(c)$, we
    have~$\pi^{-1}(z) \subseteq \zentrum(H)$ since~$c' \in \pi^{-1}(z)$ implies~$\pi(c'/c) = 1$ and
    hence~$c'/c \in \hat{T} \le \zentrum(H)$ by the exactness of~$E$ and
    Fact~\ref{fct:char-cent-exseq}. But then~$c' \in \zentrum(H)$, so we have
    indeed~$\pi^{-1}(z) \subseteq \zentrum(H)$, so the hypothesis of strict centrality
    yields~$z = 0$ or~$c \in \ker(\pi) = \hat{T}$. For the converse, we assume~$E$ is strictly
    central. Taking any $z \in P$ with~$\pi^{-1}(z) \subseteq \zentrum(H) = \hat{T}$, we must
    show~$z=0$. Since~$\pi$ is surjective, $z = \pi(c)$ for
    some~$c \in \pi^{-1}(z) \subseteq \hat{T}$, hence~$c = \iota(t)$ for some~$t \in T$. But
    then~$z = \pi(\iota(t)) = 0$ by the exactness of~$E$.
  \item\label{it:eq-iso} Isotropy of~$G \le P$ is equivalent to~$\omega_E|_G\colon G \times G \to T$
    being trivial. Since~$\hat{T} \le \tilde{G}$, we may restrict the exact sequence to
    \begin{equation*}
      \tilde{E}\colon T \oset{\iota}{\rightarrowtail} \tilde{G}
      \oset{\pi}{\twoheadrightarrow} P,
    \end{equation*}
    and we have~$\omega_{\tilde E} = \omega_E|_G$. Now the claim
    follows from Item~\eqref{it:ab-eqv-trivial-comf}.
  \item\label{it:eq-coiso} The condition~$C_H(\tilde{G}) \le \tilde{G}$ amounts to requiring for
    all~$u \in H$ that~$\forall_{v \in \tilde{G}} \: [u,v] = 1$ implies~$u \in \tilde{G}$, which is
    in turn equivalent to requiring for all~$z \in P$ and all sections~$s$ the implication
    \begin{equation*}
      \forall_{v \in \tilde{G}} \: [s(z), v] = 1
      \quad\text{implies}\quad
      z \in G .
    \end{equation*}
    It is easy to see that~$\tilde{G} = \hat{T} \, s(G)$, therefore
    the antecedent of the above implication is equivalent
    to~$\forall_{w \in G} \forall_{c \in T} \: [s(z), \iota(c) \,
    s(w)] = 1$. Since~$\hat{T}$ commutes with all of~$H$, the
    factor~$\iota(c)$ and the quantifier over~$c$ may be dropped, so
    the antecedent of the implication is actually equivalent
    to~$\forall_{w \in G} \: [s(z), s(w)] = 1$, which is the same
    as~$z \in G^\perp$.
  \item\label{it:eq-lag} This follows from~\eqref{it:eq-iso} and~\eqref{it:eq-coiso}. Alternatively,
    it may also be inferred from~\eqref{it:eq-iso} in conjunction with
    Lemma~\ref{lem:lagrangian-extremal}. (In the proof of the latter lemma, it has been observed
    that Lagrangian subgroups may also be characterized as the maximal isotropic ones in the
    \emph{full} subgroup lattice.)
  \item\label{it:eq-sum} Assume that~$G_1 + G_2 = P$. For arbitrary~$u \in H$, we must show
    that~$u \in \tilde{G}_1 \, \tilde{G}_2$. By the hypothesis~$G_1 + G_2 = P$, we
    have~$\pi(u) = z_1 + z_2$ for suitable~$(z_1, z_2) \in G_1 \times G_2$. Since~$\pi$ is
    surjective, there are~$u_1, u_2 \in H$ with~$\pi(u_1) = z_1$ and~$\pi(u_2) = z_2$. Using the
    exactness of~$E$, we conclude from~$\pi(u_1 u_2/u) = 0$ that~$u_1 u_2 / u = \iota(c)$ for
    some~$c \in T$, and we obtain the required
    representation~$u = \big( \iota(c^{-1}) \, u_1 \big) \, u_2 \in \tilde{G}_1 \, \tilde{G}_2$.
    Conversely, assume that~$\tilde{G}_1 \, \tilde{G}_2 = H$. Taking~$z \in P$ arbitrary, we must
    show that~$z \in G_1 + G_2$. We may pick any~$u \in H$ with~$\pi(u) = z$ and then
    choose~$u_1 \in \tilde{G}_1$ and~$u_2 \in \tilde{G}_2$ such that~$u = u_1 u_2$.
    Then~$z = \pi(u_1) + \pi(u_2) \in G_1 + G_2$.
  \item\label{it:eq-dir} Assuming~$G_1 \cap G_2 = 0$, take any~$u \in \tilde{G}_1 \cap \tilde{G}_2$.
    Then~$\pi(u) = 0$, and we obtain~$u \in \hat{T}$ since~$E$ is exact. Conversely, if
    $\tilde{G}_1 \cap \tilde{G}_2 = \hat{T}$ and~$z \in G_1 \cap G_2$ we have to show that~$z = 0$.
    Taking any~$u \in H$ with $\pi(u) = z$, it is clear that~$u \in \tilde{G}_1$
    and~$u \in \tilde{G}_2$, so we obtain~$u = \iota(c)$ for some~$c \in T$. But
    then~$z = \pi(\iota(c)) = 0$ as claimed.
  \item This follows from Items~\eqref{it:eq-sum}, \eqref{it:eq-dir} and~\eqref{it:eq-lag}.
  \end{enumerate}
  %
\end{proof}
The crucial property of Heisenberg groups is the existence of an abelian bisection or, equivalently
by Theorem~\ref{thm:sympl-corr}\eqref{it:ab-split}, the existence of a Lagrangian bisection in its
phase space. While this is a strong requirement, we should not expect such a bisection to be unique
as one can see even from the standard example in classical symplectic geometry.

\begin{myexample}
  The \emph{classical Heisenberg group}~$H_1(\RR)$ may be defined as the group upper triangular
  of~$3 \times 3$ matrices
  \begin{equation*}
    \left\{ \begin{pmatrix} 1 & \xi & c\\0 & 1 & x\\0 & 0 & 1 \end{pmatrix} 
    \Big|\: (\xi, x) \in \RR^2, c \in \Tor \right\}
  \end{equation*}
  with the multiplication law given by
  \begin{equation*}
    \begin{pmatrix} 1 & \xi & c\\0 & 1 & x\\0 & 0 & 1 \end{pmatrix} \cdot
    \begin{pmatrix} 1 & \xi & c\\0 & 1 & x\\0 & 0 & 1 \end{pmatrix}
    = \begin{pmatrix} 1 & \xi+\xi' & cc' \inner{\xi}{x'}_\beta\\0 & 1 & x+x'\\0 & 0 &
      1 \end{pmatrix},
  \end{equation*}
  under the bilinear form~$\beta\colon \RR \times \RR \to \Tor$
  with~$\inner{\xi}{x}_\beta := e^{i\tau \xi x}$. It is known that~$H_1(\RR)$ is nilquadratic, with
  phase space~$P = H_1(\RR)/\Tor \cong \RR^2$ via
  {\tiny$\medmat{1}{\xi}{\star}{0}{1}{x}{0}{0}{1}$}$\,\Tor \leftrightarrow (\xi, x)$ and symplectic
  form~$\omega(\xi, x; \xi', x') = \inner{\xi}{x'}_\beta/\inner{\xi'}{x'}_\beta$, easily seen to be
  nondegenerate. The standard Lagrangian bisection used in Physics is
  then~$(\tilde{G}, \tilde{\Gamma})$ with~$\tilde{G} = \RR \times 0 \le P$
  and~$\tilde{\Gamma} = 0 \times \RR \le P$. But there are may other ones, for
  example~$\tilde{G} = \{ (\xi, -\xi) \in P \mid \xi \in \RR \}$ with~$\tilde{\Gamma}$ as before.
\end{myexample}

We shall now continue working with the category~$\Cnt_2$ of central
extensions~$1 \triangleleft T \triangleleft H$, where the torus~$T$ is
somehow ``fixed in the background''. It plays the role of the
``dualizing object'', like $K$ for $K$-vector spaces or $\Tor$ for LCA
groups. We shall thus refer to~$H$ as a nilquadratic group \emph{over
  $T$}. When~$H'$ is another nilquadratic group over~$T$, we speak of
a group homomorphisms~$\phi\colon H \to H'$ \emph{over~$T$}
if~$\phi|_T = 1_T$. We call such morphisms \emph{toroidal} when the
torus~$T$ is clear from the context. One may think of this as a
generalized sort of $T$-linearity.

\begin{myremark}
  \label{rem:toroidal-as-linear}
  Note that the integral group ring~$\ZZ[H]$ may be endowed with the
  natural map~$\ZZ[T] \to \End_{\ZZ}(\ZZ[H])$ given by the
  action~$c \cdot \sum_{i \in \ZZ} n_i h_i = \sum_{i \in \ZZ} n_i \,
  (ch_i)$; thus~$\ZZ[H]$ is a (in general) noncommutative algebra over
  the ring~$\ZZ[T]$. Identifying the groups with their integral group
  rings, the group homomorphisms~$H \to H'$ appear as ring
  homomorphisms, the toroidal ones as algebra homomorphisms.
\end{myremark}

We come now to the central definition of this section---\emph{Heisenberg groups in their most
  general form} (in the scope of this paper). 

\begin{definition}
  \label{def:heis-grp-0}
  A \emph{Heisenberg group} over~$T$ is a nilquadratic group~$H$ together with an abelian bisection
  over~$\hat{T}$.
\end{definition}

We have~$\hat{T} \le \zentrum(H)$ for any Heisenberg group~$H$ over~$\hat{T}$. Indeed, for a given
abelian bisection~$(\tilde{G}, \tilde{\Gamma})$ of~$H$, any element~$u \in H$ may be written
as~$u = \tilde{x} \tilde{\xi}$ with~$(\tilde{x}, \tilde{\xi}) \in \tilde{G} \times \tilde{\Gamma}$
so that~$cu = uc$ for~$c \in \hat{T} = \tilde{G} \cap \tilde{\Gamma}$ since~$c$ commutes with
both~$\tilde{x}$ and~$\tilde{\xi}$. The inclusion~$\hat{T} \le \zentrum(H)$ is one half of the
characterization of central extensions in Fact~\ref{fct:char-cent-exseq}. The other half being
supplied by the following Lemma, we see that
\begin{equation}
  \label{eq:cent-extn-of-heis}
  1 \to T \hookrightarrow H \twoheadrightarrow H/T \to 0
\end{equation}
is in fact a \emph{central extension}.

\begin{lemma}
  \label{lem:heis-com-tor}
  If~$H$ is a Heisenberg group over~$T$, we have~$[H, H] \le T$.
\end{lemma}
\begin{proof}
  Let~$(\tilde{G}, \tilde{\Gamma})$ be an abelian bisection
  for~$H$. We show~$[u, u'] \in T$ for arbitrary~$u, u' \in
  H$. Since~$\tilde{G}$ and~$\tilde{\Gamma}$ generate~$H$, we
  have~$u = \tilde{x} \tilde{\xi}$ and $u' = \tilde{x}' \tilde{\xi}'$
  for suitable~$\tilde{x}, \tilde{x}' \in \tilde{G}$
  and~$\tilde{\xi}, \tilde{\xi}' \in \tilde{\Gamma}$. Hence we obtain
  \begin{align*}
    [u, u'] &= \tilde{x} \tilde{\xi} \tilde{x}' \tilde{\xi}' \tilde{\xi}^{-1} \tilde{x}^{-1} \tilde{\xi}'^{-1} \tilde{x}'^{-1}\\
    &= (\tilde{x} \tilde{\xi} \tilde{x}^{-1}) (\tilde{x} \tilde{x}' \tilde{\xi}' \tilde{x}'^{-1} \tilde{x}^{-1}) (\tilde{x}' \tilde{x} \tilde{\xi}^{-1} \tilde{x}^{-1}
    \tilde{x}'^{-1}) (\tilde{x}' \tilde{\xi}'^{-1} \tilde{x}'^{-1}),
  \end{align*}
  using the commutativity of~$\tilde{G}$ in the third
  parenthesis. Note that each of the four factors in this expression
  is a conjugate of an element in the subgroup~$\tilde{\Gamma}$. Since
  the latter is a maximal abelian subgroup of the nilpotent group~$H$,
  we know~\cite[Thm.~5.40]{Rotman1995} that it is a normal
  subgroup. Hence each conjugate and therefore~$[u, u']$ itself is
  contained in~$\Gamma$. By symmetry, one obtains that~$[u, u']$ is
  likewise contained in~$\tilde{G}$.
  Using~$\tilde{G} \cap \tilde{\Gamma}$, this establishes
  that~$[u, u'] \in T$.
\end{proof}

A key property of Heisenberg groups is that they come with a \emph{nice factorization} of their
elements, which is unique relative to a choice of section.

\begin{lemma}
  \label{lem:heis-uniq-decomp}
  In Heisenberg a group~$(H, \tilde{G}, \tilde{\Gamma})$ over~$T$, every element has a
  decomposition~$u = c\tilde{x}\tilde{\xi}$ with~$c \in T$
  and~$(\tilde{x}, \tilde{\xi}) \in \tilde{G} \times \tilde{\Gamma}$, which is unique
  if~$\tilde{x}, \tilde{\xi} \in s(P)$ for a fixed section~$s$ of the quotient map~$H \to H/T$.

  \medskip

  \noindent If~$s$ is normalized, we have~$u \in \tilde{G}$
  iff~$\tilde{\xi} = 1$ and~$u \in \tilde{\Gamma}$ iff~$\tilde{x} = 1$.
\end{lemma}
\begin{proof}
  Let~$u \in H$ be given. Since~$\tilde{G}$ and~$\tilde{\Gamma}$ generate~$H$, we can write the
  element as~$u=\tilde{x}_0\tilde{\xi}_0$
  with~$(\tilde{x}_0, \tilde{\xi}_0) \in \tilde{G} \times \tilde{\Gamma}$. Using the exact
  sequence~$E\colon T \hookrightarrow H \oset{\pi}{\twoheadrightarrow} H/T$, we
  set~$\tilde{x} = s\pi(\tilde{x}_0)$ and~$\tilde{\xi} = s\pi(\tilde{\xi}_0)$.  Then exactness
  of~$E$ yields~$\tilde{x}_0/\tilde{x}, \tilde{\xi}_0/\tilde{\xi} \in \ker(\pi) = T$ so
  that~$\tilde{x}_0 = c_x \tilde{x}$, $\tilde{\xi}_0 = c_\xi \tilde{\xi}$ for
  suitable~$c_x, c_{\xi} \in T$. Defining~$c = c_x c_{\xi}$, we
  obtain~$u = c \tilde{x} \tilde{\xi}$, which establishes the existence of a decomposition of the
  required form.

  Now let us show uniqueness. Assuming~$u = c' \tilde{x}' \tilde{\xi}'$ is another such
  de\-composition, we
  get~$(c/c') (\tilde{x}/\tilde{x}') = \tilde{\xi}'/\tilde{\xi} \in \tilde{G} \cap \tilde{\Gamma} =
  T$.
  Hence we may write~$\tilde{x}' = c_x \tilde{x}$ and~$\tilde{\xi}' = c_{\xi} \tilde{\xi}$ for
  suitable constants~$c_x, c_{\xi} \in T$. Since each
  of~$\tilde{x}, \tilde{x}_0, \tilde{\xi}, \tilde{\xi}_0$ must be in~$s(P)$, let us
  write~$\tilde{x} = s(z), \tilde{x}' = s(z')$
  and~$\tilde{\xi} = s(\zeta), \tilde{\xi}' = s(\zeta')$ for some~$z, z', \zeta, \zeta' \in P$.
  Applying~$\pi$ to the equalities~$\tilde{x}' = c_x \tilde{x}$,
  $\tilde{\xi}' = c_{\xi} \tilde{\xi}$ yields~$z' = z$, $\zeta' = \zeta$ by
  $c_x, c_{\xi} \in T = \ker(\pi)$. But then we have also~$\tilde{x}' = \tilde{x}$,
  $\tilde{\xi}' = \tilde{\xi}$, which forces~$c' = c$ and thus uniqueness.

  Assume~$s$ is normalized.
  If~$u = \tilde{c} \tilde{x} \tilde{\xi} \in \tilde{G}$, we
  get~$\tilde{\xi} = u (c\tilde{x})^{-1} \in \tilde{G}$. But
  then~$\tilde{\xi} \in \tilde{G} \cap \tilde{\Gamma} = T$, and we get
  a second decomposition~$u = (c\tilde{\xi}) \tilde{x}1$
  since~$1 \in s(P)$. By the uniqueness of decomposition, this
  yields~$\tilde{\xi} = 1$. Conversely, it is clear
  that~$u = c\tilde{x} \in \tilde{G}$. The statement on membership
  in~$\tilde{\Gamma}$ follows by symmetry.
\end{proof}

For turning Heisenberg groups into a category, it turns out to be more convenient not to fix the way
the torus~$T$ is embedded. In view of the equivalences~$\SES_2 \isomarrow \Cnt_2 \isomarrow \Nil_2$
mentioned before Fact~\ref{fct:char-cent-exseq}, our definition of the \emph{Heisenberg
  category}~$\Hei(\bullet)$ shall be based on~$\SES_2$ rather than~$\Nil_2$.

Thus let us
call~$E\colon T \oset{\iota}{\rightarrowtail} H \oset{\pi}{\twoheadrightarrow} P \, \in \SES_2$ a
\emph{Heisenberg extension} if~$H$ is a Heisenberg group over~$\hat{T} := \iota(T)$. The objects
of~$\Hei(\bullet)$ are then triples~$\mathcal{H} = (E, \tilde{G}, \tilde{\Gamma})$ such that~$E$ is
a Heisenberg extension with a choice of abelian bisection~$(\tilde{G}, \tilde{\Gamma})$ for~$H$.
If~$\mathcal{H}' = (E', \tilde{G}', \tilde{\Gamma}')$ is another object of~$\Hei(\bullet)$ based on
the Heisenberg
extension~$E'\colon T' \oset{\iota'}{\rightarrowtail} H' \oset{\pi'}{\twoheadrightarrow} P'$, a
morphism of~$\Hei(\bullet)$ from~$\mathcal{H}$ to~$\mathcal{H}'$ is a
morphism~$(t, h, p)\colon E \to E'$ such that~$h(\tilde{G}) \le \tilde{G}'$
and~$h(\tilde{\Gamma}) \le \tilde{\Gamma}'$.

The functor~$\Pi_{\Hei}\colon \Hei(\bullet) \to \Ab$
mapping~$(T \oset{\iota}{\rightarrowtail} H \oset{\pi}{\twoheadrightarrow} P, \tilde{G},
\tilde{\Gamma})$
to~$T$ makes~$\Hei(\bullet)$ into a \emph{category over~$\Ab$}. As usual, we denote the fibers
by~$\Hei(T) := \Pi_{\Hei}^{-1}(T)$. Given~$\mathcal{H}, \mathcal{H}' \in \Hei(T)$, we will also
apply the common notation~$\Hom_T(\mathcal{H}, \mathcal{H}') := \Pi_{\Hei}^{-1}(1_T)$ and its
generalization~$\Hom_t(\mathcal{H}, \mathcal{H}') := \Pi_{\Hei}^{-1}(t)$ for a group
homomorphism~$t\colon T \to T'$. Clearly, we have
$\Hom_T(\mathcal{H}, \mathcal{H}') = \Hom_{1_T}(\mathcal{H}, \mathcal{H}')$ for any~$T \in \Ab$.

For any Heisenberg
extension~$E\colon T \oset{\iota}{\rightarrowtail} H \oset{\pi}{\twoheadrightarrow} P$ we
have~$\hat{T} \le \zentrum(H)$ as we have seen before Lemma~\ref{lem:heis-com-tor}. If~$E$ is a
strictly central extension so that~$\hat{T} = \zentrum(H)$, we will call~$E$ a \emph{nondegenerate
  Heisenberg extension} (and~$H$ a nondegenerate Heisenberg group). The full subcategory of
such~$(E, \tilde{G}, \tilde{\Gamma})$ is denoted by~$\GenHei(\bullet)$, and its fiber over~$T$
by~$\GenHei(T)$.

\begin{proposition}
  \label{prop:hei-fibration}
  The category~$\Hei(\bullet)$ is fibered over~$\Ab$ and contains~$\GenHei(\bullet)$ as a fibered
  subcategory.
\end{proposition}
\begin{proof}
  Given any group homomorphism~$t\colon T \to T'$ in~$\Ab$ and an object~$\mathcal{H} \in \Hei(T)$,
  we must find a distinguished object~$t_*[\mathcal{H}] \in \Hei(T')$ and a distinguished
  morphism~$t_*\colon \mathcal{H} \to t_*[\mathcal{H}]$ over~$t$ such that the following universal
  property is satisfied~\cite[\S12.1.1]{BarrWells1995}: For any morphism
  $k \in \Hom_s(\mathcal{H}, \mathcal{K})$ over a group homomorphism~$s \colon T \to S$ and for any
  other group homomorphism~$s'\colon T' \to S$ such that~$s' \circ t = s$ there is a
  unique~$k' \in \Hom_{s'}(t_*[\mathcal{H}], \mathcal{K})$ with~$k' \circ t_* = k$. One may
  call~$t_*[\mathcal{H}]$ the \emph{direct image} of~$\mathcal{H}$ over~$t$, and the
  morphism~$t_*\colon \mathcal{H} \to t_*[\mathcal{H}]$ the corresponding \emph{extension of
    scalars}; this terminology is inspired by the well-known opfibration in the module category. (In
  category jargon, $t_*$ is usually called an ``opcartesian morphism''
  while~$\Pi_{\Hei}\colon \Hei(\bullet) \to \Ab$ is known as an ``opfibration''
  over~$\Ab$. Sometimes the alternative terms ``co-cartesian morphism'' and ``co-fibration'' are
  in use.)

  Given~$\mathcal{H} = (E, \tilde{G}, \tilde{\Gamma})$ with the Heisenberg
  extension~$E\colon T \oset{\iota}{\rightarrowtail} H \oset{\pi}{\twoheadrightarrow} P$
  and~$t\colon T \to T'$ in~$\Ab$, we define the direct image as~$t_*[H] := (T' \times H)/Z$
  with~$Z := \{ (tc, \iota c^{-1}) \mid c \in T\}$. We wrap this group into the central
  extension~$t_*[E]\colon T' \oset{\iota'}{\rightarrowtail} t_*[H] \oset{\pi'}{\twoheadrightarrow}
  P$,
  with homomorphisms~$\iota'(c') := (c', 1) \, Z$ and~$\pi'(c', u) \, Z := \pi(u)$. Note
  that~$t_*[E]$ coincides with the forward induced extension of~\cite[I.1]{BeylTappe2006}, for the
  special case of central extensions considered here.

  Let us now show that~$t_*[H]$ is a Heisenberg group over~$\hat{T}'$ with abelian
  bisection~$(\tilde{G}', \tilde{\Gamma}')$ given by $\tilde{G}' := (T' \times \tilde{G})/Z$ and
  $\tilde{\Gamma}' := (T' \times \tilde{\Gamma})/Z$. We split the work into the following tasks:
  \begin{enumerate}
  \item It is easy to see that~$\tilde{G}'$ and~$\tilde{\Gamma}'$ are maximal abelian
    in~$t_*[H]$. One way to show this is via the bijective correspondence between the subgroups
    of~$t_*[H]$ and those of~$T' \times H$ containing~$Z$; another method is to show
    $C_{t_*[H]}(\tilde{G}') = \tilde{G}'$ and~$C_{t_*[H]}(\tilde{\Gamma}') = \tilde{\Gamma}'$.
  \item We verify that~$\tilde{G}' \, \tilde{\Gamma}' = t_*[H]$. Indeed, for
    any~$(c', u) \, Z \in t_*[H]$ we can find~$c \in T$
    and~$(\tilde{x}, \tilde{\xi}) \in \tilde{G} \times \tilde{\Gamma}$
    with~$u=\iota(c) \, \tilde{x}\tilde{\xi}$, using Lemma~\ref{lem:heis-uniq-decomp} with any
    section of~$\pi\colon H \to P$. But then we
    obtain $(c', u) \, Z = (c', \tilde{x}) \, Z \cdot (tc, \tilde{\xi}) \, Z$, so the
    sub\-groups~$\tilde{G}'$ and~$\tilde{\Gamma}'$ do generate~$t_*[H]$.
  \item Finally, we must show~$\tilde{G}' \cap \tilde{\Gamma}' = \hat{T}'$, where the embedded torus
    is~$\hat{T}' := \iota'(T') = (T' \times \hat{T})/Z$ as one sees immediately. Inclusion from
    right to left being clear, assume~$(c', u) \, Z \in \tilde{G}' \cap \tilde{\Gamma}'$; we must
    show~$(c', u) \, Z \in (T' \times T)/Z$. The hypothesis means that
    $(c', u) \, Z = (d', \tilde{x}) \, Z$
    and~$(c', u) \, Z = (e', \tilde{\xi}) \, Z$
    for~$d', e' \in T'$
    and~$(\tilde{x}, \tilde{\xi}) \in \tilde{G} \times \tilde{\Gamma}$.
    Then~$(d'/e', \tilde{x}/\tilde{\xi}) \in Z$ implies
    $d' = t(c) \, e'$ and
    $\tilde{\xi} = \iota(c) \, \tilde{x} \in \tilde{G} \cap \tilde{\Gamma} = \hat{T}$ for
    some~$c \in T$.  Thus we obtain as claimed
    $(c', u) \, Z = (e', \tilde{\xi}) \, Z \in (T' \times \hat{T})/Z$.
  \end{enumerate}
  
  Thus we have~$t_*[\mathcal{H}] := (t_*[E], \tilde{G}', \tilde{\Gamma}') \in \Hei(T')$, and the
  next step is to exhibit a morphisms~$t_*\colon \mathcal{H} \to t_*[\mathcal{H}]$
  of~$\Hei(\bullet)$ with the requisite universal property. Following the convention of naming
  $\SES_2$ morphisms by their middle maps (confer Proposition~\ref{eq:ses2-nil2}), we
  define~$t_*\colon H \to t_*[H]$ by~$u \mapsto (1,u) \, Z$. It is immediate from the definition of
  the group~$t_*[H]$ that~$t_* \iota = \iota' t$, and~$\pi = \pi' \, t_*$ is also obvious. Hence we
  have
  \begin{equation*}
  \xymatrix @M=0.5pc @R=1pc @C=1.5pc%
  { & T \ar@{>->}[r]^{\iota} \ar[d]_t & H \ar@{->>}[r]^{\pi} \ar[d]_{t_*}
    & P, \ar@{=}[d]\\
    & T' \ar@{>->}[r]_{\iota'} & t_*[H] \ar@{->>}[r]_{\pi'} & P,}
  \end{equation*}
  which means~$t_* = (t, t_*, 1_P)$ is a morphism of~$\SES_2$. Since we also
  have $t_*(\tilde{G}) \le \tilde{G}'$ and~$t_*(\tilde{\Gamma}) \le \tilde{\Gamma}'$, we see
  that~$t_*\colon \mathcal{H} \to t_*[\mathcal{H}]$ is indeed a morphism of~$\Hei(\bullet)$. It
  remains to show the universal property.

  Thus let~$k \in \Hom_s(\mathcal{H}, \mathcal{K})$ be a morphism from~$\mathcal{H}$ to another
  object
  $\mathcal{K} = (S \oset{\kappa}{\rightarrowtail} K \oset{\rho}{\twoheadrightarrow} R, \tilde{L},
  \tilde{\Lambda})$
  over~$s \colon T \to S$, and let~$s'\colon T' \to S$ be some homomorphism in~$\Ab$ such
  that~$s' \circ t = s$. These hypotheses amount to the commutative diagram
  \begin{equation}
    \label{eq:fib-cd}
  \xymatrix @M=0.5pc @R=0pc @C=1.5pc%
  { & T \ar@{>->}[r]^{\iota} \ar[dl]_t \ar[dd]_s & H \ar@{->>}[r]^{\pi} \ar[dd]_k 
    & P \ar[dd]_{r},\\
    T' \ar[dr]_{s'}\\
    & S \ar@{>->}[r]_{\kappa} & K \ar@{->>}[r]_{\rho} & R,}
  \end{equation}
  still following the convention of identifying $\SES_2$ morphisms with their middle maps. Using the
  commutativity of the triangle and left square of~\eqref{eq:fib-cd}, we see that the
  map~$T' \times H \to K$ given by $(c', u) \mapsto \kappa(s'c') \, k(u)$ annihilates~$Z$, hence
  descends to a homomorphism $k'\colon t_*[\mathcal{H}] \to K$. One sees immediately
  that~$k' \iota' = \kappa s'$, while using the left commutative square of~\eqref{eq:fib-cd}
  establishes~$\rho k' = r \pi'$. Hence~$(s', k', r)$ is a $\SES_2$ morphism, which we write also
  as~$k'$ by our middle-map convention. For seeing that~$k'$ is in fact a morphism
  of~$\Hei(\bullet)$, it remains to ensure that~$k'(\tilde{G}') \le \tilde{L}$
  and~$k'(\tilde{\Gamma}') \le \tilde{\Lambda}$. For the former, take
  any~$(c', \tilde{x}) \, Z\in \tilde{G}'$. Then we have
  $k' \big( (c', \tilde{x}) \, Z \big) = \kappa(s'c') \, k(\tilde{x}) \in \tilde{L}$
  because~$\kappa(s'c') \in \hat{S} \le \tilde{L}$ and~$k(\tilde{x}) \in \tilde{L}$ from the fact
  that~$(\tilde{L}, \tilde{\Lambda})$ is an abelian bisection of~$K$. The argument
  for~$k'(\tilde{\Gamma}') \le \tilde{\Lambda}$ is analogous. We have thus established
  $k' \in \Hom_{s'}(t_*[\mathcal{H}], \mathcal{K})$ as a morphism of~$\Hei(\bullet)$, and it is easy
  to see that~$(s', k', r) \circ (t, t_*, 1_P) = (s, k, r)$.

  Having established existence of the morphism~$k'$, our last task is to prove its uniqueness. So
  assume~$k'' \in \Hom_{s'}(t_*[\mathcal{H}], \mathcal{K})$ is any morphism of~$\Hei(\bullet)$
  with~$k'' t_* = k$, and take~$\tilde{u} = (c', u) \, Z \in t_*[H]$. We compute
  now~$k''(\tilde{u})$ as
  \begin{equation*}
    k''\big( (c',1) \, Z \big) \cdot k''\big( (1,u) \, Z \big)
    = k''(\iota' c') \, k''(t_* u) = \kappa(s' c') \, k(u) = k'(\tilde{u}),
  \end{equation*}
  using the main hypothesis~$k'' t_* = k$ on~$k''$ and the fact that~$k'' \iota' = \kappa s'$,
  which follows from~$k''$ being a morphism over~$s'$.

  The statement about~$\GenHei(\bullet)$ is easy to verify: It is clear that~$\Pi_{\Hei}$ restricts
  to a functor~$\GenHei(\bullet) \to \Ab$. Since~$\GenHei(\bullet)$ is defined as a full
  subcategory of~$\Hei(\bullet)$, it suffices to show that~$t_*[\mathcal{H}] \in \GenHei(T')$
  whenever~$\mathcal{H} \in \GenHei(T)$ and~$t\colon T \to T'$ in~$\Ab$. An easy computation shows
  that~$\smash{\zentrum\big( (T' \times H)/Z \big)} = \smash{\big(T' \times \zentrum(H)\big)/Z}$.
  Hence~$\zentrum(H) = \smash{\hat{T}}$
  implies~$\smash{\zentrum\big( (T' \times H)/Z \big)} = (T' \times \hat{T})/Z = \hat{T}'$, as
  required.
\end{proof}

Note that the
construction~$(t, \mathcal{H}) \mapsto \big(t_*\colon \mathcal{H} \to t_*[\mathcal{H}]\big)$ of
Proposition~\ref{prop:hei-fibration} provides a \emph{cleavage but not a
  splitting}~\cite[12.1.3]{BarrWells1995} of the fibered
category~$\Pi_{\Hei}\colon \Hei(\bullet) \to \Ab$ since~$(1_T)_*\colon H \isomarrow (T \times H)/Z$
is only an isomorphism, not the identity~$1_H$. Similarly, the composition of scalar extensions
produces a codomain that is only isomorphic but not equal to the codomain produced by the scalar
extension of the composite. Nevertheless, any fibered category yields a family of \emph{pushforward
  functors}~\cite[12.1.8]{BarrWells1995} by the universal property mentioned in the proof of
Proposition~\ref{prop:hei-fibration}: In our case, each morphism~$t\colon T \to T'$ gives rise to
the functor~$\Hei(t)\colon \Hei(T) \to \Hei(T')$ defined on objects
by~$\mathcal{H} \mapsto t_*[\mathcal{H}]$ and on
morphisms~$h \in \Hom_T(\mathcal{H}_1, \mathcal{H}_2)$ as the unique $\Hei(T')$
morphism~$h' \in \Hom_{T'}(t_*[\mathcal{H}_1], t_*[\mathcal{H}_2])$ that makes the diagram
\begin{equation*}
  \xymatrix @M=0.5pc @R=1.5pc @C=1.5pc%
  { \mathcal{H}_1 \ar[d]_h \ar[r]^{t_*} & t_*[\mathcal{H}_1] \ar[d]^{h'}\\
  \mathcal{H}_2 \ar[r]_{t_*} & t_*[\mathcal{H}_2] }  
\end{equation*}
commute. However, since the underlying cleavage is not a splitting in our case, the
assignment~$\Hei\colon \Ab \to \Cat$ is \emph{not} a functor, just a morphism of graphs~\cite[\S
II.7]{MacLane1998}. Of course, similar remarks hold for the fibered
subcategory~$\GenHei\colon \Ab \to \Cat$ of nondegenerate Heisenberg objects. Note also
that~$T_0 := (T = T \twoheadrightarrow 0, T, T)$ is an initial object
of~$\GenHei(T) \subset \Hei(T)$ for each~$T \in \Ab$, with the assignment~$T \mapsto T_0$ defining
a cartesian section.

The next step in our algebraic investigation of Heisenberg groups is to provide a way of
constructing them in terms of \emph{bilinear forms}~$\beta\colon G \times \Gamma \to T$
with~$G, \Gamma, T \in \Ab$. Such a bilinear form naturally induces the left right curried
homomorphism~$G \to \hat{\Gamma} := \Hom(\Gamma, T)$ and~$\Gamma \to \hat{G} := \Hom(G, T)$,
respectively. Their kernels~$G_0$ and~$\Gamma_0$ are accordingly called the \emph{left kernel} and
the \emph{right kernel} of~$\beta$. The bilinear form~$\beta$ is called \emph{nondegenerate} if the
unilateral kernels are trivial and \emph{proper} if they are complemented (there are
subgroups~$G_1 \le G$ and~$\Gamma_1 \le \Gamma$ such that~$G = G_0 \dotplus G_1$
and~$\Gamma = \Gamma_0 \dotplus \Gamma_1$).

Bilinear forms may be collected into the \emph{bilinear category}~$\Bil(\bullet)$, with the
following morphisms: If~$\beta\colon G \times \Gamma \to T$
and~$\beta'\colon G' \times \Gamma' \to T'$ are bilinear forms, and if~$g\colon G \to G'$,
$\gamma\colon \Gamma \to \Gamma'$ and~$t\colon T \to T'$ are group homomorphisms, we
have~$(g \times \gamma, t)\colon \beta \to \beta'$
whenever~$\beta' \circ (g \times \gamma) = \beta \circ t$ and~$g(G_0) \le G_0'$,
$\gamma(\Gamma_0) \le \Gamma_0'$. In other words, $\Bil(\bullet)$ is the subcategory of the comma
category~$\times \downarrow 1$ that respects the kernels (where we view the direct product of
abelian groups as a functor~$\times\colon \Ab^2 \to \Ab$ and the identity as a
functor~$1\colon \Ab \to \Ab$). The full subcategories of~$\Bil(\bullet)$ consisting of
\emph{nondegenerate forms} and \emph{proper forms} are, respectively, denoted by~$\GenBil(\bullet)$
and~$\PropBil(\bullet)$. We have the obvious
inclusions $\GenBil(\bullet) \subset \PropBil(\bullet) \subset \Bil(\bullet)$.

Like~$\Hei(\bullet)$, the bilinear category is also fibered over~$\Ab$, but this time we have in
fact a \emph{split fibration}. Hence let~$\Pi_{\Bil}\colon \Bil(\bullet) \to \Ab$ be the obvious
projection functor mapping an object~$\beta\colon G \times \Gamma \to T$ of~$\Bil(\bullet)$
to~$T \in \Ab$ and a morphism~$(g \times \gamma, t)$ of~$\Bil(\bullet)$ to its base
morphism~$t \in \Ab(T, T')$. The pushforward~$t_*[\beta]$ is the bilinear form
$t\beta\colon G \times \Gamma \to T'$, with the extension of
scalars~$t_*\colon \beta \to t_*[\beta]$ given by~$(1_G \times 1_\gamma, t)$. It is easy to check
that the opfibration axioms~\cite[\S12.1.1]{BarrWells1995} are satisfied, together with the
splitting axioms~\cite[\S12.1.3]{BarrWells1995} for the cleavage~$(\beta, t) \mapsto t_*$.
Thus~$T \mapsto \Bil(T)$ is a genuine functor~$\Ab \to \Cat$ by~\cite[Prop.~12.1.8]{BarrWells1995},
and we write~$\Bil(T) := \Pi_{\Bil}^{-1}(T)$ for its fibers.

It will be noted, however, that~$\GenBil(\bullet)$ is \emph{not} fibered since pushing forward a
bilinear form under~$t=0$ clearly renders it degenerate (but restricting to monomorphisms in the
base category~$\Ab$ of~$\Bil(\bullet)$ does lead to a fibered subcategory of nondegenerate bilinear
forms). In the same fashion, we cannot expect~$\PropBil(\bullet)$ to be fibered, as the
example~$\beta\colon \ZZ \times \ZZ \to \ZZ$ with~$\beta(m,n) := mn$ the usual product
and~$t\colon \ZZ \twoheadrightarrow \ZZ/p\ZZ$ the canonical map shows, where~$G_0 = \Gamma_0 = p\ZZ$
obviously have no complement.

We can now give the promised construction of the Heisenberg group of a
bilinear form. So let~$\beta\colon G \times \Gamma \to T$ be a
bilinear form. We will show that the group~$H := TG \rtimes \Gamma$ is
a \emph{Heisenberg group} over the \emph{expanded
  torus}~$\Texp := TG_0 \times \Gamma_0$, with an abelian bisection given
by~$(TG \times \Gamma_0, TG_0 \times \Gamma)$.
Writing~$P := G/G_0 \times \Gamma/\Gamma_0$ for the prospective phase
space, we therefore set
\begin{equation}
  \label{eq:heigrp-functor-obj}
  \HeiGrp(\beta) := (\Texp \oset{\iota}{\rightarrowtail} H \oset{\pi}{\twoheadrightarrow} P, TG \times \Gamma_0, TG_0 \times \Gamma),
\end{equation}
where~$\iota$ is the inclusion and~$\pi$ the
map~$c \, (x, \xi) \mapsto (xG_0, \xi\Gamma_0)$.  Furthermore,
if~$(g \times \gamma, t)$ is a $\Bil(\bullet)$ morphism
from~$\beta\colon G \times \Gamma \to T$ to another bilinear
form~$\beta'\colon G' \times \Gamma' \to T'$, we define
\begin{equation}
  \label{eq:heigrp-functor-mor}
  \HeiGrp(g \times \gamma, t) := (\exptor{t}, h, p)
\end{equation}
whose components are the group homomorphisms
\begin{equation*}
  \begin{array}{rclrcl}
    H & \stackrel{h}{\to} & H', & P & \stackrel{p}{\to} & P',\\
            c \, (x, \xi) & \mapsto & tc \, (gx, \gamma\xi), \quad & (xG_0, \xi\Gamma_0) & \mapsto & \big( (gx)G_0', (\gamma\xi)\Gamma_0 \big).
  \end{array}
\end{equation*}
and the torus map~$\exptor{t}\colon \Texp \to \Tpexp$ induced by
restriction of~$h\colon H \to H'$. Let us now verify that~$\HeiGrp$ is
a functor.

\begin{proposition}
  \label{prop:heigrp-construction}
  We have a functor~$\HeiGrp\colon \Bil(\bullet) \to \Hei(\bullet)$
  given by Equations~\eqref{eq:heigrp-functor-obj} and
  \eqref{eq:heigrp-functor-mor}; its image is contained
  in~$\GenHei(\bullet)$.
\end{proposition}
\begin{proof}
  Let us first check that~$\HeiGrp(\beta)$ is an object
  of~$\Hei(\bullet)$. It is evident
  that~$E\colon \Texp \oset{\iota}{\rightarrowtail} H
  \oset{\pi}{\twoheadrightarrow} P$ is exact, hence~$H$ is
  nilquadratic (Proposition~\ref{eq:ses2-nil2}). By a short
  computation one obtains the associated commutator form
  \begin{equation*}
    \omega_E(xG_0, \xi\Gamma_0; yG_0, \eta\Gamma_0) = \beta(\xi,
    y)/\beta(\eta, x) .
  \end{equation*}
  It is easy to check that~$G/G_0 \times 0$
  and~$0 \times \Gamma/\Gamma_0$ are Lagrangian subgroups of~$P$ with
  respect to~$\omega_E$. Since they obviously form a direct
  decomposition of~$P$, we obtain a Lagrangian
  bisection~$(G/G_0 \times 0, 0 \times \Gamma/\Gamma_0)$ of~$P$. It is
  also straightforward to check
  that~$\pi^{-1}(G/G_0 \times 0) = TG \times \Gamma_0$ and
  $\pi^{-1}(0 \times \Gamma/\Gamma_0) = TG_0 \times \Gamma$, hence we
  conclude by Item~\ref{it:ab-split} of Theorem~\ref{thm:sympl-corr}
  that~$(TG \times \Gamma_0, TG_0 \times \Gamma)$ is indeed an abelian
  bisection of~$H$ over the expanded torus~$\Texp$. It is also easy to
  see that~$\zentrum(H) = \Texp$, so~$H$ is nondegenerate as claimed.

  We turn now to the morphisms. Assuming~$(g \times \gamma, t)$ is a
  morphism of~$\Bil(\bullet)$, we must verify
  that~\eqref{eq:heigrp-functor-mor} yields a
  morphism~$(\exptor{t}, h, p)$ of~$\Hei(\bullet)$. First of all, it
  is clear that~$h\colon H \to H'$ does restrict
  to~$\exptor{t}\colon \Texp \to \Tpexp$ since we
  have~$g(G_0) \le G_0'$ and~$\gamma(\Gamma_0) \le \Gamma_0'$ from the
  fact that $\Bil(\bullet)$ morphisms respect the left and right
  kernels (this is also the reason why~$p\colon P \to P'$ is
  well-defined). This takes care of the left commutative square needed
  for showing~$(\exptor{t}, h, p)$ to be a $\SES_2$ morphism; the
  commutativity of the right square follows immediately from the
  definitions of the maps involved. It remains to show that~$h$
  respects the abelian bisections,
  meaning~$h(TG \times \Gamma_0) \le T'G' \times \Gamma_0'$
  and~$h(TG_0 \times \Gamma) \le T'G_0' \times \Gamma'$; this follows
  again directly from the fact that the morphisms of $\Bil(\bullet)$
  respect the left and right kernels.

  Once it is clear that~$\HeiGrp$ is well-defined on morphisms, it is
  trivial to check that it is indeed a
  functor~$\Bil(\bullet) \to \Hei(\bullet)$ since identities and
  compositions are obviously preserved by~$\HeiGrp$.
\end{proof}



The way back from a Heisenberg group to ``its'' bilinear form is
easier in the sense that we need not modify the torus. Hence we can
describe this \emph{reverse construction} as a functor within the same
fibers.

\begin{proposition}
  \label{prop:bilform-functor}
  Fix~$T \in \Ab$. The function~$\BilFrm\colon \Hei(T) \to \Bil(T)$
  mapping~$(H, \tilde{G}, \tilde{\Gamma})$ to the commutator
  form~$\tilde{G}/T \times \tilde{\Gamma}/T \to T$ is a functor with
  image contained in~$\GenBil(T)$.
\end{proposition}
\begin{proof}
  We have seen that any Heisenberg
  group~$(H, \tilde{G}, \tilde{\Gamma}) \in \Hei(T)$ induces a central
  extension~$E$ as in~(\ref{eq:cent-extn-of-heis}). By
  Proposition~\ref{prop:symp-dual}, its commutator
  form~$\omega_E\colon P \times P \to T$ on the phase space~$P := H/T$
  is then bilinear. Its restriction~$\beta$ to the subgroups
  $\tilde{G}/T$ and~$\tilde{\Gamma}/T$ of~$P$ is bilinear as well,
  hence~$(\beta, \tilde{G}/T, \tilde{\Gamma}/T) \in \Bil(T)$.

  Letting~$(H', \tilde{G}', \tilde{\Gamma}')$ be another Heisenberg
  group over~$T$, consider a Heisenberg morphisms~$h\colon H \to
  H'$. We define~$g\colon \tilde{G}/T \to \tilde{G}'/T$ and
  $\gamma\colon \tilde{\Gamma}/T \to \tilde{\Gamma}'/T$ as the
  homomorphisms induced by~$h$ on the corresponding subgroups; these
  are well-defined since~$h$ respects the abelian bisections and fixes
  the torus~$T$. For checking
  that~$\beta' \circ (g \times \gamma) = \beta$, we compute
  \begin{align*}
    \beta' \circ (g \times \gamma) (\tilde{x}T, \tilde{\xi}T)
    &= \beta'\big( h(\tilde{x})T, h(\tilde{\xi})T \big)
    = [h(\tilde{x}), h(\tilde{\xi})]\\
    &= h \, [\tilde{x}, \tilde{\xi}]
    = [\tilde{x}, \tilde{\xi}] = \beta\big( \tilde{x}T, \tilde{\xi}T \big),
  \end{align*}
  using the fact that commutators are fixed by~$h$
  since~$[H, H] \le T$ by centrality of~$E$.

  Let us now check that~$\beta$ is nondegenerate. Using the fact
  that~$\tilde{\Gamma}$ is maximal abelian, it is easy to compute the
  left kernel of~$\beta$ as
  \begin{equation*}
    (\tilde{G}/T)_0 = \frac{\tilde{G} \cap C_H(\tilde{\Gamma})}{T} = \frac{\tilde{G} \cap
      \tilde{\Gamma}}{T} = T/T = 0 \in \tilde{G}/T .
  \end{equation*}
  By the symmetry of the argument, the right
  kernel~$(\tilde{\Gamma}/T)_0$ is likewise trivial, so~$\beta$ is
  indeed nondegenerate. Hence~$\BilFrm(h) := (g, \gamma)$
  automatically respects the one-sided kernels of~$\beta$ and~$\beta'$
  and thus is a morphism in~$\Bil(T)$. Since it is obvious that
  unitality/composition of morphisms in~$\Hei(T)$ corresponds to
  unitality/composition in~$\Bil(T)$, this completes the proof that
  $\BilFrm\colon \Hei(T) \to \Bil(T)$ is a functor.
\end{proof}

The functor~$\BilFrm\colon \Hei(T) \colon \Bil(T)$ of
Proposition~\ref{prop:bilform-functor} will also be written
as~$\BilFrm_T$ if the \emph{dependence on the torus} needs to be made
explicit. Likewise, we shall write~$\HeiGrp_T$ for the construction of
the Heisenberg group from a given bilinear form in~$\Bil(T)$ in
Proposition~\ref{prop:heigrp-construction}. For making this
``construction'' into a proper functor, we have to view its codmain as
a category comprised of all Heisenberg groups over arbitrary tori. The
key to realize this is once again a fibration whose model is that of
the module category.

In view of Remark~\ref{rem:toroidal-as-linear}, we expect a way of
\emph{restricting and extending scalars}, just as for modules, where a
ring homomorphism~$f\colon R \to S$ induces an extension
functor~$f_*\colon \Mod_R \to \Mod_S$ with~$M \mapsto M \otimes_R S$
and a restriction functor~$f^*\colon \Mod_S \to \Mod_R$ with the
action on~$f^*(N)$ defined by~$\lambda \cdot x = f(\lambda) \, x$
for~$x \in N$ and~$\lambda \in R$, such that extension is left adjoint
to restriction~\cite[Exc.~XV.5]{MacLaneBirkhoff1999}. For our present
setting we shall only need the extension functor, which may be derived
as a special case of the so-called Calculus of Induced
Extensions~\cite[I.1]{BeylTappe2006}.

\begin{theorem}
  \label{thm:heisenberg-adjunction}
  Defining the pair of
  functors~$\Phi\colon \Bil(\bullet) \to \Hei(\bullet)$ and
  $\Psi\colon \Hei(\bullet) \to \Bil(\bullet)$ by
  \begin{align*}
    \Phi(\beta, G, \Gamma, T) &= (TG \rtimes \Gamma, TG \times \Gamma_0, TG_0 \times \Gamma, TG_0
                                \times \Gamma_0),\\
    \Psi(H, \tilde{G}, \tilde{\Gamma}, T) &= (\omega|_{\tilde{G}/T \times \tilde{\Gamma}/T},
                                            \tilde{G}/T, \tilde{\Gamma}/T, T),
  \end{align*}
  we obtain an adjunction~$\Phi \dashv \Psi$.
\end{theorem}
\begin{proof}
  Next we check that~$\Phi$ is a functor. Hence assume~$g\colon G \to G'$ and
  $\gamma\colon \Gamma \to \Gamma'$ are homomorphisms with~$\beta' \circ (g \times \gamma) = \beta$
  and~$g(G_0) \le G_0'$, $\gamma(\Gamma_0) \le \Gamma_0'$. Then we define~$\phi := \Phi(g, \gamma)$
  to be the map~$c \, (x, \xi) \mapsto c \, (gx, \gamma\xi)$. We must first check that~$\phi$ fixes
  the torus~$T$.

  We need for all~$\mathcal{H} = (H, \tilde{G}, \tilde{\Gamma}) \in \Hei(T)$ an
  arrow~$\alpha_{\mathcal{H}}\colon \mathcal{H} \to \Phi\Psi(\mathcal{H})$ universal
  from~$\mathcal{H}$ to~$\Phi$. First of all, we need a group homomorphism
  $H \to T (\tilde{G}/T) \rtimes (\tilde{\Gamma}/T)$. Fixing an arbitrary set-theoretic section~$s$
  of the projection~$\pi\colon H \to P := H/T$, we employ Lemma~\ref{lem:heis-uniq-decomp} in
  defining the required map by~$\alpha(c\tilde{x}\tilde{\xi}) = c \, (\tilde{x}T, \tilde{\xi}T)$
  for~$c \in T$ and~$(\tilde{x}, \tilde{\xi}) \in \tilde{G} \times \tilde{\Gamma}$ such
  that~$\tilde{x}, \tilde{\xi} \in s(P)$. For checking that~$\alpha$ is a homomorphism, note that
  $c\tilde{x}\tilde{\xi} \cdot c'\tilde{x}'\tilde{\xi}' =
  (cc'\tilde{c})(\tilde{x}\tilde{x}')(\tilde{\xi}\tilde{\xi}')$
  where~$\tilde{c} := [\tilde{\xi}, \tilde{x}'] \in [H, H] \le T$ by
  Fact~\ref{fct:char-cent-exseq}. Applying~$\alpha$ thus
  yields~$(cc'\tilde{c}) \, \big( (\tilde{x}\tilde{x}')T, (\tilde{\xi}\tilde{\xi}')T \big) \in T
  (\tilde{G}/T) \rtimes (\tilde{\Gamma}/T)$,
  which equals the product of~$\alpha(c\tilde{x}\tilde{\xi}) = c \, (\tilde{x}T, \tilde{\xi}T)$
  and~$\alpha(c'\tilde{x}'\tilde{\xi}') = c' \, (\tilde{x}'T, \tilde{\xi}'T)$ by the definition of
  multiplication in~$T (\tilde{G}/T) \rtimes (\tilde{\Gamma}/T)$ via the bilinear
  form~$\inner{}{}\colon (\tilde{G}/T) \times (\tilde{\Gamma}/T) \to T$ given
  by~$\inner{\tilde{\xi}T}{\tilde{x}'T} = [\tilde{\xi}, \tilde{x}'] = \tilde{c}$.

  Next we must check that~$\alpha_{\mathcal{H}}$ is consistent with the abelian bisections. In other
  words, we need to ensure
  \begin{equation}
    \label{eq:resp-ab-split}
    \alpha(\tilde{G}) \le T (\tilde{G}/T) \times (\tilde{\Gamma}/T)_0
    \quad\text{and}\quad
    \alpha(\tilde{\Gamma}) \le T (\tilde{G}/T)_0 \times (\tilde{\Gamma}/T) .
  \end{equation}
  Since the left and right kernels vanish, this means the elements of~$\alpha(\tilde{G})$
  and~$\alpha(\tilde{\Gamma})$, respectively, must have the form~$c \, (\tilde{x}T,0)$
  and~$c \, (0,\tilde{\xi}T)$, with $c \in T$
  and~$\tilde{x} \in \tilde{G}, \tilde{\xi} \in \tilde{\Gamma}$, which follows immediately from
  the fact that
  \begin{align*}
    \tilde{G} &= \big\{ c\tilde{x} \mid c \in \hat{T}, \, \tilde{x} \in \tilde{G} \cap s(P) \big\},\\
    \tilde{\Gamma} &= \big\{ c\tilde{\xi} \mid c \in \hat{T}, \, \tilde{\xi} \in \tilde{\Gamma} \cap
                     s(P) \big\} .
  \end{align*}
  So here the inclusions~\eqref{eq:resp-ab-split} are, in fact, equalities.
\end{proof}

For our treatment of Heisenberg groups and the Fourier transform, we start with an abelian group~$P$, along with its commutator
form~$\omega_E \in \Hom_{\ZZ}(\Lambda^2 P, T)$. Before we come to this we provide the theory within a more general setting.

\section{Heisenberg Groups with nonabelian Phase Groups}
We will now provide the general concepts for abelian torus group $T$ but with possibly nonabelian groups $G,\Gamma$.
Therefore we will prefer multiplicative notation in all three groups during this exposition.
We start with a short repetition of the basic concepts in the general environment.

\vspace{2mm}
If $X$ is an arbitrary set, we write $S_X$ for the symmetric group on $X$.

\subsection{Bilinear forms with values in an abelian group}
We fix a set-map $\beta\colon\Gamma\times G\to T$ between groups $(\Gamma,\cdot),\,(G,\cdot)$ and the abelian group $(T,\cdot)$--the {\tt torus}. We write $\dual{\xi}{x}:=\beta(\xi,x)$.
\begin{definition}\label{def:left-kernel-right-kernel}\
\begin{enumerate}
\item $_0\beta:=\Gamma_0:=\{\xi\in\Gamma\mid\forall_{x\in G}\,\dual{\xi}{x}=1\}$
\item $\beta_0:=G_0:=\{x\in G\mid\forall_{\xi\in\Gamma}\,\dual{\xi}{x}=1\}$
\item $X\subseteq \Gamma\Imp X^\perp:=\{g\in G\mid\dual{X}{g}=1\}$
\item $Y\subseteq G\Imp Y^\perp:=\{\xi\in \Gamma\mid\dual{\xi}{Y}=1\}$
\end{enumerate}
\end{definition}
$\perp$ denotes the $\beta$-orthogonal.
Derived from $\beta\colon\Gamma\times G\to T$ there are the phase maps:
\begin{definition}
\begin{eqnarray}\label{bilinear-ops}
\lhd_\beta\colon\Gamma\to(T\times G)^{T\times G},\ \lhd_\beta(\xi)(c,x)=(c\dual{\xi}{x},x)\cr
\rhd_\beta\colon G\to(T\times\Gamma)^{T\times\Gamma},\ \rhd_\beta(x)(c,\xi)=(c\dual{\xi}{x},\xi).
\end{eqnarray}
Their invariants are
\begin{eqnarray*}
{\rm Inv}(\lhd_\beta)&:=&\{(c,x)\in T\times G\mid\forall_{\xi\in\Gamma}\,\lhd_\beta(\xi)(c,x)=(c,x)\}\cr
{\rm Inv}(\rhd_\beta)&:=&\{(c,\xi)\in T\times\Gamma\mid\forall_{x\in G}\,\rhd_\beta(x)(c,\xi)=(c,\xi)\}.
\end{eqnarray*}
\end{definition}
We consider $(T\times G)^{T\times G}$ and $(T\times\Gamma)^{T\times\Gamma}$ as monoids with composition of maps.

\begin{definition} Given a map $\beta\colon\Gamma\times G\to T$. We call $\beta$
\begin{eqnarray*}
{\tt left-linear}&\iff&\forall_{\xi,\eta}\forall_x\,\dual{\xi\eta}{x}=\dual{\xi}{x}\dual{\eta}{x}\cr
{\tt right-linear}&\iff&\forall_\xi\forall_{x,y}\,\dual{\xi}{xy}=\dual{\xi}{x}\dual{\xi}{y}\cr
{\tt bilinear}&\iff&\beta\text{ is left-and right-linear}\cr
{\tt duality}&\iff&\beta\text{ is bilinear and }G_0=\Gamma_0=0.
\end{eqnarray*}
\end{definition}

\begin{proposition} Given $\beta\colon\Gamma\times G\to T$. Then
\begin{eqnarray*}
&&{\rm Inv}(\lhd_\beta)=T\times\beta_0\text{ and }{\rm Inv}(\rhd_\beta)=T\times{}_0\beta\cr
&&\lhd_\beta\colon\Gamma\to S_{T\times G}\text{ and }\rhd_\beta\colon G\to S_{T\times\Gamma}\cr
&&\beta\ {\tt left-linear}\iff\lhd_\beta\in\hom\left(\Gamma,S_{T\times G}\right)\iff G\buildrel\rhd_\beta\over\to{\rm Aut}(T\times\Gamma)\cr
&&\beta\ {\tt right-linear}\iff\Gamma\buildrel\lhd_\beta\over\to{\rm Aut}(T\times G)\iff\rhd_\beta\in\hom\left(G,S_{T\times\Gamma}\right)\cr
&&\beta\ {\tt bilinear}\iff\lhd_\beta\in\hom\big(\Gamma,{\rm Aut}(T\times G)\big)\iff\rhd_\beta\in\hom\big(G,{\rm Aut}(T\times\Gamma)\big)\cr
&&\beta\ {\tt left-linear}\Imp\Gamma_0=\ker(\lhd_\beta)\text{ and }\beta(\bullet,x)\in\Gamma^\star\ \forall x\in G\cr
&&\beta\ {\tt right-linear}\Imp G_0=\ker(\rhd_\beta)\text{ and }\beta(\xi,\bullet)\in G^\star\ \forall\xi\in\Gamma\cr
&&\beta\ {\tt bilinear}\Imp\beta_1\colon\Gamma\to G^\star,\,\xi\mapsto\beta(\xi,\bullet)\text{ and }\beta_2\colon G\to\Gamma^\star,\,x\mapsto\beta(\bullet,x)\cr
&&\hspace{3cm}\text{ are group homomorphisms}.
\end{eqnarray*}
\end{proposition}

Thus, for bilinear $\beta$ the phase maps (\ref{bilinear-ops}) define group actions
\begin{eqnarray}\label{bilinear-actions}
\Gamma\times(T\times G)\buildrel\bullet\over\longrightarrow T\times G&\text{and}&G\times(T\times\Gamma)\buildrel\bullet\over\longrightarrow T\times\Gamma\cr
\xi\bullet(c,x)=(c\dual{\xi}{x},x)&&x\bullet(c,\xi)=(c\dual{\xi}{x},\xi).
\end{eqnarray}
A map $\beta\colon\Gamma\times G\to T$ comes together with its `dual map' $\beta^0\colon G\times \Gamma\to T$, connected by the diagram
\begin{equation*}
\xymatrix@R=2mm@C=2mm{
(\xi,x)\ar@{<->}[dd]&\Gamma\times G\ar@{<->}[dd]\ar[dr]^-\beta& \cr
&&T\cr
(x,\xi)&G\times\Gamma\ar[ur]_-{\beta^0}
}\end{equation*}
i.e., $\dual{x}{\xi}^0:=\beta^0(x,\xi)=\beta(\xi,x)=\dual{\xi}{x}$. Obviously, $\beta^0$ is left-linear iff $\beta$ is right-linear, thus $\beta^0$ is bilinear iff $\beta$ is so.
Supplying the superscript `0' to all items related to $\beta^0$ we obtain:
\begin{equation*}\xymatrix@R=1mm{
\Gamma\ar[r]^-{\lhd_\beta}&(T\times G)^{T\times G} & G\ar[r]^-{\lhd_\beta^0}&(T\times\Gamma)^{T\times\Gamma}\cr
G\ar[r]^-{\rhd_\beta}&(T\times\Gamma)^{T\times\Gamma} & \Gamma\ar[r]^-{\rhd_\beta^0}&(T\times G)^{T\times G}
}\end{equation*}
\begin{eqnarray*}
&&\rhd^0_\beta(\xi)(c,x)=\big(c\dual{x}{\xi}^0,x\big)=\big(c\dual{\xi}{x},x\big)=\,\lhd_\xi(c,x)\cr
&&\lhd^0_\beta(x)(c,\xi)=\big(c\dual{x}{\xi}^0,\xi\big)=\big(c\dual{\xi}{x},\xi\big)=\,\rhd_x(c,\xi)
\end{eqnarray*}
i.e., $\rhd_\beta^0\,=\,\lhd_\beta$ and $\lhd_\beta^0\,=\,\rhd_\beta$.
Therefore it is enough to pay attention to the action $\lhd_\beta$.

Let $\Gamma\,,G$ be groups, $T$ an abelian group, and $\beta\colon\Gamma\times G\to T$ bilinear.
Evidently then, $X^\perp$ and $Y^\perp$ are groups, that means, the $\beta$-orthogonal takes values in the lattice of subgroups and the 
composition $\perp\circ\perp$, being idempotent, monotone and increasing is a hull operator.

\subsection{The Heisenberg Group of a Bilinear Form}
\begin{definition}\label{def:heis-grp}
Let $(\Gamma,\cdot),(G,\cdot)\in\Grp$, $(T,\cdot)\in\Ab$ and $\beta\colon\Gamma\times G\lr T$ bilinear.
The \textbf{Heisenberg group induced by $\beta$} is the semi-direct product wrt. the (left) group action $\Gamma\times(T\times G)\brbullet T\times G$ resulting from
$\lhd_\beta\,\colon\Gamma\lr{\rm Aut}(T\times G)$ (cf. \ref{bilinear-actions})
We write $H(\beta)=(T\times G)\rtimes\Gamma$.

\vspace{2mm}
We call a group $X$ a {\tt Heisenberg group} iff
\begin{equation*}
\exists_{\Gamma,G\in\Grp}\,\exists_{T\in\Ab}\,\exists_{\beta\colon\Gamma\times G\lr T\text{ bilinear }}X\cong H(\beta).
\end{equation*}
$X$ is a {\tt proper Heisenberg group}
\begin{equation*}
\exists_{\Gamma,G,T\in\Ab}\,\exists_{\beta\colon\Gamma\times G\lr T\text{ duality }}X\cong H(\beta).
\end{equation*}
\end{definition}

\noindent Let $c,d\in T$, $x,y\in G$, $\xi,\eta\in\Gamma$. The basic computation rules are:
\begin{eqnarray*}
(c,x,\xi)\cdot(d,y,\eta)&=&\big((c,x)(\xi\cdot(d,y)),\xi\eta\big)=\big((c,x)(d\dual{\xi}{y},y),\xi\eta\big)\cr
&=&(cd\dual{\xi}{y},xy,\xi\eta)
\end{eqnarray*}
\noindent Then $1=(1,1,1)$ and $(c,x,\xi)^{-1}=\big(\frac{\dual{\xi}{x}}{c},x^{-1},\xi^{-1}\big)$.

\vspace{2mm}
There are embeddings $t\colon T\lr H(\beta)$, $g\colon G\lr H(\beta)$, $\gamma\colon\Gamma\lr H(\beta)$. We abbreviate their values by setting $c=(c,1,1)$, $x=(1,x,1)$, $\xi=(1,1,\xi)$.
\begin{equation*}\xymatrix{
       & T\ar[d]\ar@{.>}[dr]^-t\cr
1\ar[r]& T\times G\ar[r]^-i & H(\beta)\ar[r]^-\pi & \Gamma\ar[r]\ar@/^1.1pc/@{.>}[l]^-\gamma  & 0\cr
       & G\ar[u]\ar@{.>}[ur]_-g
}\end{equation*}
Then the following identities hold in $H(\beta)$:
\begin{eqnarray*}
cx&=&(c,x,1)=xc\cr
c\xi&=&(c,1,\xi)=\xi c\cr
x\xi&=&(1,x,\xi)\text{ and }\xi x=(\dual{\xi}{x},x,\xi)=\dual{\xi}{x}x\xi\cr
cx\xi&=&(c,x,\xi)
\end{eqnarray*}
Thus, $T$ is central in $H(\beta)$ and each element $h\in H(\beta)$ has a unique factorization
\begin{equation*}
h=cx\xi\ (c\in T,\,x\in G,\,\xi\in\Gamma).
\end{equation*}
\begin{proposition} Commutator and center in $H(\beta)$ are
\begin{eqnarray}\label{beta-commutator}
&&[(c,x,\xi),(d,y,\eta)]=\left(\frac{\dual{\xi}{y}}{\dual{\eta}{x}},[x,y],[\xi,\eta]\right)\cr
&&Z(H(\beta))=T\times\big(G_0\cap Z(G)\big)\times\big(\Gamma_0\cap Z(\Gamma)\big)
\end{eqnarray}
Consequently $\dual{\xi}{x}=[\xi,x]$ ($\xi\in\Gamma,\,x\in G$).
\end{proposition}
Note that the center of $H(\beta)$ is indeed a direct product of abelian groups
$Z(H(\beta))=T\oplus\big(G_0\cap Z(G)\big)\oplus\big(\Gamma_0\cap Z(\Gamma)\big)$.
We are concerned with the two exact sequences
\begin{eqnarray}\label{SES_1-4}
&&1\lr T\times G\bri H(\beta)\lr\Gamma\lr0\cr
&&1\lr T\briota H(\beta)\brpi G\times\Gamma\lr0\cr
\end{eqnarray}
the first, coming from the semidirect product, is splitting. The second one is a central extensions with the abelian kernel $T$ hence
inducing the trivial action on $T$. We set
\begin{equation*}
\pmb{\varepsilon}(\beta)\ =\ 1\lr T\buildrel i\over\lr H(\beta)\buildrel\pi\over\lr P\lr1\text{ with }P=G\times\Gamma.
\end{equation*}
Since $\pmb{\varepsilon}(\beta)$ is central, it corresponds to a cohomology class ${\rm cls}(\gamma)\in H^2(P,T)$, where $T$ is considered as trivial $P$-module.
For computing a cocycle $\gamma$ we can use the `standard section' of $\pi$
\begin{equation*}
s_0\colon P\lr H(\beta),\ s_0(x,\xi)=(1,x,\xi).
\end{equation*}

\begin{lemma}\label{lem:factorset-from-beta}
Each cocycle $\gamma\in Z^2(P,T)$ has the form
\begin{equation*}
\gamma\big((x,\xi),(y,\eta)\big)=\dual{\xi}{y}\frac{h(x,\xi)h(y,\eta)}{h(xy,\xi\eta)}
\end{equation*}
where $h$ is an arbitrary function $P\lr T$. In particular, the cocycle $\gamma_0$ induced by the standard section is
\begin{equation*}
\gamma_0=\beta\circ(\pi_\Gamma\times\pi_G)
\end{equation*}
where $\pi_\bullet$ denote the respective projection $P\lr\Gamma$, $P\lr G$.
\end{lemma}

\begin{proof}
\begin{eqnarray*}
\gamma_0\big((x,\xi),(y,\eta)\big)&=&i^{-1}\left(s_0(x,\xi)s_0(y,\eta)s_0((x,\xi)(y,\eta))^{-1}\right)\cr
&=&i^{-1}\left((1,x,\xi)(1,y,\eta)(1,xy,\xi\eta)^{-1}\right)\cr
&=&i^{-1}\left((\dual{\xi}{y},xy,\xi\eta)(\dual{\xi\eta}{xy},(xy)^{-1},(\xi\eta)^{-1})\right)\cr
&=&\dual{\xi}{y}\dual{\xi\eta}{xy}\dual{\xi\eta}{(xy)^{-1}}\cr
&=&\dual{\xi}{y}
\end{eqnarray*}
This shows that
\begin{eqnarray*}
\gamma_0=\beta\circ(\pi_\Gamma\times\pi_G)
\end{eqnarray*}
If $\gamma\in Z^2(P,T)$ is an arbitrary cocycle then there is a function $h\in T^P$ s.t. $\gamma=\gamma_0\cdot\partial^2h$. Therefore
\begin{eqnarray*}
\gamma\big((x,\xi),(y,\eta)\big)=\gamma_0\big((x,\xi),(y,\eta)\big)\cdot(\partial^2h)\big((x,\xi),(y,\eta)\big)=\dual{\xi}{y}\frac{h(x,\xi)h(y,\eta)}{h(xy,\xi\eta)}
\end{eqnarray*}
Of course this can be seen also by using an arbitrary section, which must be of the form $s(x,\xi)=(h(x,\xi),x,\xi)$
\end{proof}

The following trivial observations will be useful.
\begin{lemma}
\label{L-Einbettung}\
\begin{enumerate}
\item $H\le G$ and $\Delta\le\Gamma$, $\varphi\colon T\times H\times\Delta\lr H(\beta)$, $(c,x,\xi)\mapsto cx\xi$. Then $\varphi$ is a homomorphism iff
$\dual{\Delta}{H}=1$.
\item If $\Gamma$ and $G$ both have bilinear commutator, then so has $H(\beta)$.
\end{enumerate}
\end{lemma}

\begin{definition} An SES $\varepsilon:1\lr T\lr E\lr P\lr1$ (with $T\in\Ab$) is called {\tt Heisenberg extension} provided that
$\exists$ a factorization $P=G\times\Gamma$, $\exists$ a bilinear $\beta\colon\Gamma\times G\lr T$ s.t. $\varepsilon$ is equivalent with
${\pmb{\varepsilon}}(\beta)$. The sequence $\varepsilon$ is then a {\tt proper Heisenberg extension} in case $T,\,P\in\Ab$ and $\beta$ is a duality.
\end{definition}

\begin{corollary}
Let $\varepsilon:1\lr T\lr E\lr P\lr1$ be a Heisenberg extension. Then $E\in{\rm Nil}_2\iff P\in{\rm Nil}_2$.
\end{corollary}
\begin{proof}
There is a factorization $P=G\times\Gamma$, a bilinear $\beta\colon\Gamma\times G\lr T$ and an isomorphism $\varphi$ s.t.
\begin{equation*}\xymatrix{
1\ar[r]&T\ar@{=}[d]\ar[r]&E\ar[d]^-\varphi_-\cong\ar[r]&P\ar@{=}[d]\ar[r]&1\cr
1\ar[r]&T\ar[r]&H(\beta)\ar[r]&G\times\Gamma\ar[r]&1
}\end{equation*}
commutes.
By Corollary \ref{L4}\,.2, if $E\in{\rm Nil}_2$ then $P\in{\rm Nil}_2$. Conversely, assume $P\in{\rm Nil}_2$. Then
$G,\Gamma\in{\rm Nil}_2$ (by Lemma. \ref{L4}\,.1. By Fact \ref{fct:char-cent-ext}, $\Gamma$ and $G$ have a bilinear commutator.
By Lemma \ref{L-Einbettung} (2) $H(\beta)$ has a bilinear commutator, hence $H(\beta)\in{\rm Nil}_2$. It follows that $E\in{\rm Nil}_2$.
\end{proof}
Since $\pmb{\varepsilon}(\beta)$ is a central extension, the commutator map factors to $\omega\colon G\times\Gamma\times G\times\Gamma\lr H(\beta)'$.
The symbol $\perp$ refers to this map.

\begin{proposition}\label{prop-30.11}\
\begin{enumerate}
\item $\omega\colon P\times P\lr H(\beta)'$, $\omega((x,\xi),(y,\eta))=\Big(\frac{\dual{\xi}{y}}{\dual{\eta}{x}},[x,y],[\xi,\eta]\Big)$
\item $G^\perp=Z(G)\times\Gamma_0\unlhd P$ and $\Gamma^\perp=G_0\times Z(\Gamma)\unlhd P$
\item $\widetilde{G}=T\times G\unlhd H(\beta)$ and $\widetilde{\Gamma}=T\times\Gamma\unlhd H(\beta)$
\item $\widetilde{G}\cap\widetilde{\Gamma}=T$ and $\widetilde{G}\cdot\widetilde{\Gamma}=H(\beta)$
\item $C_{H(\beta)}(\widetilde{G})=T\times Z(G)\times\Gamma_0$ and $C_{H(\beta)}(\widetilde{\Gamma})=T\times G_0\times Z(\Gamma)$.
\end{enumerate}
\end{proposition}
\begin{proof}
\begin{eqnarray*}
1. &\omega((x,\xi),(y,\eta))&=\ \omega(\pi(c,x,\xi),\pi(d,y,\eta))=[(c,x,\xi),(d,y,\eta)]\cr
&&=\ \Big(\frac{\dual{\xi}{y}}{\dual{\eta}{x}},[x,y],[\xi,\eta]\Big)\cr
2. &(x,\xi)\in G^\perp&\iff\forall_{y\in G}\,\omega((x,\xi),(y,1))=1\cr
&&\iff\forall_{y\in G}\,\Big(\frac{\dual{\xi}{y}}{\dual{1}{x}},[x,y],[\xi,1]\Big)=1\cr
&&\iff\forall_{y\in G}\,(\dual{\xi}{y}=1\land[x,y]=1)\iff\xi\in\Gamma_0\land x\in Z(G)\cr
&&\iff(x,\xi)\in Z(G)\times\Gamma_0
\end{eqnarray*}
If $(x,\xi)\in G^\perp$ and $(y,\eta)\in G\times\Gamma$ then
\begin{equation*}
(y,\eta)(x,\xi)(y,\eta)^{-1}=(yxy^{-1},\eta\xi\eta^{-1})=(x,\eta\xi\eta^{-1}).
\end{equation*}
For arbitrary $z\in G$, $\dual{\eta\xi\eta^{-1}}{z}=\dual{\xi}{z}=1$, hence $\eta\xi\eta^{-1}\in\Gamma_0$. Therefore
$(y,\eta)(x,\xi)(y,\eta)^{-1}=(x,\eta\xi\eta^{-1})\in Z(G)\times\Gamma_0=G^\perp$, which shows that $G^\perp\unlhd P$.
\begin{eqnarray*}
(x,\xi)\in\Gamma^\perp&\iff&\forall_{\eta\in\Gamma}\,\omega((x,\xi),(1,\eta))=1\iff\forall_{\eta\in\Gamma}\,\Big(\frac{\dual{\xi}{1}}{\dual{\eta}{x}},[x,1],[\xi,\eta]\Big)=1\cr
&\iff&\forall_{\eta\in\Gamma}\,(\dual{\eta}{x}^{-1}=1\land[\xi,\eta]=1)\iff x\in G_0\land\xi\in Z(\Gamma)\cr
&\iff&(x,\xi)\in G_0\times Z(\Gamma)
\end{eqnarray*}
If $(x,\xi)\in\Gamma^\perp$ and $(y,\eta)\in G\times\Gamma$ then
\begin{equation*}
(y,\eta)(x,\xi)(y,\eta)^{-1}=(yxy^{-1},\eta\xi\eta^{-1})=(yxy^{-1},\xi),\text{ and again}
\end{equation*}
for arbitrary $\rho\in\Gamma$, $\dual{\rho}{yxy^{-1}}=\dual{\rho}{x}=1$, hence $(yxy^{-1},\xi)\in G_0\times Z(\Gamma)$, hence also $\Gamma^\perp\unlhd P$.

\vspace{2mm}
3. As $\widetilde{G}=\pi^{-1}(G)$ it is clear that it is normal in $H(\beta)$; the same for $\widetilde{\Gamma}$.
$\widetilde{G}=\{(c,x,\xi)\mid(x,\xi)\in G\}=\{(c,x,\xi)\mid\xi=1\}$. Multiplying two elements from $\widetilde{G}$ shows that $\widetilde{G}=T\times G$. Similarly
\begin{equation*}
\widetilde{\Gamma}=\{(c,x,\xi)\mid x=1\}=T\times\Gamma.
\end{equation*}
4. $(c,x,\xi)\in\widetilde{G}\cap\widetilde{\Gamma}\iff\xi=1\land x=1$, hence $\widetilde{G}\cap\widetilde{\Gamma}=T$.\\
$(c,x,\xi)=(c,x,1)(1,1,\xi)\in\widetilde{G}\cdot\widetilde{\Gamma}$ demonstrates the last point.

\vspace{2mm}
5. By Proposition \ref{max-abelian-central-ext} $C_{H(\beta)}(\widetilde{G})=\pi^{-1}(G^\perp)$ and
$C_{H(\beta)}(\widetilde{\Gamma})=\pi^{-1}(\Gamma^\perp)$.

\vspace{2mm}
If $(c,x,\xi)\in C_{H(\beta)}(\widetilde{G})$ then $(x,\xi)\in G^\perp= Z(G)\times\Gamma_0$, then $(c,x,\xi)\in T\cdot(Z(G)\times\Gamma_0)$.
By Lemma \ref{L-Einbettung}\,(1), $T\times Z(G)\times\Gamma_0\lr T\cdot(Z(G)\times\Gamma_0)$, $(c,x,\xi)\mapsto cx\xi$ is an embedding, whence
$T\cdot(Z(G)\times\Gamma_0)=T\times Z(G)\times\Gamma_0$, and so $(c,x,\xi)\in T\times Z(G)\times\Gamma_0$.

\vspace{2mm} If, conversely, $(c,x,\xi)\in T\times Z(G)\times\Gamma_0$ then $\pi(c,x,\xi)=(x,\xi)\in Z(G)\times\Gamma_0=G^\perp$, hence 
$(c,x,\xi)\in\pi^{-1}(G^\perp)=C_{H(\beta)}(\widetilde{G})$, which shows the first equality.
The same arguments provide also $C_{H(\beta)}(\widetilde{\Gamma})=T\times G_0\times Z(\Gamma)$.
\end{proof}

In order to bring together the 2 different approaches to Heisenberg groups we introduce the folowing concept.
\begin{definition}\label{L9} $E\in\Grp$. A pair of subgroups $K,N\le E$ is called a {\tt normal splitting} of $E$ provided that
\begin{enumerate}
\item $K$ and $N$ are normal in $E$;
\item $K\,N=E$ and $K\cap N\subseteq Z(E)$;
\item $\exists_{X\le E}\,K=(K\cap N)\times X$;
\item $\exists_{Y\le E}\,N=(K\cap N)\times Y$.
\end{enumerate}

A normal splitting is called {\tt abelian splitting} when in addition $K$ and $N$ are maximal abelian.
\end{definition}
\begin{theorem}\label{L51}\
\begin{enumerate}
\item Let $T\in\Ab$ and $\beta\colon\Gamma\times G\lr T$ bilinear. Then $\widetilde{G},\,\widetilde{\Gamma}$ is a normal splitting of $H(\beta)$.
\item If $T,\Gamma, G\in\Ab$ and $\beta$ is a duality then $\widetilde{G},\,\widetilde{\Gamma}$ is an abelian splitting of $H(\beta)$.
\end{enumerate}
\end{theorem}
\begin{proof}
1. is obvious from Proposition \ref{prop-30.11}. For point 2., take $x,\,y\in G$:
\begin{equation*}
\omega(x,y)=\omega\big((x,1),(y,1)\big)=[(1,x,1),(1,y,1)]=\left(\frac{\dual{1}{y}}{\dual{1}{1}},[x,y],[1,1]\right)=1.
\end{equation*}
This shows that $G\subseteq G^\perp$. Take $p=(x,\xi)\in G^\perp\le P$. Then $\forall_{y\in G}\,\omega\big((x,\xi),(y,1)\big)=1$. Therefore
\begin{equation*}
1=[(1,x,\xi),(1,y,1)]=\left(\frac{\dual{\xi}{y}}{\dual{1}{x}},[x,y],[\xi,1]\right)=(\dual{\xi}{y},1,1)
\end{equation*}
and so $\forall y\in G$ $\dual(\xi){y}=1$, hence, $\xi=1$. Therefore $p=(x,1)\in G$.
Consequently $G^\perp=G$. By Proposition \ref{max-abelian-central-ext}, $\widetilde{G}$ is maximal abelian. The same arguments apply to $\Gamma$.
\end{proof}

\begin{theorem}\label{prop:char-heisgrp}
Assume that $E\in\Grp$, $K,N\unlhd E\land E=KN\land K\cap N\subseteq Z(E)$. Set $T=K\cap N$ and $P=E/T$ so that
$\varepsilon:1\lr T\bri E\brpi P\lr1$ is central and $\omega\colon P\times P\lr E'$ exists. Moreover let $G:=\pi(K)$ and $\Gamma=\pi(N)$.
Then $P=G\times\Gamma$ and $\omega\vert_{\Gamma\times G}$ provides a bilinear map $\beta\colon\Gamma\times G\lr T$. If, in addition, $(K,N)$ constitutess a normal splitting, then $E\cong H(\beta)$ and the resulting central sequence $\pmb{\varepsilon}(\beta)$ is equivalent with $\varepsilon$.
If $(K,N)$ is even an abelian splitting then $G,\Gamma\in\Ab$ and $\beta\colon\Gamma\times G\lr T$ is a duality.
\end{theorem}
\begin{proof}
Obviously $G,\Gamma\unlhd P$, $\widetilde{G}=\pi^{-1}\pi(K)=K$, $\widetilde{\Gamma}=\pi^{-1}\pi(N)=N$.

\vspace{2mm}
For $a\in P$ we have $a=\pi(e)=\pi(kn)=\pi(k)\pi(n)\in G\Gamma$, i.e. $P=G\Gamma$.

\vspace{2mm}
$a\in G\cap\Gamma$ $\Imp$ $a=\pi(e)$, $e\in\pi^{-1}(G)=\pi^{-1}\pi(K)=K$ and $e\in\pi^{-1}(\Gamma)=\pi^{-1}\pi(N)=N$. $\Imp$ $e\in K\cap N=T$. Therefore $a=\pi(e)=1$, i.e.
$G\cap\Gamma=\{1\}$. This shows that $P=G\times\Gamma$.

\vspace{2mm}
Let $\beta:=\omega\vert_{\Gamma\times G}$. Take $\xi\in\Gamma$, $x\in G$. Then $\xi=\pi(n)$, $x=\pi(k)$ ($n\in N$, $k\in K$).
\begin{equation*}
\beta(\xi,x)=\omega\big(\pi(n),\pi(k)\big)=[n,k]=nkn^{-1}k^{-1}\in K\cap N=T
\end{equation*}
thus $\beta$ takes values in $T$. Using normality of $N,K$ and centrality of $K\cap N$:
\begin{eqnarray*}
\beta(\xi,x)\beta(\eta,x)&=&\omega\big(\pi(m),\pi(k)\big)\omega\big(\pi(n),\pi(k)\big)=[m,k][n,k]=mkm^{-1}k^{-1}nkn^{-1}k^{-1}\cr
&=&mkm^{-1}[k^{-1},n]k^{-1}=mk[k^{-1},n]m^{-1}k^{-1}=mkk^{-1}nkn^{-1}m^{-1}k^{-1}\cr
&=&mnkn^{-1}m^{-1}k^{-1}=[mn,k]=\omega\big(\pi(mn),\pi(k)\big)\cr
&=&\omega\big(\pi(m)\pi(n),\pi(k)\big)=\beta(\xi\eta,x)
\end{eqnarray*}
\begin{eqnarray*}
\beta(\xi,x)\beta(\xi,y)&=&\omega\big(\pi(n),\pi(k)\big)\omega\big(\pi(n),\pi(l)\big)=[n,k][n,l]=nkn^{-1}k^{-1}\underbrace{nln^{-1}l^{-1}}\cr
&=&nkn^{-1}\underbrace{nln^{-1}l^{-1}}k^{-1}=nkln^{-1}l^{-1}k^{-1}=[n,kl]=\omega\big(\pi(n),\pi(kl)\big)\cr
&=&\omega\big(\pi(n),\pi(k)\pi(l)\big)=\beta(\xi,xy)
\end{eqnarray*}
Therefore $\beta\colon\Gamma\times G\lr T$ is bilinear and we can build $\pmb{\varepsilon}(\beta)$.
Now assume there are groups $X,Y$ complementary to $T$ in $K$ and $T$ respectively
\begin{equation*}
K=T\times X,\ N=T\times Y.
\end{equation*}
Obviously $\pi\vert_X\colon X\cong G$, so let $s=\big(\pi\vert_X\big)^{-1}$. Similarly, $\pi\vert_Y\colon Y\cong\Gamma$ and we set
$t=\big(\pi\vert_Y\big)^{-1}$. Now define $\varphi(c,x,\xi)=c\,s(x)\,t(\xi)$ ($c\in T,\,x\in G,\,\xi\in\Gamma$).

\vspace{2mm}
For the last point asuume that $K$ and $N$ are maximal abelian. Then $G=\pi(K)$ and $\Gamma=\pi(N)$ are abelian. Becaus $K=\widetilde{G}$ and
$N=\widetilde{\Gamma}$, Proposition \ref{max-abelian-central-ext} yields $G=G^\perp$ and $\Gamma=\Gamma^\perp$.
If $\beta(\xi,x)=1$ $\forall x\in G$ then $\omega(\xi,x)=1$ $\forall x\in G$ whence $\xi\in\Gamma\cap G^\perp=\Gamma\cap G=1$. Similarly
$\beta(\xi,x)=1$ $\forall\xi\in\Gamma$ implies $x=1$. Consequently $\beta$ is a duality.
\end{proof}

\section{Characterization of Heisenberg Extensions}
For dealing with the cohomological conditions it is necessary to return to abelian phase groups. This is owed to the requisiteness that the
omega-form induced by a Heisenberg extension $\pmb{\varepsilon}(\beta)$ should take values in the torus (as opposed to $[H(\beta),H(\beta)]$).

\vspace{2mm}
\begin{quote}
Therefore, for the rest of the paper, we assume that all groups $T,G,\Gamma$ be abelian, and we write $G$ and $\Gamma$ in additive notation
(the torus $T$ shall stay multiplicative.) The symbol $\star$ will denote the contravariant functor $\hom(\bullet,T)$.
\end{quote}
Consider a fixed extension problem $1\bri T\lr E\brpi P\lr0$.
Given a factorization $P=G\oplus\Gamma$ and a bilinear map $\beta\in(\Gamma\otimes G)^\star$, there is the associated Heisenberg group $H(\beta)=(T\times G)\rtimes\Gamma$, its operations now written as
\begin{eqnarray*}
(c,x,\xi)\cdot(d,y,\eta)&=&(cd\dual{\xi}{y},x+y,\xi+\eta)\cr
1&=&(1,0,0)\cr
(c,x,\xi)^{-1}&=&\Big(\frac{\dual{\xi}{x}}{c},-x,-\xi\Big)
\end{eqnarray*}
and the corresponding Heisenberg extension $\pmb{\pmb{\varepsilon}}(\beta)$
\begin{equation*}
\pmb{\varepsilon}(\beta): 1\lr T\bri H(\beta)\brpi G\oplus\Gamma\lr0
\end{equation*}
with $i(c)=(c,0,0)$ and $\pi(c,x,\xi)=(x,\xi)$. Since $\pmb{\varepsilon}(\beta)$ is a central extension of an abelian group, the comuutator form
takes values in $T$. As the commutator is $[(c,x,\xi),(d,y,\eta)]=\Big(\frac{\dual{\xi}{y}}{\dual{\eta}{x}},0,0\Big)$, the corresponding form is
\begin{equation*}
\omega\colon P\times P\lr\ \omega((x,\xi),(y,\eta))=\frac{\dual{\xi}{y}}{\dual{\eta}{x}}.
\end{equation*}
With the standard section $s_0\colon P\lr H(\beta)$, $(x,\xi)\mapsto(1,x,\xi)$ we obtain the 2-cocycle
\begin{equation*}
\gamma_0\colon P\times P\lr T,\ \gamma_0\big((x,\xi),(y,\eta)\big)=\dual{\xi}{y}.
\end{equation*}
Thus, each $\gamma\in Z^2(P,T)$ is $\gamma=\gamma_0\cdot(\partial^2h)$ for some $h\in T^P$, i.e.,
\begin{equation*}
\gamma\big((x,\xi),(y,\eta)\big)=\dual{\xi}{y}\frac{h(x,\xi)h(y,\eta)}{h(x+y,\xi+\eta)}.
\end{equation*}
The standard cocycle $\gamma_0$ is bilinear and it factors through $\Gamma\otimes G$: With projections $\pi^1\colon P\lr G$, $\pi^2\colon P\lr\Gamma$
we get
\begin{equation*}\xymatrix{\Gamma\otimes G\ar[dr]_-\beta&P\otimes P\ar[l]_-{\pi^2\otimes\pi^1}\ar[d]^-{\gamma_0}\cr&T}\end{equation*}
By left exactness of $\hom(\bullet,T)$ the map $(\pi^2\otimes\pi^1)^\star$ is a monomorphism. Because bilinear forms are specific 2-cocycles
we obtain the embedding
\begin{equation*}
\xymatrix{\varphi=(\Gamma\otimes G)^\star\ar[rr]^-{(\pi^2\otimes\pi^1)^\star}&&(P\otimes P)^\star\ar[r]&Z^2(P,T)}
\end{equation*}
$\varphi(\beta)\big((x,\xi),(y,\eta)\big)=\beta(\xi\otimes y)=\dual{\xi}{y}$.

\begin{definition} Let $P=G\oplus\Gamma$ in $\Ab$. Then $H^2_{\Gamma\times G}(P,T)$ denotes the set of those cohomology classes in $H^2(P,T)$ that contain a
bilinear cocycle which factors through $\Gamma\otimes G$.
\begin{equation*}
H^2_{\Gamma\times G}(P,T)=\frac{{\rm im}(\varphi)+B^2(P,T)}{B^2(P,T)}
\end{equation*}
\end{definition}

\begin{lemma}\label{trivial-form-lemma} Let $\beta\colon\Gamma\times G\lr T$ be bilinear, and $h\colon G\times\Gamma\lr T$ an arbitrary map of sets. If
\begin{equation}\label{trivial-form}
\forall_{\xi,\eta\in\Gamma}\,\forall_{x,y\in G}\,\beta(\xi,y)=\frac{h(x,\xi)h(y,\eta)}{h(x+y,\xi+\eta)}
\end{equation}
then $\beta$ is the constant map $\beta=1$.
\end{lemma}
\begin{proof}
Spezializing $x=\eta=0$ in (\ref{trivial-form}) gives
\begin{equation}\label{trivial-form-1}
\beta(\xi,y)=\frac{h(0,\xi)h(y,0)}{h(y,\xi)}
\end{equation}
$\xi=y=0$ in (\ref{trivial-form}) yields
\begin{equation*}
1=\frac{h(x,0)h(0,\eta)}{h(x,\eta)}
\end{equation*}
and therefore $h(x,\eta)=h(x,0)h(0,\eta)$ $\forall\eta\in\Gamma,\,\forall x\in G$. Rephrasing this identity:
\begin{equation*}
h(y,\xi)=h(y,0)h(0,\xi)\ \forall\xi\in\Gamma\,\forall y\in G
\end{equation*}
and plugging into (\ref{trivial-form-1}) gives
\begin{equation*}
\beta(\xi,y)=\frac{h(0,\xi)h(y,0)}{h(y,0)h(0,\xi)}=1.
\end{equation*}
\end{proof}
\begin{proposition}
${\rm im}(\varphi)\cap B^2(P,T)=0$.
\end{proposition}
\begin{proof}
Take $\gamma\in{\rm im}(\varphi)\cap B^2(P,T)$. Then $\gamma=\varphi(\beta)$, hence
\begin{equation*}
\gamma((x,\xi),(y,\eta))=\varphi(\beta)((x,\xi),(y,\eta))=\beta(\xi\otimes y)
\end{equation*}
Since $\gamma$ is a coboundary, $\exists h\colon P\lr T$ with $\gamma=\partial^2h$. Therefore
\begin{equation*}
\beta(\xi\otimes y)=\gamma((x,\xi),(y,\eta))=\frac{h(x,\xi)h(y,\eta)}{h(x+y,\xi+\eta)}
\end{equation*}
Lemma \ref{trivial-form-lemma} yields $\beta=1$, hence $\gamma=1$.
\end{proof}

\begin{corollary}
\begin{equation}
H^2_{\Gamma\times G}(P,T)\cong(\Gamma\otimes G)^\star.
\end{equation}
\end{corollary}
\begin{proof}
\begin{equation*}
H^2_{\Gamma\times G}(P,T)=\frac{{\rm im}(\varphi)+B^2(P,T)}{B^2(P,T)}\cong\frac{{\rm im}(\varphi)}{{\rm im}(\varphi)\cap B^2(P,T)}={\rm im}(\varphi)
\end{equation*}
Consequently $(\Gamma\oplus G)^\star\cong H^2_{\Gamma\times G}(P,T)$.
\end{proof}
As a consequence we obtain the following result.
\begin{proposition}\label{prop:H2-GGamma-Subgroup}
$P=\Gamma\otimes G$. Then $\beta\mapsto{\rm cls}(\pmb{\varepsilon}(\beta))$ defines an isomorphism
\begin{equation*}
(\Gamma\otimes G)^\star\cong H^2_{\Gamma\times G}(P,T).
\end{equation*}
\end{proposition}
\begin{proof} Obvious. \end{proof}
When $\gamma$ is a cocycle in $Z^2(P,T)$, the function ${\rm cls}(\gamma)\mapsto\omega$ with $\omega(x\land y)=\frac{\gamma(x,y)}{\gamma(y,x)}$
is a homomorphism $H^2(P,T)\lr(P\land P)^\star=\Omega^2(P,T)$, as can be verified directly computing the cocycle identity or by representing $\gamma$
by a set-theoretic cross-section w.r.t. a central extenion. From the appropriate instance of the universal coefficient theorem
\begin{equation*}
1\lr{\rm Ext}^1(P,T)\lr H^2(P,T)\brq\Omega^2(P,T)\lr0
\end{equation*}
it is known that $q$ is an epimorphism--in fact the sequence splits. There is a special case where a splitting can be obtained easily.
\begin{lemma}\label{lemma:2-divisible-torus}
Assume the abelian group $T$ is uniquely 2-divisible, i.e., each element has a unique square root. Then
\begin{equation*}
\sigma\colon\omega\mapsto{\rm cls}(\sqrt{\omega})
\end{equation*}
is a cross section of $H^2(P,T)\brq\Omega^2(P,T)$.
\end{lemma}
\begin{proof}
Since $\omega$ is bilinear and $\xymatrix{\sqrt{\omega}=P\land P\ar[r]^-\omega&T\ar[r]^-{\sqrt{\ }}_-{\cong}&T}$ it is plain that $\sqrt{\omega}$,
considered as map $P\times P\lr T$ is in $Z^2(P,T)$. Obviously $\sqrt{\omega\cdot\mu}=\sqrt{\omega}\cdot\sqrt{\mu}$, and therfore also
\begin{equation*}
{\rm cls}\big(\sqrt{\omega\cdot\mu}\big)={\rm cls}\big(\sqrt{\omega}\big)\cdot{\rm cls}\big(\sqrt{\mu}\big)
\end{equation*}
which demonstrates that $\sigma$ is a homomorphism.
\begin{eqnarray*}
(q\circ\sigma)(\omega)(x\wedge y)&=&q\big({\rm cls}(\sqrt{\omega})\big)(x\wedge y)=\frac{\sqrt{\omega}(x\wedge y)}{\sqrt{\omega}(y\wedge x))}\cr
&=&\sqrt{\omega(x\wedge y)}\cdot\sqrt{\omega(x\wedge y)^{-1}}=\sqrt{\omega(x\wedge y)}^2\cr
&=&\omega(x\wedge y).
\end{eqnarray*}
therefore $q\circ\sigma=1_{\Omega^2(P,T)}$.
\end{proof}
A uniquely 2-divisible abelian group is a module over the localization $\Z_2=\Z[\frac{1}{2}]$. For this reason it will be written additively in the
following Proposition.
\begin{proposition}\label{prop:skewing-section-2div}
Let $\chi\colon A\lr T$ be a homomorphism, where $(A,+)\in\Ab$ is uniquely 2-divisible. If $\omega\in\chi_\star\big(\Omega^2(P,T)\big)$
then we can compute a cocycle $\gamma\in Z^2(P,T)$ with $q({\rm cls}(\gamma))=\omega$.
\end{proposition}
\begin{proof}
Consider the following diagram
\begin{equation*}
\xymatrix{H^2(P,A)\ar[d]_-{\chi_\star}\ar[r]^-{q_A}&\Omega^2(P,A)\ar[d]^-{\chi_\star}\cr H^2(P,T)\ar[r]^-{q_T}&\Omega^2(P,T)}
\end{equation*}
\begin{eqnarray*}
&&(\chi_\star\circ q_A)({\rm cls}(\mu))(x\wedge y)=\big(\chi\circ q_A({\rm cls}(\mu))\big)(x\wedge y)=\chi(\mu(x,y)-\mu(y,x))\cr
&=&\frac{(\chi\circ\mu)(x,y)}{(\chi\circ\mu)(y,x)}=\frac{\chi_\star(\mu)(x,y)}{\chi_\star(\mu)(y,x)}=
q_T\big({\rm cls}(\chi_\star(\mu))\big)(x\wedge y)\cr
&=&(q_T\circ\chi_\star)({\rm cls}(\mu))(x\wedge y)
\end{eqnarray*}
shows commutativity. Choose $\mu\in\Omega^2(P,A)$ with $\chi_\star(\mu)=\omega$. By Lemma \ref{lemma:2-divisible-torus}, $\frac{1}{2}\mu$ is a cocyle in $Z^2(P,A)$ that maps to $\mu$, i.e., $q_A\big({\rm cls}(\frac{1}{2}\mu)\big)=\mu$. Therefore
\begin{equation*}
q_T\left(\chi_\star\Big({\rm cls}\big(\frac{1}{2}\mu\big)\Big)\right)=\chi_\star\left(q_A\big({\rm cls}\big(\frac{1}{2}\mu)\big)\Big)\right)
=\chi_\star(\mu)=\omega.
\end{equation*}
\end{proof}
We can give a slight generalization of Lemma \ref{lemma:2-divisible-torus}.

\vspace{2mm}
Consider the torus $(T,\cdot)\in\Ab$ and write $T^2$ for the image of the map $T\lr T$, $x\mapsto x^2$. Assume that the set of roots of 1
in $T$
has a complement, i.e.

\begin{equation*}
\exists_{S\le T}\,\big(\ker(x^2)\cdot S=T\land\ker(x^2)\cap S=\{1\}\big).
\end{equation*}

This means, the SES
$\varepsilon^2:1\lr\ker(x^2)\lr T\buildrel x^2\over\lr T^2\lr1$
splits and there is a cross section, i.e., a homomorphism $r\colon T^2\lr T$ s.t. $r(c)^2=c$ $\forall c\in T^2$.
So $r$ is a partial square-root function for $T$.

\begin{proposition}\label{prop:slight-generalization-of-lemma:2-divisible-torus}
Assume the torus has the property mentioned above.
Given an abelian group $(P,+)$ and let $\omega\in(P\wedge P)^\star$ be such that ${\rm im}(\omega)\in T^2$.
\begin{equation*}\xymatrix{
P\wedge P\ar[r]^-\omega&T^2\ar[r]^-r&T\cr P\times P\ar[u]\ar[ur]_-{\omega^\wedge}
}\end{equation*}
As before, let $q\colon H^2(P,T)\lr(P\wedge P)^\star$ denote the map occurring in the universal coefficient theorem.
Then $\gamma:=r\circ\omega^\wedge\in Z^2(P,T)$ and $q\big({\rm cls}(\gamma)\big)=\omega$.
\end{proposition}
\begin{proof}
Since $\omega^\wedge$ is bilinear and $r$ is a homomorphism, $\gamma$ is bilinear, hence a cocycle.
\begin{eqnarray*}
q\big({\rm cls}(\gamma)\big)(x\wedge y)&=&\frac{\gamma(x,y)}{\gamma(y,x)}=\frac{r(\omega(x\wedge y))}{r(\omega(y\wedge x))}=
r\left(\frac{\omega(x\wedge y)}{\omega(y\wedge x)}\right)\cr
&=&r\big(\omega(x\wedge y)\,\omega(x\wedge y)\big)=r(\omega(x\wedge y)^2)=r(\omega(x\wedge y))^2\cr
&=&\omega(x\wedge y)
\end{eqnarray*}
\end{proof}

The association of a bilinear form to its Heisenberg extension is functorial:
\begin{definition}\
\begin{enumerate}
\item We write $\Bi$ for the conmma category $\otimes\downarrow\Ab$. That is, $\Bi$ has bilinear forms $\beta\colon\Gamma\times G\lr T$
as objects--all three groups $\Gamma,G,T$ being abelian.
\vspace{2mm}
A morphism $\beta\lr\beta'$ is a triple of homomorphisms $(\gamma,g,t)$ that commutes the diagram
\begin{equation}\label{L10}
\xymatrix{
\Gamma\times G\ar[d]_-{\gamma\times g}\ar[r]^-\beta&T\ar[d]^-t\cr\Gamma'\times G'\ar[r]^-{\beta'}&T'
}\end{equation}
The full subcategory of $\Bi$ consisting of dualities on abelian groups as objects is denoted by $\Du$.
\item The objects of the category $\overline{\bf Hei}$ are central extensions $\varepsilon$ of abelian groups $G\oplus\Gamma$; a morphism
$\varepsilon\lr\varepsilon'$ is a quadrupel $(t,e,g,\gamma)$ of homomorphisms commuting the diagram
\begin{equation}\label{L11}
\xymatrix{
\varepsilon:1\ar[r]&T\ar[d]_-t\ar[r]^-i&E\ar[d]^-e\ar[r]^-\pi&G\oplus\Gamma\ar[d]^-{g\oplus\gamma}\ar[r]&0\cr
\varepsilon':1\ar[r]&T'\ar[r]^-{i'}&E'\ar[r]^-{\pi'}&G'\oplus\Gamma'\ar[r]&0\cr
}\end{equation}
We write $\Hei$ for the full subcategory of $\overline{\bf Hei}$ whose objects are central extensions of abelian groups $G\oplus\Gamma$ in such a way
that $\widetilde{G},\,\widetilde{\Gamma}$ provide an abelian splitting.
\end{enumerate}

The functor $\pmb{\varepsilon}\colon\Bi\lr\overline{\bf Hei}$ acts as
\begin{equation*}
\xymatrix{
\beta\ar[r]^-{(\gamma,g,t)}&\beta'&\ar@{|->}[r]^-{\pmb{\varepsilon}}&&\pmb{\varepsilon}(\beta)\ar[rr]^-{(t,t\times g\times\gamma,g,\gamma)}&&\pmb{\varepsilon}(\beta')
}\end{equation*}

The functor $\pmb{\beta}\colon\overline{\bf Hei}\lr\Bi$ acting on objects as $\pmb{\beta}(\varepsilon)=\omega(\varepsilon)|_{\Gamma\times G}$, is given by
\begin{equation*}
\xymatrix{
\varepsilon\ar[rr]^-{(t,e,g,\gamma)}&&\varepsilon'&\ar@{|->}[r]^-{\pmb{\beta}}&&\omega(\varepsilon)|_{\Gamma\times G}\ar[rr]^-{(\gamma,g,t)}&&\omega(\varepsilon')|_{\Gamma'\times G'}
}\end{equation*}
so the value of the morphism $(t,e,g,\gamma)$ is the commutative diagram (\ref{L10}).
\end{definition}
\noindent Choosing liftings $(\widetilde{\xi},\widetilde{x})\in E\times E$ we obtain $(t\circ\pmb{\beta}(\varepsilon))(\xi,x)=ti^{-1}[\widetilde{\xi},\widetilde{x}]$
hence $i'\big((t\circ\pmb{\beta}(\varepsilon))(\xi,x)\big)=i'ti^{-1}[\widetilde{\xi},\widetilde{x}]=eii^{-1}[\widetilde{\xi},\widetilde{x}]=[e(\widetilde{\xi}),e(\widetilde{x})]$.
Since $\pi'(e(\widetilde{\xi}))=(\pi'\circ e)(\widetilde{\xi}))=(g\oplus\gamma\circ\pi)(\widetilde{\xi})=(g\oplus\gamma)(\xi)=\gamma(\xi)$ and similar
$\pi'(e(\widetilde{x}))=g(x)$, we get
\begin{equation*}
i'\big(\pmb{\beta}(\varepsilon')\circ\gamma\times g)(\xi,x)\big)=i'\big(\omega(\varepsilon')|_{\Gamma'\times G'}(\gamma(\xi),g(x))\big)=[e(\widetilde{\xi}),e(\widetilde{x})]
\end{equation*}
Thus $t\circ\pmb{\beta}(\varepsilon)=\pmb{\beta}(\varepsilon')\circ\gamma\times g$, i.e., $(\gamma,g,t)$ is indeed a morphism in $\Bi$.

\begin{proposition}\label{L61}
$\pmb{\beta}\circ\pmb{\varepsilon}=1_{\Bi}$.
\end{proposition}
\begin{proof}
The commutator form $\omega(\pmb{\varepsilon}(\beta))$ for a bilinear $\beta\colon\Gamma\times G\lr T$ is
\begin{equation*}
\omega(\pmb{\varepsilon}(\beta))\colon G\oplus\Gamma\times G\oplus\Gamma\lr T,\ \big((x,\xi),(y,\eta)\big)\mapsto\frac{\dual{\xi}{y}}{\dual{\eta}{x}}
\end{equation*}
and therefore $\pmb{\beta}(\pmb{\varepsilon}(\beta))=\omega(\pmb{\varepsilon}(\beta))|_{\Gamma\times G}=\beta$. On morphisms we get
\begin{equation*}
\xymatrix{
\beta\ar[r]^-{(\gamma,g,t)}&\beta'\ar@{|->}[r]^-{\pmb{\varepsilon}}&\pmb{\varepsilon}(\beta)\ar[rr]^-{(t,t\times g\times\gamma,g,\gamma)}&&\pmb{\varepsilon}(\beta')
\ar@{|->}[r]^-{\pmb{\beta}}&\beta\ar[r]^-{(\gamma,g,t)}&\beta'
}\end{equation*}
which proves the statement.
\end{proof}
\begin{proposition}\label{prop:heis-functor}
Consider $\varepsilon:1\lr T\buildrel i_\varepsilon\over\lr E\buildrel\pi_\varepsilon\over\lr G\oplus\Gamma\lr0\ \in\overline{\bf Hei}$, and let $\widetilde{G}$,
$\widetilde{\Gamma}$ denote the inverse images in $E$ of $G$, $\Gamma$ respectively. If $\widetilde{G}, \widetilde{\Gamma}$ provide a normal splitting, then
$\varepsilon\sim\pmb{\varepsilon}\big(\pmb{\beta}(\varepsilon)\big)$.
\end{proposition}
\begin{proof}
Due to the splitting concept (cf. Definition \ref{L9}) there are normal subgroups $X,Y\unlhd E$ with
\begin{equation*}
\widetilde{G}=i_\varepsilon(T)\cdot X,\ \widetilde{\Gamma}=i_\varepsilon(T)\cdot Y.
\end{equation*}
It follows that restriction of $\pi_\varepsilon$ to $X$ and $Y$ are isomorphisms
\begin{equation*}
\pi_\varepsilon|X\colon X\cong G,\ \pi_\varepsilon|Y\colon Y\cong\Gamma.
\end{equation*}
Let $g:=\big(\pi_\varepsilon|X\big)^{-1}$, $\gamma:=\big(\pi_\varepsilon|Y\big)^{-1}$, and define the map
\begin{equation*}
\varphi\colon H(\pmb{\beta}(\varepsilon))\lr E,\ \varphi(c,x,\xi)=i_\varepsilon(c)g(x)\gamma(\xi)\quad(c\in T,\,x\in G,\,\xi\in\Gamma).
\end{equation*}
Then $\varphi$ is a homomorphism:
\begin{eqnarray*}
\varphi\big((c,x,\xi)(d,y,\eta)\big)&=&\varphi\big(c\,d\cdot\omega(\varepsilon)(\xi,y),x+y,\xi+\eta\big)\cr
&=&i_\varepsilon\big(c\,d\cdot\omega(\varepsilon)(\xi,y)\big)\cdot g(x+y)\cdot\gamma(\xi+\eta)\cr
&=&i_\varepsilon(c\,d)\cdot i_\varepsilon(i_\varepsilon)^{-1}[\gamma(\xi),g(y)]\cdot g(x)\,g(y)\,\gamma(\xi)\,\gamma(\eta)\cr
&=&i_\varepsilon(c)\,i_\varepsilon(d)\,[\gamma(\xi),g(y)]\cdot g(x)\,g(y)\,\gamma(\xi)\,\gamma(\eta)\cr
&&\cr
\varphi(c,x,\xi)\,\varphi(d,y,\eta)&=&i_\varepsilon(c)\,g(x)\,\gamma(\xi)\,i_\varepsilon(d)\,g(y)\,\gamma(\eta)\cr
&=&i_\varepsilon(c)\,i_\varepsilon(d)\,g(x)\,\overbrace{\gamma(\xi)\,g(y)}\,\gamma(\eta)\cr
&=&i_\varepsilon(c)\,i_\varepsilon(d)\,g(x)\,[\gamma(\xi),\,g(y)]\,g(y)\gamma(\xi)\,\gamma(\eta)\cr
&=&i_\varepsilon(c)\,i_\varepsilon(d)\,[\gamma(\xi),\,g(y)]\,g(x)\,g(y)\gamma(\xi)\,\gamma(\eta)
\end{eqnarray*}
which comes to the same. Obviously $\varphi\circ i=i_\varepsilon$ and $\pi_\varepsilon\circ\varphi=\pi$. This demonstrates the equivalence of the central extensions
$\varepsilon$ and $\pmb{\varepsilon}\big(\pmb{\beta}(\varepsilon)\big)$.
\end{proof}
\begin{corollary}\label{L63}
Restriction of the functors $\pmb{\varepsilon}$ and $\pmb{\beta}$ to the full subcategories $\Du$ and $\Hei$ respectively results in an equivalence
$\Du\simeq\Hei$.
\end{corollary}
\begin{proof}
If $\varepsilon\in\Hei$ then $\varepsilon$ is central with $(\widetilde{G},\widetilde{\Gamma})$ as an abelian splitting. By Theorem \ref{prop:char-heisgrp},
$\pmb{\beta}(\varepsilon)$ is a duality, hence $\pmb{\beta}(\varepsilon)\in\Du$.

\vspace{2mm}
If $\beta\in\Du$ the it is a duality $\Gamma\times G\buildrel\beta\over\lr T$ on abelian groups $\Gamma,G,T$. $\pmb{\varepsilon}(\beta)$ is then the central extension
\begin{equation*}
1\lr T\lr H(\beta)\lr G\oplus\Gamma\lr0
\end{equation*}
which has $(\widetilde{G},\widetilde{\Gamma})$ as abelian splitting (due to Theorem \ref{L51}). Therefore, by its very definition,
$\pmb{\varepsilon}(\beta)\in\Hei$. Consequently, $\pmb{\varepsilon}$ and $\pmb{\beta}$ restrict to
\begin{equation*}
\xymatrix{
\Du\ar@<2pt>[r]^-{\pmb{\varepsilon'}}&\Hei\ar@<2pt>[l]^-{\pmb{\beta'}}
}\end{equation*}
From Proposition \ref{L61} we get $\pmb{\beta'}\circ\pmb{\varepsilon'}=1_{\Du}$. Since an abelian splitting is in particular a normal splitting, we get
from Proposition \ref{prop:heis-functor} that $\varepsilon\sim\pmb{\varepsilon'}(\pmb{\beta'}(\varepsilon))$ $\forall\varepsilon\in\Hei$.
In particular, $\varepsilon\cong(\pmb{\varepsilon'}\circ\pmb{\beta'})(\varepsilon)$ in the category $\Hei$. Therefore the functors
$\pmb{\varepsilon'},\,\pmb{\beta'}$ provide an equivalence. Plainly then they yield also a mutually adjoint situation.
\end{proof}

We conclude this section with an observation concerning dualities.
\begin{lemma}\label{lem:bichar-nondeg}
Let $F$ be a field of zero characteristic. For a bilinear form $\funcinner{}{}\colon V \times V \to F$ on a vector space,
  let~$\inner{}{}\colon V \times V \to T$ be its associated bicharacter with respect to any fixed
  standard character~$\chi\colon (F, +) \to (T, \cdot)$. Then~$\inner{}{}$ is nondegenerate
  iff~$\funcinner{}{}$ is nondegenerate.
\end{lemma}
\begin{proof}
Write $\alpha\colon V\times V\lr F$ and $\beta=\chi\circ\alpha\colon V\times V\lr T$ for the bilinear forms.
For the left and right kernels we have $_0\alpha\subseteq{}_0\beta$ and $\alpha_0\subseteq\beta_0$
(cf. Definition \ref{def:left-kernel-right-kernel}), therefore if $\beta$ is a duality then so is $\alpha$.

\vspace{2mm}
Now let $\alpha$ be a duality, i.e. $_0\alpha=\alpha_0=0$ and assume $\forall_y\,\beta(v,y)=1$. $\chi$ being standard means
$\ker(\chi)$ is the prime ring of $F$, which is $\Z$,hence $\forall_y\,\alpha(v,y)\in\Z$.

\vspace{2mm}
Assume $\exists_y\,\alpha(v,y)\ne0$. Let $k:=\min\{n\in\Z_{>0}\mid\exists_y\,\alpha(v,y)=n\}$ and set
$\alpha(v,y_0)=k$. Then
\begin{equation*}
\alpha\Big(v,\frac{y_0}{2}\Big)+\alpha\Big(v,\frac{y_0}{2}\Big)=\alpha(v,y_0)=k
\end{equation*}
whence $\alpha\big(v,\frac{y_0}{2}\big)=\frac{k}{2}<k$ a contradiction. Therefore $\forall_y\,\alpha(v,y)=0$ and since $\alpha$ is a duality,
$v=0$. Consequently $_0\beta=0$. The same happens with the right kernels. Therefore $\beta$ is a duality. 
\end{proof}

\section{Heisenberg Groups with a View towards Fourier Theory}

The crucial object to obtain a Heisenberg group is a \emph{Lagrangian bisection}, which we
generalize from the classical setting~\cite[p.~21]{BatesWeinstein1997}
as a pair~$(G, \Gamma)$ of transverse Lagrangian subgroups, meaning
~$G, \Gamma \le \mathbf{Iso}(P) \cap \mathbf{Co}(P)$ such
that~$G + \Gamma = P$. It is a simple observation that one has then a
direct sum decompostion, provided the commutator form is nondegenerate.

\begin{lemma}
  Assume~$P \in \Ab$. 
  If~$(G, \Gamma)$ is a Lagrangian bisection of~$P \in \Ab$,
  then~$G \dotplus \Gamma = P$.
\end{lemma}
\begin{proof}
  We need only show~$G \cap \Gamma = 0$, so
  assume~$x \in G \cap \Gamma$. Since~$G^\perp = G$
  and~$\Gamma^\perp = \Gamma$, this implies that~$x$ is orthogonal to
  any element of~$G$ and also orthogonal to any element of~$\Gamma$. Hence we have~$\omega()$
\end{proof}

If~$\gamma$ is a factor set of~$T \oset{\iota}{\rightarrowtail} H \oset{\pi}{\twoheadrightarrow} P$,
strict centrality may be characterized by the condition
that~$\forall_{w \in P}\, \gamma(z,w) = \gamma(w,z)$ implies~$z = 0$. As this condition is invariant
under cohomological equivalence, the \emph{strictly central extensions} form a well-defined
subset~$H^2_{!}(P, T) \subset H^2(P, T)$ that is, however, not a subgroup. Moreover, abelian
extensions correspond to symmetric factor sets; identifying the corresponding cohomology groups in
the sense that~$\Ext^1(P, T) \le H^2(P, T)$, we obtain the invariance
property~$H^2_{!}(P, T) + \Ext^1(P, T) = H^2_{!}(P, T)$. It is interesting to consider an example of
this.

\begin{myexample}
  \label{ex:pert-discrete-heis}
  The central
  extension~$\ZZ_N \oset{\iota}{\rightarrowtail} H \oset{\pi}{\twoheadrightarrow} \ZZ_N \oplus
  \ZZ_N$
  is defined by the semidirect product~$H := (\ZZ_N \times \ZZ_N) \rtimes \ZZ_N$, where~$\ZZ_N$ acts
  on~$\ZZ_N \times \ZZ_N$ via addition in the second factor. In other words, the group operation
  in~$H$ is defined by
  \begin{equation*}
    c(x,\xi) \cdot c'(x',\xi') = (c+c'+\xi+x') \, (x+x', \xi+\xi')
  \end{equation*}
  for~$c(x,\xi), c'(x',\xi') \in (\ZZ_N \times \ZZ_N) \rtimes \ZZ_N$, continuing with the center
  convention introduced after Equation~\eqref{ex:class-heis-mod}. As we shall see later
  (Example~\ref{ex:finite-group}), the group~$H$ is an important specimen of a \emph{discrete
    Heisenberg group}. The insertion~$\iota$ is here defined by~$c \mapsto c(0,0)$, and the
  projection~$\pi$ by~$c(x,\xi) \mapsto (x,\xi)$. Note that~$H$ has indeed center~$\ZZ_N$, so the
  given extension is \emph{strictly central}.

  Let us now ``perturb'' the Heisenberg group~$H$ by \emph{adding the abelian
    extension}~$\ZZ_N \oset{\iota'}{\rightarrowtail} A \oset{\pi'}{\twoheadrightarrow} \ZZ_N \oplus
  \ZZ_N$,
  where the torus~$\ZZ_N$ is embedded into~$A := \ZZ_{N^2} \oplus \ZZ_N$
  via~$\iota'(k+N\ZZ) := (kN+N^2\ZZ, N\ZZ)$ with the corresponding
  projection~$\pi'(m+N^2\ZZ,n+N\ZZ) := (m+N\ZZ, n+N\ZZ)$. Note that this extension is not trivial
  in~$\Ext^1(\ZZ_N \oplus \ZZ_N, \ZZ_N)$ since~$A$ is \emph{not} isomorphic to a direct sum
  of~$\ZZ_N$ and~$\ZZ_N \oplus \ZZ_N$. It is easy to see that the cocycle of the central
  extension~$H$ is given by~$\gamma(x, \xi; x', \xi') = \xi+x'$, that of the abelian extension~$A$
  by
  \begin{align*}
    \alpha(m &+ N\ZZ, n + N\ZZ; m' + N\ZZ, n' + N\ZZ) = 
    \big\lfloor \tfrac{m \bmod N + m' \bmod N}{N} \big\rfloor + N\ZZ\\
    &= \tfrac{m \bmod N + m' \bmod N - (m+m') \bmod N}{N} + N\ZZ
  \end{align*}
  This extension corresponds to the ``natural section''~$s'\colon \ZZ_N \oplus \ZZ_N \to A$
  of~$\pi'$ given by~$s'(m+N\ZZ, n+N\ZZ) := (m \bmod N + N^2\ZZ, n+N\ZZ)$, which is of course
  \emph{not} a homomorphism since the abelian extension~$A$ is not trivial, as noted
  before. (Without the modding operation, the map would not be well-defined; using~$mN+N^2\ZZ$ in
  place of~$m \bmod N + N^2\ZZ$ yields a homomorphism that is not a section of~$\pi'$.)

  For adding the two cocycles, it is practical to use positive representatives~$0, 1, \dots, N-1$
  of~$\ZZ_N$. Moreover, we shall suppress the arguments~$n+N\ZZ$ and~$n'+N\ZZ$ of~$\alpha$ since
  they are irrelevant. Under these conventions, the two cocycles are given by
  \begin{equation*}
    \gamma(x, \xi; x', \xi') = (\xi+x') \bmod N
    \quad\text{and}\quad
    \alpha(x; x') = \tfrac{x+x'-(x+x') \bmod N}{N} .
  \end{equation*}
  Using Iverson's bracket convention~\cite[(2.5)]{GrahamKnuthPatashnik1994}, one can write the
  abelian cocycle as~$\alpha(x; x') = [x+x' \ge N]$, revealing it as the \emph{carry
    bit}~$\epsilon_{x,x'}$ of addition modulo~$N$. Reverting to coset notation, the operation
  \begin{equation*}
    c(x,\xi) \cdot c'(x',\xi') = (c+c'+\xi+x'+\epsilon_{x,x'}) \, (x+x', \xi+\xi')    
  \end{equation*}
  in the sum group~$\tilde{H} := T \times_{\gamma+\alpha} P$ is indeed a perturbation of the one
  given above for~$H$. Since the commutator form of the corresponding central extension
  for~$\tilde{H}$ is still~$\omega(x,\xi) = \gamma(x,\xi)-\gamma(\xi,x)$ as for~$H$, we see that the
  assignment of commutator forms to (equivalence classes of) strictly central extensions is not
  injective. In other words, the mapping $H^2_{!}(P, T) \to \Hom(\Lambda^2 P, T)$ is not an injection.
\end{myexample}

Having a symplectic structure~$\omega$ on the phase space~$P$ of a central
extension~$T \oset{\iota}{\rightarrowtail} H \oset{\pi}{\twoheadrightarrow} P$, it is natural to
consider its \emph{isotropic} subgroups $G \le P$ with respect to~$\omega$. As usual, these are
defined by the condition~$G \le G^\perp$, where one defines the symplectic orthogonal
by~$G^\perp := \{ z \in P \mid \forall_{w \in G}\: \omega(z, w) = 1\}$. Equivalently, one may also
require that~$\omega$ restricts to~$1$ on~$G$. Similarly, a \emph{coisotropic}
subgroup~$\Gamma \le P$ is defined by~$\Gamma^\perp \le \Gamma$; equivalently, $\omega$ descends to
a nondegenerate symplectic form on the quotient~$P/\Gamma$. A subgroup that is both isotropic and
coisotropic is called \emph{Lagrangian}. As in the case of vector spaces, this is equivalent to
being maximal isotropic or minimal coisotropic (with respect to inclusion).


Confer~\cite{BeylTappe2006} and~\cite{Webb2019}.

Motivate ``our'' definition of Heisenberg groups as a natural algebraic analog of Mumford's
definition in~\cite{Mumford2007b}. But also clarify the contrast to ``Mumford groups'' in the sense
of~\cite[Def.~5.1]{BonattoDikranjan2017}: The stipulation of a bijective instead of injective map
into the dual (yielding a strong rather than a weak symplectic duality!) is too strong for the
algebraic setting. It can only be realized in the topological category where one uses the continuous
dual instead of any purely algebraic dual.

The motivation of Mumford is also very clear: Heisenberg groups in his sense characterize the cases
where essentially unique irreducible representations exist, as established
in~\cite[Thm.~1.2ii]{Mumford2007b} and more explicitly in~\cite[Thm.~1.4]{Vemuri2008}. Of course
here one always thinks of unitary representations, as in the classical Stone--von Neumann
representation of the standard Heisenberg group on~$L^2(\RR^n)$.
In the special case~$G=\RR^n$, Mumford's notion of Heisenberg
group leads to the important topic of \emph{harmonic analysis in phase space}~\cite{Folland2016},
via the so-called symplectic Fourier transform.
Note that here a \emph{representation} is defined as an equivalence class under unitary isomorphism; its
representatives are then called the \emph{realizations} of this representation. In the case of
locally compact abelian groups, there is a bijective correspondence between irreducible
representations and characters~\cite[\S4.1]{Folland1994}. In the nonabelian case, irreducible
(unitary) representations constitute the so-called \emph{unitary dual}, which is isomorphic to the
spectrum of the group algebra~$C^*(G)$.


Using the various characterizations, we can now exhibit a series of infinite groups that are
nilquadratic but fail to be Heisenberg groups.

\begin{myexample}
  \label{ex:free-nil}
  The \emph{free nilquadratic group on~$s$ generators}~\cite{AlbertJohn1991} may be defined as the
  group~$N_s := N_{2,s}$ generated by~$\{x_i, x_{ij} \mid 1 \le i < j \le s\}$ subject to the
  relations~$[x_i, x_j] = x_{ij}$, and~$[x_{ij}, x_k] = [x_{ij}, x_{kl}] = 1$. It is easy to see
  that every element~$x \in N_s$ has the unique representation
  \begin{equation*}
    x = \left( \prod_{1 \le i \le s} x_i^{\alpha_i} \right) \left( \prod_{1 \le i < j \le s}
      x_{ij}^{\beta_{ij}} \right)
  \end{equation*}
  with arbitrary exponents~$\alpha_i, \beta_{ij} \in \ZZ$. Writing~$s' := \binom{s}{2}$, it is clear
  that~$[N_s, N_s] = \zentrum(N_s) = \langle x_{ij} \mid 1 \le i < j \le s\rangle \cong \ZZ^{s'}$
  with quotient the
  abelianization~$N_s^{\mathrm{ab}} = N_s/[N_s, N_s] = \langle [x_i] \mid 1 \le i \le s\rangle \cong
  \ZZ^s$. Hence one obtains the central extension
  \begin{equation*}
    \xymatrix @M=0.75pc @R=1.25pc @C=1pc%
    { 1 \ar[r] & [N_s, N_s] \ar[r] & N_s \ar[r] & N_s^{\mathrm{ab}} \ar[r] & 1}
  \end{equation*}
  whose commutator form~$\omega_s\colon N_s^{\mathrm{ab}} \times N_s^{\mathrm{ab}} \to [N_s, N_s]$
  may be computed via commutator identities and~$[N_s, N_s] = \zentrum(N_s)$ to yield
  \begin{equation*}
    \omega_s(x_1^{\alpha_1} \cdots x_s^{\alpha_s}, x_1^{\bar{\alpha}_1} \cdots
    x_s^{\bar{\alpha}_s}) = \prod_{i,j=1}^s [x_i, x_j]^{\alpha_i \bar{\alpha}_j} = \prod_{1 \le i <
      j \le s} x_{ij}^{2\alpha_{[i}, \bar{\alpha}_{j]}},
  \end{equation*}
  where the indexed brackets denote the antisymmetrizer (as customary in physics).

  Assume~$A, B \le N_s$ are abelian subgroups such that $\langle A \cup B \rangle = N_s$
  and~$A \cap B = \zentrum(N_s) = [N_s, N_s]$. For any~$x_i \: (1 \le i \le s)$ we have
  either~$x_i \in A$ or~$x_i \in B$, but not both
  since~$x_i \not\in [N_s, N_s] = \langle x_{ij} \rangle$. Hence we may, without loss of generality,
  reorder the generators so that~$x_1, \dots, x_t \in A \setminus B$
  and~$x_{t+1}, \dots, x_s \in B \setminus A$. Hence we must have
  \begin{align*}
    A &= \langle x_i, x_{ij} \mid 1 \le i \le t \land 1 \le i < j \le s\rangle,\\
    B &= \langle x_i, x_{ij} \mid t+1 \le i \le s \land 1 \le i < j \le s\rangle.
  \end{align*}
  In the degenerate case~$s=1$ with~$N_1 \cong \ZZ$, this just
  means~$A \cong \ZZ$ and~$B \cong 1$.  For~$s=2$ we have the
  subgroups~$A = \langle x_1, x_{12} \rangle \le N_2$
  and~$B = \langle x_2, x_{12} \rangle \le N_2$. Therefore we conclude
  that~$N_2$ is in fact a Heisenberg group. However, for $s>2$ the
  free nilpotent group~$N_s$ cannot be a Heisenberg group because at
  least one of~$A$, $B$ must contain two generators~$x_i$ and thus
  cannot be abelian. This settles also the following questions: Given
  a nilquadratic group~$H$ with torus~$T$, is there always an abelian
  bisection over $T$? Equivalently (confer
  Theorem~\ref{thm:sympl-corr}), given an abelian group~$P$ with a
  symplectic form~$\omega\colon P \times P \to T$, is there always a
  Lagrangian bisection? Obviously, the answer to both questions is no.
\end{myexample}

\bigskip\hrule\bigskip

For introducing Fourier transforms, it is advantageous to envisage a somewhat wider concept of
Heisenberg modules/algebras. The basic ideas go back to David Mumford's \emph{magnum opus} on theta
functions~\cite{Mumford2007b}.

As before we consider a central extension of the abelian group $P$ with kernel the torus $T$
\begin{equation}
  \label{eq:heis-seq}
  \xymatrix @M=0.5pc @R=1pc @C=2pc%
  { 1 \ar[r] & T \ar[r]^\iota & H \ar[r]^\pi & P \ar[r] & 0, }
\end{equation}
It is of crucial importance to observe that the \emph{Heisenberg
  twist}~$J$ is a morphism of~$\Heis$. Indeed, given any
duality~$\beta\colon H(\beta) \to H(\beta\trp)$, we have the
$\Nil_2$-morphism
\begin{equation*}
  \xymatrix @M=0.5pc @R=1pc @C=2pc%
  { 1 \ar[r] & T \ar[r] \ar@{=}[d] & H(\beta) \ar[r] \ar[d]^J
      & G \oplus \Gamma \ar[r] \ar[d]^{\,j} & 0\\
    1 \ar[r] & T \ar[r] & H(\beta\trp) \ar[r] & \Gamma \oplus G \ar[r] & 0,}
\end{equation*}
where~$j\colon G \oplus \Gamma \to \Gamma \oplus G$ is the map
introduced in Subsection~\ref{sub:heis-twist}. It should be noted that~$J$
does not descend to~$\Du$, meaning there is no
morphism~$\tilde{J}\colon \beta \to \beta\trp$ in~$\Du$
with~$J = H(\tilde{J})$.

Just as for~$H(\beta)$, there is also a generalized concept of
\emph{Heisenberg module} for~$H \in \Heis$. Recall that an
action~$\eta$ of~$H$ on a $K$-vector space~$S$ may be written as a
group homomorphism~$\eta\colon H \to \Aut_K(S)$. Since the torus
extends via~$\epsilon_T\colon T \to \nnz{K}$ into the multiplicative
group of the field and the latter acts naturally on~$S$ via its scalar
action~$\Delta\colon \nnz{K} \to \Aut_K(S)$, there is an induced
natural
action~$\tau = \Delta \circ \epsilon_T \colon T \to
\Aut_K(S)$. Since~$\Delta$ is a faithful action, we may
identify~$\Delta(\nnz{K})$ with a subgroup of~$\Aut_K(S)$. We obtain a
morphism of short exact sequences
\begin{equation}
  \label{eq:heis-action}
  \xymatrix @M=0.5pc @R=1pc @C=2pc%
  { 1 \ar[r] & T \ar@{^{(}->}[r]^\iota \ar[d]^{\epsilon_T} & H \ar[r]^\pi \ar[d]^\eta
      & P \ar[r] \ar[d]^{\zeta} & 0\\
    1 \ar[r] & \nnz{K} \ar@{^{(}->}[r]^(.4)\Delta & \Aut_K(S) \ar[r] & \Aut_K(S)/\nnz{K} \ar[r] & 0,}
\end{equation}
where the quotient action~$\zeta$ is induced from~$\eta$ via
$\zeta_z(s) := \eta_{\sigma(z)}(s) \, \nnz{K}$ for~$z \in P$
and~$s \in S$. Again, the choice of section~$\sigma\colon P \to H$
does not influence the definition of~$\zeta$. For a given
homomorphism~$\eta\colon H \to \Aut_K(S)$, the
diagram~\eqref{eq:heis-action} is thus commutative as soon as one
has~$\eta_c(s) = \epsilon_T(c) \, s$ for all~$c \in T$ and~$s \in
S$. In other words, the torus~$T$ must act naturally
through~$\epsilon_T$, which is exactly what we have required of a
Heisenberg module over~$H(\beta)$ in Definition~\fourcite{def:heis-alg}.

\begin{definition}
  Let~$T$ be a torus for~$K$ with torus
  map~$\epsilon_T\colon T \to \nnz{K}$, and fix a Heisenberg
  group~$H \in \Heis$ with central
  extension~\eqref{eq:heis-seq}. Then a \emph{Heisenberg module} is a
  $K$-vector space~$S$ with a \emph{Heisenberg
    action}~$\eta\colon H \to \Aut_K(S)$, meaning the
  diagram~\eqref{eq:heis-action} commutes.
\end{definition}

Our new definition of Heisenberg module thus includes Heisenberg
modules over~$H(\beta)$ in the previous sense. \emph{Heisenberg morphism}
over~$H \in \Heis$ are defined as before to be equivariant
maps~$\Phi\colon S \to S'$. In other words, we
require~$\Phi \colon \eta_h = \eta_h' \circ \Phi$ for all~$h \in H$,
where~$\eta\colon H \to \Aut_K(S)$ and~$\eta'\colon H \to \Aut_K(S')$
are the corresponding Heisenberg actions.

For generalizing the notion of Heisenberg algebra to
arbitrary~$H \in \Heis$, we need a bit more structure on the
Heisenberg group~$H$. Recall that in the case~$H = H(\beta)$, we had to distinguish between
the scalars~$TG$ and the operators~$\Gamma$. Since we do not require
the phase space to split into~$G$ and~$\Gamma$, we should not insist
on any particular decomposition into scalars and operators. We must therefore enrich the concept of
Heisenberg group.


We have $H(\beta) = TG \rtimes \Gamma$ and
$H(\beta\trp) = \Gamma \ltimes TG$, with the twist map given
by~$(cx,\xi) \mapsto (\xi, \tfrac{c}{\inner{x}{\xi}} \, x)$ in this
setting. Pulling phase factors out, this is equivalent
to~$J(x,\xi) = (\xi, x)/\inner{x}{\xi}$. The
product~$(x,c,\xi) (\tilde x, \tilde c, \tilde\xi)$ of both~$H(\beta)$
and~$H(\beta\trp)$ may be described on the
set~$G \times T \times \Gamma$ as componentwise in the outer factors
with extra factor~$\inner{\tilde x}{\xi}$ and~$\inner{x}{\tilde\xi}$,
respectively. We have the following diagram:
\begin{equation*}
  \xymatrix @M=0.5pc @R=1pc @C=2pc%
  { && TG \rtimes \Gamma \ar[dr] \ar[dd]^J&&\\
    0 \ar[r] & TG \ar[ur] \ar[dr] && \Gamma \ar[r] & 0\\ 
    && \Gamma \ltimes TG \ar[ur] && }
\end{equation*}
Furthermore, we have for the opposite
algebra~$K\!H(\beta)^o = KH(\beta\trp)$. Hence we may identify
right~$H(\beta\trp)$-modules with left~$H(\beta)$-modules.

The procedure is as follows: We fix a
duality~$\beta\colon G \times \Gamma$ and call~$TG$ the
\emph{Heisenberg scalars} and~$\Gamma$ the \emph{Heisenberg
  operators}. This refers both to~$H(\beta)$ and
to~$H(\beta\trp)$. Now we define a left/right Heisenberg algebra as a
left/right $TG$-algebra with~$\Gamma$ as its group of operators.

One checks that the twist~$J_\beta\colon H(\beta) \to H(\beta\trp)$ induces a ring
isomorphism~$J_\beta\colon H_K(\beta) \to H_K(\beta\trp)$ given
by~$\lambda \, (x,\xi) \mapsto \lambda \, \epsilon_T\inner{x}{\xi}^{-1} \, (-\xi, x)$ with
inverse~$\lambda \, (\xi, x) \mapsto \lambda \, \epsilon_T\inner{x}{\xi}^{-1} (x,-\xi)$.
Furthermore, we have an isomorphism $\tau\colon H_K(\beta\trp) \to H_K(\beta)^o$ given
by~$\lambda \, (\xi, x) \mapsto \lambda \, (x, \xi)$. We may thus identify $H_K(\beta\trp)$
with~$H_K(\beta)^o$, preferring to keep the position-before-momentum order
in~$H(\beta) = TG \rtimes \Gamma$ everywhere and composing the twist map with~$\tau$ to obtain the
new twist~$J_\beta\colon H_K(\beta) \to H_K(\beta)^o$, which is now an \emph{anti}-homomorphism and
in fact an involution. In other words, $H_K(\beta)$ is an \emph{involution algebra} (also known as
$*$-algebra) over the coefficient field~$K$, endowed with the trivial involution. It is easy to see
that its self-adjoint elements are given by~$KG \le H_K(\beta)$.

We think of the Fourier transform~$\Four\colon S \to \Sigma$ as a Heisenberg morphism between a left
$H(\beta)$-algebra~$S$ and a right $H(\beta)$-algebra~$\Sigma$. The
composition~$\Four_\Sigma \circ \Four_S = \Par$ should be understood as in the following diagram:

\begin{equation*}
  \xymatrix @M=0.5pc @R=1pc @C=2pc%
  { _\beta S \ar[r]^{\Four_S} & \text{\hphantom{$_\beta$}}\Sigma_\beta \ar[d]^{J^*}\\
    & _\beta\Sigma\text{\hphantom{$_\beta$}} \ar[r]^{\Four_\Sigma}
    & \text{\hphantom{$_\beta$}}S_\beta \ar[d]^{J^*}\\ 
    && _\beta S\text{\hphantom{$_\beta$}} }
\end{equation*}

Using moon-phase notation, the parity behavior is then deduced via
$\Four_\Sigma \Four_S(h \cdot s) = \Four_\Sigma(\Four_S s \cdot h) = \Four_\Sigma(\hat{h} \cdot
\Four_S s) = \Four_\Sigma \Four_S s \cdot \hat{h} = \bar{h} \cdot \Four_\Sigma \Four_S s$.
The point is that $H(\beta)$ is allowed to vary in the category of all Heisenberg algebras (left or
right), but it is fixed within a Fourier singlet/doublet.

References to be included are \cite{BonattoDikranjan2017}, \cite{LuefManin2009},
\cite{PrasadShapiroVemuri2010}, \cite{ParasadVemuri2008}, \cite{Vemuri2008}.

Two subspaces are called \emph{transverse} if they span the given vector
space, and, following~\cite[p.~21]{BatesWeinstein1997}, we define a
\emph{Lagrangian bisection} of a symplectic vector space~$V$ as a pair
of transverse Lagrangian subspaces~$L, L'$; then
automatically~$L \dotplus L' = V$. Each Lagrangian bisection
corresponds uniquely to a symplectic
isomorphism~$V \isomarrow L \oplus L^*$. This was called ``double
polarization'' earlier since this is more than a polarization but less
than a choice of basis. The notion of Lagrangian bisection is akin to
an orthogonal decomposition in a euclidean space. One should, however,
note the following two crucial differences: (1) While the orthogonal
complement of a subspace is always unique, there are in general many
choices for symplectic complement (see the reference for ``double
polarization'' given at another place). Hence one needs \emph{two}
choices. (2) The two pieces are isomorphic to each other---at least in
the vector space case: Choosing a basis in each, one may perform


From Wikipedia article ``Symplectic Vector Space'': Formally, the
symmetric algebra of V is the group algebra of the dual,
$\Sym(V) := K[V^*]$, and the Weyl algebra is the group algebra of the
(dual) Heisenberg group $W(V) = K[H(V^*)]$. Since passing to group
algebras is a contravariant functor, the central extension map
$H(V) \to V$ becomes an inclusion $\Sym(V) \to W(V)$.


\subsection{Generalized Heisenberg Groups}\label{sub:gen-heis-grp}
Given any locally compact abelian group~$G$, its \emph{characters} in the classical sense are the
continuous homomorphisms from~$G$ into the complex torus~$\Tor := S^1 \subseteq \CC$. The
collection~$\hat{G}$ of all characters is then again an abelian group known as the \emph{dual group}
of~$G$, and the famous Pontryagin duality theorem asserts that the natural
pairing~$\pont\colon G \times \hat{G} \to \Tor$ is a nondegenerate bicharacter; confer for example
Theorem~1.7.2 of~\cite{Rudin2017}.

In our algebraic setup, we start from a
\emph{duality}~$\beta\colon G \times \Gamma \to T$, which we define as
a nondegenerate bicharacter from the abelian groups~$G$ and~$\Gamma$
to an arbitrary torus~$T$ over a fixed ground field~$K$. By
definition, a \emph{torus over~$K$} is a (multiplicatively written)
abelian group~$T$ equipped with an
action~$\ast\colon T \to \Aut_K(K)$.  In other words, we have the
action laws $1 \ast \lambda = \lambda$ and
$(cd) \ast \lambda = c \ast (d \ast \lambda)$ as well as the linearity
laws~$c \ast (\lambda + \mu) = c \ast \lambda + c \ast \mu$ and
$c \ast (\lambda \mu) = (c \ast \lambda) \mu = \lambda (c \ast \mu)$,
for all~$c, d \in T$ and~$\lambda, \mu \in K$. Equivalently, the
action~$\ast$ may be described by the \emph{torus
  map}~$\epsilon_T\colon (T, \cdot) \to (\nnz{K}, \cdot)$
with~$\epsilon_T(c) := c \ast 1_K$ since we
have~$c \ast \lambda = \epsilon_T(c) \, \lambda$. We shall often
suppress the action symbol~$\ast$, provided this does not give rise to
confusion. Moreover, $\beta(x,\xi)$ will usually be written
as~$\inner{x}{\xi}_\beta$, with the index~$\beta$ omitted when the
context makes it clear.

In the classical case cited above, $G$ and~$\Gamma := \hat{G}$ as well as~$K = \CC$ and~$T = \Tor$
all carry a topology, which we discard to focus on the algebraic and algorithmic aspects. Given any
duality~$\beta\colon G \times \Gamma \to T$, we shall refer to the elements of~$G$ as
\emph{positions} and to those of~$\Gamma$ as \emph{momenta}; this is motivated by the fundamental
example of symplectic duality (Example~\ref{ex:heisgrp-sympl} below). Due to the nondegeneracy
of~$\beta$, we can identify positions~$x \in G$ with their \emph{position characters}
\mbox{$\inner{x}{-}\colon \Gamma \to T$} given by~$\xi \mapsto \inner{x}{\xi}$,
momenta~$\xi \in \Gamma$ with their \emph{momentum characters} $\inner{-}{\xi}\colon G \to T$,
$x \mapsto \inner{x}{\xi}$. Since~$G$ and~$\Gamma$ do not have any specific topology, one may view
them as discrete groups. Consequently, their dual groups~$\hat{G} = \Hom(G, T)$
and~$\hat{\Gamma} = \Hom(\Gamma, T)$ consist of \emph{all} homomorphisms, and nondegeneracy means
one has embeddings $G \hookrightarrow \hat{\Gamma}$ and $\Gamma \hookrightarrow \hat{G}$ for
encoding positions and momenta by their characters.

\begin{myexample}
  \label{ex:classical-torus-group}
  One of the most famous examples in the classical setting of~$K = \CC$ and~$T = \Tor$ is given by
  the duality~$\Tor \times \ZZ \to \Tor$ defined by~$\inner{x}{\xi} := x^\xi$. It is clear that this
  may be extended to~$\Tor^n \times \ZZ^n \to \Tor$
  with~$\inner{x}{\xi} := x_1^{\xi_1} \cdots x_n^{\xi_n}$. We refer to this example as the
  \emph{classical torus duality}~$\inner{\Tor^n}{\ZZ^n}$. The same term will also be used for its
  \emph{conjugate duality}~$\inner{\ZZ^n}{\Tor^n}$.
\end{myexample}

\begin{myexample}
  \label{ex:classical-vector-group}
  The most important special case of Example~\fourcite{ex:vecgrp} is given by the complex field~$K = \CC$
  with the classical torus~$T = \Tor$ and vector spaces~$V = V' = \RR^n$ having their canonical
  Euclidean structure. Since the latter is also equivalent to the natural
  paring~$\RR^n \times \RR_n \to \RR$, we shall use the dot notation for both.

Consier $K = \CC$, torus~$T = \Tor$ and vector spaces~$V=V'=\RR^n$. Via the standard
  character $\chi(c) = e^{i\tau c}$, one obtains the
  bicharacter~$\inner{x}{\xi} = e^{i \tau x\cdot\xi}$. We shall refer to this example as the
  \emph{standard vector duality}~$\inner{\RR^n}{\RR^n}$, which may also be presented
  as~$\inner{\RR^n}{\RR_n}$ under the natural pairing. Note that here the underlying scalar field
  $F = \RR$ of the vector spaces is distinct from the ground field~$K = \CC$ of the duality.
\end{myexample}

\begin{myexample}
  \label{ex:modular-vector-group}
  If~$V = F^n$ is a vector space over a Galois field~$F = \GF(q)$
  with~$q = p^m \; (m \in \NN)$ elements, we may again use the usual
  dot product as a bilinear form.  Using the multiplicative cyclic
  group~$\langle\zeta_p\rangle \subset \QQ(\zeta_p)$ as in the proof
  of Lemma~\ref{lem:bichar-nondeg}, a standard
  character~$\chi_a\colon \GF(q) \to \langle\zeta_p\rangle$ is given
  by~$c \mapsto \zeta_p^{\tr ac}$, with~$a \in \nnz{\ZZ_p}$ chosen
  arbitrarily and~$\tr\colon \GF(q) \to \ZZ_p$ the trace
  map~$\tr c := c + c^p + \cdots + \smash{c^{p^{m-1}}}$; for details
  regarding characters on Galois fields
  see~\cite{BensonRatcliff2008}. With the induced
  duality~$\inner{x}{\xi} = \chi_a(x \cdot \xi)$, we call this example
  the \emph{modular vector duality}~$\inner{\GF(q)^n}{\GF(q)^n}$.
\end{myexample}

As is well known, the torus duality gives rise to Fourier series
 and the vector duality to Fourier integrals.
There is another well-known breed of Fourier
transforms, known as the \emph{discrete Fourier transform};
it is based on the following important duality.

\begin{myexample}
  \label{ex:finite-group}
  Given any $N \in \NN$, let us introduce the following two concrete
  realizations of the \emph{cyclic group} of order~$N$. The first is
  the common representation~$\ZZ_N := \{ 0, 1, \dots, N-1 \}$ with
  addition modulo~$N$, the second
  is~$\Tor_N := N^{-1} \, \ZZ_N = \{ 0, \tfrac{1}{N}, \dots,
  \tfrac{N-1}{N} \}$ with addition modulo~$1$. The
  duality~$\Tor_N^n \times \ZZ_N^n \to \CC$ is now defined
  by~$\inner{x}{\xi} := e^{i \tau x \cdot \xi}$.  This can be
  obtained from the classical torus duality in two steps: First one
  restricts to the
  bicharacter~$\inner{}{}\colon\Tor_N^n \times \ZZ^n \to \CC$ via the
  embedding~$\Tor_N^n \hookrightarrow \Tor^n, \xi_i \mapsto e^{i\tau
    \xi_i}$. Then one applies the usual universal
  construction~\cite[\S I.9]{Lang2002} to make the bicharacter into a
  duality by taking the quotient on the right modulo the right
  kernel~$(N\ZZ)^n$; nothing is needed on the left since the left
  kernel is trivial. The resulting duality will called the
  \emph{cyclic duality}~$\inner{\Tor_N^n}{\,\ZZ_N^n}$.

  Of course, one may equally well form the \emph{conjugate duality}~$\inner{\ZZ_N^n}{\Tor_N^n}$. But
  unlike their infinite relatives, the dualities~$\inner{\Tor_N^n}{\,\ZZ_N^n}$
  and~$\inner{\Tor_N^n}{\ZZ_N^n}$ are not only similar but actually the \emph{same} (i.e.\@
  isomorphic): Since the cyclic groups~$\Tor_N$ and~$\ZZ_N$ are the same abstract group~$\ZZ/N$,
  both are one and the same duality~$\inner{(\ZZ/N)^n}{\,(\ZZ/N)^n}$,
  given~\cite[Thm.~4.5d]{Folland1994} by
  \begin{equation}
    \label{eq:finite-duality}
    \inner{k + N \ZZ^n}{\,l + N \ZZ^n} = e^{i\tau (k \cdot l)/N}
  \end{equation}
  for~$k, l \in \ZZ^n$. Nevertheless, it can be worthwhile to
  distinguish the two realizations of this duality.

  The different nature of~$\Tor_N^n$ and~$\ZZ_N^n$ can also be seen
  in the context of \emph{normalizing Haar measure}~$\mu$. While such
  a choice is \emph{per se} immaterial, it must be consistent between
  the primal and dual group for the inversion theorem to
  hold~\cite[\S1.5.1]{Rudin2017}. For compact groups~$G$, the
  canonical choice is to set~$\mu(G) = 1$, while for discrete
  groups~$G$ one sets~$\mu(\{x\}) = 1$ for all points~$x \in G$. But
  since~$(\ZZ/N)^n$ happens to be both discrete and compact, one must
  decide whether to impose the discrete or the compact normalization
  on the primal group~$(\ZZ/N)^n$, so that the other normalization is
  then conferred onto its dual, which is again~$(\ZZ/N)^n$. From the
  above construction, it is clear that the natural choice is to
  endow~$\Tor_N^n$ with the compact and~$\ZZ_N^n$ with the discrete
  normalization. In other words, we have~$\mu(\Tor_N^n) = 1$
  and~$\mu(\ZZ_N^n) = N^n$, thus also~$\mu(\{x\}) = 1/N^n$
  for~$x \in \Tor_N^n$ but~$\mu(\{x\}) = 1$ for~$x \in \ZZ_N^n$.
\end{myexample}

\medskip
\hrule
\medskip

As is well-known, the classical definition~\cite[Ex.~14.6]{Hall2013}, \cite{Howe1980},
\cite[\S13]{Woit2016}, \cite[\S1.1]{Taylor1986} of the Heisenberg group is as a \emph{matrix Lie
  group}
\begin{equation}
  \label{eq:matrix-heisgrp}
  H_n(K) = \left\{ \begin{pmatrix}1 & \xi & c\\ 0 & I_n & x\\ 0 & 0 & 1\end{pmatrix} \:\middle|\:
    \xi \in K_n, x \in K^n, c \in K \right\},
\end{equation}
where the underlying field~$K$ is usually either~$\RR$ or~$\CC$. It is also customary to
cast~\eqref{eq:matrix-heisgrp} in terms of symplectic vector spaces (see
Example~\ref{ex:heisgrp-sympl}). What is not so well-known is that this setup can be generalized to
arbitrary LCA groups under Pontryagin duality~$\pont$. This has been exhibited forcefully by
\emph{Andr{\'e} Weil} in his lucid article~\cite{Weil1964}.

One may interpret~(\fourcite{eq:heis-grouplaw}) as a \emph{schematic matrix group}
like~\eqref{eq:matrix-heisgrp}, but now with~$x \in G$, $\xi \in \Gamma$ and~$c \in T$. Here one
should keep in mind that the addition of the upper right matrix elements corresponds to the
(multiplicatively written!) group operation in~$T$.  Furthermore, one agrees that a left matrix
element~$\xi$ multiplies with a right matrix element~$x$ to yield~$\inner{x}{\xi}$; all other
products are trivial since they involve~$0$ or~$1$.

For the standard vector duality~$\inner{\RR^n}{\RR^n}$ of
Example~\ref{ex:classical-vector-group}, the link to the matrix
group~\eqref{eq:matrix-heisgrp} can be made more precise by reframing
the Heisenberg group in terms of a \emph{symplectic vector
  space}~$V$. This is the formulation commonly used in more advanced
treatments of the Heisenberg group~\cite[\S5.1]{BinzPods2008},
\cite{Folland2016}, \cite{LibermannMarle2012}.

\begin{myexample}
  \label{ex:heisgrp-sympl}
  Let us take~$G = V$ and~$\Gamma = V^*$ for a vector space and its dual over a field~$F$, writing
  the natural pairing as~$\funcinner{}{}\colon V \times V^* \to F$.  Then~$Z := V \oplus V^*$ is a
  symplectic vector space under the canonical symplectic form~$\omega\colon Z \times Z \to F$
  with~$\omega(x, \xi; \tilde x, \tilde\xi) = \funcinner{\tilde x}{\xi} - \funcinner{x}{\tilde\xi}$.
  Using Lemma~\ref{lem:bichar-nondeg}, we define the
  duality~$\inner{x}{\xi} := \chi \funcinner{x}{\xi}$. We shall refer to this as the
  \emph{symplectic duality}~$\inner{V}{V^*}$. While such dualities are associated to the Hamiltonian
  phase space~$T^*V = V \oplus V^* = Z$, the abstract vector dualities~$\inner{V}{V}$ of
  Example~\fourcite{ex:vecgrp} are linked to the Lagrangian phase space~$TV = V \oplus V$. In both cases,
  there is an evident physical interpretation~\cite[\S1.1]{MarsdenRatiu1994}, where elements of~$V$
  denote positions while its tangent/cotangent vectors are the corresponding
  \emph{velocities/momenta} (this generalizes to the nonlinear case where the configuration space is
  a manifold~$M$ rather than a vector space~$V$, with Lagrangian phase space~$TM$ and Hamiltonian
  phase space~$T^*M$ enjoying the same interpretation). The Hamiltonian case is the important since
  it naturally leads to quantization, replacing the commutative algebra of classical
  observables~$C^\infty(T^*V)$ by the noncommutative algebra~$\mathcal{H}\big(L^2(V)\big)$ of
  self-adjoint (alias ``Hermitian'') operators on the Hilbert space~$L^2(V)$. The observables
  position and momentum are then quantized~\cite[\S3.5]{Hall2013} into the position
  operator~$f(x) \mapsto x \, f(x)$ and its canonically conjugate momentum
  operator~$f(x) \mapsto \der f/\der x$). This is intimately linked to the unitary irrep of the
  Heisenberg group (see below at the end of the example).

  There are two flavors of Heisenberg group on the symplectic vector space~$Z$, depending on whether
  or not one regards the (double) \emph{polarization} $Z = V \oplus V^*$ as part of the given
  data. In general, a polarization of a symplectic vector space~$(Z, \omega)$ is a choice of
  Lagrangian subspace~$V \le Z$. Such subspaces~$V$ exist in abundance---they comprise the so-called
  Lagrange-Grassmann manifold---and even for fixed~$V$, there are
  plenty~\cite[Prop.~A6.1.6]{LibermannMarle2012} of Lagrangian complements~$V'$, each of which may
  be identified with the dual space~$V^*$. For any Lagrangian decomposition $Z = V \oplus V'$, the
  bases of~$V$ are in bijective correspondence with the symplectic bases of~$Z$.

  \begin{enumerate}[(a)]
  \item So if one does have the double polarization~$Z = V \oplus V^*$, one may define the
    \emph{polarized symplectic Heisenberg group} $H(V, V^*) := K \times Z$ with multiplication given
    by
    \begin{equation*}
      (u, x, \xi) \, (u', x', \xi') := 
      \big( u+u'+\funcinner{x'}{\xi}, x+x', \xi+\xi' \big).
    \end{equation*}
    In the standard case~$V = \RR^n$, this is the definition of~\cite[p.~19]{Folland2016},
    where~$H(V, V^*)$ is written~$\mathbf{H}_n^{\mathrm{pol}}$; this yields the matrix Lie
    group~\eqref{eq:matrix-heisgrp}. In general, $\pi := \chi \times 1_Z$ is an
    epimorphism~$H(V, V^*) \twoheadrightarrow H(\beta)$, $(u, x, \xi) \mapsto \chi(u) \, (x, \xi)$
    with~$\ker(\pi) = \ker(\chi) \times Z$. Hence the Heisenberg group of
    Definition~\fourcite{def:assoc-heisgrp} is essentially~$H(V, V^*)$, modulo the prime ring. Some
    sources~\cite[Exc.~5.1-4]{AbrahamMarsdenRatiu1983} define the Heisenberg group as~$H(\beta)$,
    but restricted to the symplectic duality~$\beta=\chi \circ (-|-)$.
  \item Using only the symplectic structure~$(Z, \omega)$, one is led
    to the \emph{apolar symplectic Heisenberg group}~$H(Z, \omega)$;
    confer for example~\cite[(8)]{Tilgner1970}. As a set, we have
    again~$H(Z, \omega) := K \times Z$ but with the group law
    \begin{equation*}
      (t, z) \, (t', z') := (t+t'+\tfrac{1}{2} \, \omega(z, z'), z+z').
    \end{equation*}
    In the standard case~$V = \RR^n$, this coincides with~\cite[p.~19]{Folland2016}, where the
    notation~$\mathbf{H}_n$ is employed for~$H(Z, \omega)$. Here the conventional factor of~$1/2$ is
    motivated by exponentiating the canonical commutators of Hamiltonian
    mechanics~\cite[(1.15)]{Folland2016}. Given an arbitrary Lagrangian decomposition
    $Z = V \oplus V^*$, one obtains the isomorphism $\Psi\colon H(Z, \omega) \isomarrow H(V, V^*)$
    defined by~$\Psi(t, x, \xi) := (t + \tfrac{1}{2} \, \funcinner{x}{\xi}, x, \xi)$.
  \end{enumerate}
  Thus the polarized and the symplectic Heisenberg group are the same---provided we have picked a
  polarization. For our definition of the abstract Heisenberg group~$H(\beta)$, this is the case and
  will be of importance for our further development (confer Definition~\fourcite{def:heis-alg}). In the
  special case of the symplectic duality~$\inner{V}{V^*}$, one may thus refer to~$H(\beta)$ as the
  ``symplectic Heisenberg group''.

  The symplectic Heisenberg group has two important representations: There is (up to isomorphism)
  just one 
\end{myexample}

\begin{myremark}
  Since the polarized symplectic Heisenberg group~$H(V, V^*)$ may be
  viewed as the additive counterpart of~$H(\beta)$, one might ask for
  an apolar variant~$H_\Omega(\beta)$ of the latter having as its
  additive counterpart the apolar symplectic Heisenberg
  group~$H(Z, \omega)$. Since we do not need this~$H_\Omega(\beta)$
  for our treatment of Heisenberg modules,
we will only mention it briefly here.

  We define~$H_\Omega(\beta) := T \times Z$ as a set and endow it with the group
  law~$(c, z) \: (\tilde{c}, \tilde{z}) := \big( c\tilde{c} \: \Omega(z, \tilde{x})^{1/2},
  z+\tilde{z} \big)$.
  Here~$\Omega^{1/2}\colon Z \times Z \to T$ is the multiplicative symplectic form
  with~$\Omega(x, \xi; \tilde{x}, \tilde\xi)^{1/2} := \inner{\tilde{x}}{\xi}^{1/2} /
  \inner{x}{\tilde\xi}^{1/2}$,
  where~$\inner{}{}^{1/2} := \chi(\tfrac{\funcinner{}{}}{2})\colon Z \to T$ is the ``square root
  duality''. As a consequence, we have~$\chi(\omega/2) = \Omega^{1/2}$ as the square root of the
  given symplectic form. Note that in the classical setting~$K = \CC$, $T = S^1 \subset \CC$
  with~$\chi(x) = e^{i\tau x}$, the determination of the square root corresponds to choosing the
  principal branch.

  We have again an
  epimorphism~$\pi_\Omega = \chi \times 1_Z\colon H(Z, \omega) \twoheadrightarrow H_\Omega(\beta)$,
  and the isomorphism~$\Psi\colon H(Z, \omega) \isomarrow H(V, V^*)$ introduced in
  Example~\ref{ex:heisgrp-sympl} lies above the ``multiplicative''
  isomorphism~$\Psi_\Omega\colon H_\Omega(\beta) \isomarrow H(\beta)$ given
  by~$\Psi_\Omega(c, x, \xi) := (\inner{x}{\xi}^{1/2} \, c, x, \xi)$, in the sense
  that~$\pi \Psi = \Psi_\Omega \pi_\Omega$. Despite these analogies, there is a crucial difference
  between the Heisenberg group~$H(\beta)$ of Definition~\fourcite{def:assoc-heisgrp} and its
  unpolarized variant~$H_\Omega(\beta)$ in that the former is a semidirect product while the latter
  is not.
\end{myremark}


The Heisenberg group over the standard vector duality~$\inner{\RR^n}{\RR_n}$ could be called the
\emph{kinematic Heisenberg group} since it underpins physical kinematics, both classical (Hamilton's
equations) and quantum (Heisenberg equation). Indeed, the Schr\"odinger
representation~$\rho_h\colon H_n \to \mathcal{H}\big( L^2(\RR^n) \big)$ yields the latter; it is
parametrized by the Planck constant $h$ whose so-called semi-classical limit $h \to 0$ leads to
Hamiltonian mechanics. See also~\cite{Howe1980}, \cite{Howe1980a} for more on the representation
theory of~$H_n$.

If one dislikes the idea of constants tending to zero (though one might interpret this as taking
place in a hypothetical sequence of universes with progressively less significant quantum effects),
the framework of Plain Mechanics (p-mechanics) offers an alternative viewpoint~\cite{Kisil1996}:
Before specializing to any quantum or classical (or hyperbolic quantum) version, physical systems
are described in the ``plain'' setting of the Heisenberg group~$H_n$, using~$\big(L^1(H_n), \star)$
as the algebra of observables (including in particular the Hamiltonian). A so-called universal
equation rules the evolution of the system, which transforms to the Heisenberg equation under the
representation~$\rho_h$ and to Hamilton's equations under~$\rho_0$. In this framework, one may
develop corresponding brackets~\cite{Kisil2002}, notions of state~\cite{Kisil2003} and a detailed
mechanical theory~\cite{Kisil2004}. In a newer presentation~\cite[IV.1]{Kisil2018}, p-mechanics is
treated in a more geometric context where the elliptic / parabolic / hyperbolic cases correspond to
different number rings (complex / dual / double numbers), which in turn lead to different physical
frameworks (quantum / classic / hyperbolic quantum).

It would be nice to reformulate the last results in term of suitable tori. For example, the
classical case should correspond to the torus~$\Tor_\epsilon := \{ 1 + b \epsilon \mid b \in \RR \}$
of the dual numbers~$\RR_\epsilon := \RR[\epsilon \mid \epsilon^2 = 0]$ with the
duality~$\inner{\RR^n}{\RR_n}_\epsilon$ given
by~$\inner{x}{\xi}_\epsilon := 1 + \epsilon (x \cdot \xi)$. Its $L^2$ representation (encoded in the
$L^2$ Fourier doublet) should then yield the classical Hamilton's equations just as the standard
duality yields the Heisenberg equation. In an even more ambitious enterprise, one might try to
develop a general theory of ``universal equations'' for a given duality that yields classical and
quantum kinematics as special cases. For this purpose, some tools developed in~\cite{Akbarov1995a},
\cite{Akbarov1995b} may be of help.

\begin{myexample}
  For another instance of the construction~$R_* R_+ \rtimes R_+$, take the ring~$R = \ZZ_2$. Hence
  we consider the duality~$\beta\colon \ZZ_2 \times \ZZ_2 \to \ZZ_2$ defined by~$\beta(m,n) = mn$.
  Conforming to our conventions, we write the torus multiplicatively via
  $\ZZ_2 \isomarrow \ZZ^\times, [c] \mapsto (-1)^c$, meaning~$[0] \leftrightarrow +1$,
  $[1] \leftrightarrow -1$. As usual, we often express the values~$\pm 1$ just by the sign. The
  duality is then given by~$\beta(m,n) = (-1)^{mn}$. We will show that~$H(\beta)$ is the dihedral
  group~$D_4$, the symmetry group of the square (which we assume centered in the origin with
  axis-paraellel sides). If~$t$ denotes the counter-clockwise $90^\circ$ turn and~$r$ the reflection
  in the vertical axis, we obtain the presentation
  \begin{equation*}
    D_4 = \langle t, r \mid t^4 = r^2 = 1, rt = t^3 r \rangle
    = \{ 1, t, t^2, t^3, r, tr, t^2r, t^3r \},
  \end{equation*}
  Where~$tr, t^2r, t^3r$ may be respectively interpreted as reflections in the
  anti-diagonal~$x+y=0$, horizontal~$y=0$, and diagonal~$x-y=0$. We choose~$Z(D_4) = \{ 1, t^2 \}$
  as our torus~$T = \ZZ_2$, which enforces the identification~$1 \leftrightarrow +1$,
  $t^2 \leftrightarrow -1$. For the position group~$G = \ZZ_2$ we take~$\{ 1, r \}$, leading to the
  identification~$1 \leftrightarrow [0]$, $r \leftrightarrow [1]$; for the momentum
  group~$\Gamma = \ZZ_2$ we use~$\{ 1, tr \}$ with identification~$1 \leftrightarrow [0]$,
  $tr \leftrightarrow [1]$. In these terms, the
  duality~$\beta\colon \ZZ_2 \times \ZZ_2 \to \ZZ_2, (m,n) \mapsto mn$ sends~$(r, tr)$ to~$-1$ and
  all other pairs to~$+1$.

  \medskip
  \noindent\begin{tabular}[h]{||l|l||l|l||}
    \hline
    $D_4$ & $H(\beta)$ & $D_4$ & $H(\beta)$\\\hline
    $1$ & $+00$ & $r$ & $+10$\\
    $t$ & $-11$ & $tr$ & $+01$\\
    $t^2$ & $-00$ & $t^2r$ & $-10$\\
    $t^3$ & $+11$ & $t^3r$ & $-01$\\\hline
  \end{tabular}
  \medskip

  \noindent For defining a group isomorphism~$D_4 \isomarrow H(\beta)$, we construct first the
  unique homomorphism on the free group with~$1 \mapsto +00$, $t \mapsto -11$, $r \mapsto +10$. As
  one sees immediately, this homomorphism annihilates the relators~$t^4, r^2, trtr$ and thus yields
  a homomorphism~$\iota\colon D_4 \to H(\beta)$. Computing all other elements in terms of the
  generators~$t$ and~$r$, one will verify that~$\iota$ is given as in the table above. Since this is
  obviously a bijection, it provides us with the required isomorphism~$D_4 \isomarrow H(\beta)$.
\end{myexample}


\subsection{The Heisenberg Twist.}\label{sub:heis-twist} In
Example~\ref{ex:heisgrp-sympl} we have seen how the symplectic
dualities~$\inner{V}{V^*}$ arise naturally on any given vector space
(relative to a standard character). As usual, the \emph{symplectic
  structure}~$\omega\colon Z \times Z \to F$ induces a
map~$j\colon Z \to Z^*$, given by~$j(x, \xi) = (-\xi, x)$ in the
case~$Z = V \oplus V^*$. Obviously, $j^2\colon Z \to Z$ is the
involution~$(x,\xi) \mapsto -(x,\xi)$, so we have in
particular~$j^4 = 1_Z$.

\begin{myremark}
  In case~$F = \RR$, the map $j$ is akin to a complex structure, except that the latter maps a
  vector space to itself rather than to its dual. Indeed, the standard symplectic form on the
  Hamiltonian phase space~$Z = T^* V = V \oplus V^*$ is \emph{formally} identical to the standard
  complex structure on the Lagrangian phase space~$L = TV = V \oplus V$, namely the analogous
  map~$j\colon L \to L$ with $j(x,y) = (-y,x)$.
\end{myremark}

We will now forge the map~$j$ into an isomorphism between Heisenberg
groups. So let~$\beta\colon G \times \Gamma \to T$ be a
duality. Generalizing the case of symplectic vector spaces, we define
first
\begin{equation}
  \label{eq:symplectic-map}
  j\colon\quad G \times \Gamma \to \Gamma \times G,\quad
  (x,\xi) \mapsto (-\xi, x).
\end{equation}
Writing~$\beta\trp\colon G \times \Gamma \to T$ for the
\emph{transposed duality}~$\beta\trp(\xi, x) := \beta(x, \xi)$, we
want to obtain a group
isomorphism~$J\colon H(\beta) \to H(\beta\trp)$. Since phase factors
are only added for capturing the modulation twists, we expect~$J$ to
act trivially on~$T$; thus we are only concerned with the action
of~$J$ on~$G \times \Gamma \cong 1 \times (G \times \Gamma)$. Assuming
further that~$J$ coincides with~$j$ on~$G \cong G \times 0$
and~$\Gamma \cong 0 \times \Gamma$, we obtain
\begin{align*}
  J(x,\xi) &= J\big( (x,0) (0,\xi) \big) = J(x,0) \, J(0,\xi) = (0,x) (-\xi, 0)\\
           &= \inner{x}{-\xi} \, (-\xi, x) = j(x,\xi) / \inner{x}{\xi}
\end{align*}
as our definition of the \emph{Heisenberg twist}~$J\colon H(\beta) \to H(\beta\trp)$. More
explicitly, this yields now the map~$J\colon T \times G \times \Gamma \to T \times \Gamma \times G$
defined by $c \, (x, \xi) \mapsto \sfrac{c}{\inner{x}{\xi}} \, (-\xi, x)$. This shows that the
passage from the abelian groups~$G \times \Gamma$ and~$\Gamma \times G$ to the nonabelian
groups~$H(\beta)$ and~$H(\beta\trp)$ just imposes the additional phase
factor~$\inner{x}{\xi}^{-1}$. It is easy to check that~$J\colon H(\beta) \to H(\beta\trp)$ is indeed
an isomorphism of groups and that~$J^2\colon H(\beta) \to H(\beta)$ is again an involution,
namely~$J^2 = 1 \times j^2$.

\begin{remark}
  The inverted twist map~$J^{-1}\colon H(\beta\trp) \to H(\beta)$ is
  given
  by~$c (\xi, x) \mapsto \sfrac{c}{\inner{x}{\xi}} \, (x, -\xi)$; it
  should not be confused with the twist
  map~$H(\beta\trp) \to H(\beta)$ intrinsic to the transposed
  Heisenberg group, acting
  via~$c(\xi, x) \mapsto \sfrac{c}{\inner{x}{\xi}} \, (-x,
  \xi)$. Indexing the twist map to its duality, this shows
  that~$J_\beta^{-1} \neq J_{\beta\trp}$ while we have of
  course~$J_\beta^2 = J_{\beta\trp}^2 =: J^2$ and in
  fact~$J_\beta^{-1} = J_{\beta\trp} \circ J^2$.  In the sequel, it
  will always be clear from the context which duality is being
  referred to; hence we shall suppress the index on~$J$. If there is
  any danger of confusion, the duality in question will be specified
  \emph{expressis verbis}.
\end{remark}


\end{document}